\documentclass[10pt, final]{amsart} 

\usepackage{amsmath,amssymb,amscd,latexsym}
\usepackage{epsfig, graphics, color}
\usepackage{cite}
\usepackage[notref, notcite]{showkeys}

\renewcommand{\geq}{\geqslant}
\renewcommand{\leq}{\leqslant}

\renewcommand{\lll}{\langle}
\newcommand{\rrr}{\rangle}

\renewcommand{\a}{\alpha}
\renewcommand{\b}{\beta}

\renewcommand{\k}{\kappa}
\newcommand{\pp}{\varphi}

\newcommand{\p}{\partial}

\newcommand{\R}{\mathbb{R}}
\renewcommand{\Re}{\text{Re}\,}

\renewcommand{\l}{\lambda}

\newcommand{\G}{\Gamma}
\renewcommand{\L}{\Lambda}

\renewcommand{\H}{\mathcal{H}}
\newcommand{\B}{\mathcal{W}}
\newcommand{\Rr}{\mathcal{B}}
\newcommand{\J}{\mathcal{J}}
\newcommand{\K}{\mathcal{K}}

\newcommand{\W}{\mathbf{W}}
\newcommand{\U}{U^\bot}
\newcommand{\V}{U^\parallel}

\newcommand{\conv}{\partial_t+q\partial_\alpha}

\newcommand{\We}{{\scriptstyle We}}

\newtheorem{theorem}{Theorem}[section]
\newtheorem{lemma}[theorem]{Lemma}
\newtheorem{proposition}[theorem]{Proposition}
\newtheorem{remark}[theorem]{Remark}

\newtheorem{corollary}[theorem]{Corollary}
\newtheorem*{main-theorem}{Main Theorem}
\newtheorem*{remark*}{Remark}

\numberwithin{equation}{section}

\title[Regularity of water waves]
{Gain of regularity for \\ water waves with surface tension}

\author[Vera Hur]{Vera~Mikyoung~Hur}
	\address{Department of Mathematics, University of Illinois at Urbana-Champaign,
	Urbana, IL 61801}
	\email{verahur@\allowbreak math.\allowbreak uiuc.\allowbreak edu}	

\begin{document}

\maketitle

\begin{abstract}
Regularizing effects of surface tension are studied for interfacial waves between a two-dimensional, infinitely-deep and irrotational flow of water and vacuum. The water wave problem under the influence of surface tension is formulated as a system of second-order in time nonlinear dispersive equations. The main result states that if the curvature of the initial fluid surface is in the Sobolev space of order $k+7/2$ and if its derivatives decay faster than a polynomial of degree $k+1$ does then the curvature of the fluid surface corresponding to the solution of the problem instantaneously gains $k/2$ derivatives of smoothness compared to the initial state. The proof uses energy estimates under invariant vector fields of the associated linear operator. \end{abstract}

\section{Introduction}\label{S:intro}
 
Recently in \cite{CHS-lsm}, \cite{CHS} and \cite{ABZ-water} regularizing effects due to surface tension were put into perspective for wave motions at the interface separating in two spatial dimensions an infinitely-deep irrotational flow of an incompressible inviscid fluid from vacuum. In particular the solution of the water wave problem under the influence of surface tension, localized in space as well as in time, was shown to gain a $1/4$ derivative of smoothness compared to the initial data. Here we take matters further and demonstrate gain of high regularities. Specifically, if the curvature of the initial fluid surface is contained in the Sobolev space of order $k+7/2$ and if its derivatives decay faster than $\a^{-k-1}$ does as $|\a|\to \infty$, where $\a\in\R$ parametrizes the curve, if the initial velocity of a fluid particle at the surface is likewise, then the curvature of the fluid surface corresponding to the solution of the problem at subsequent times will acquire $k/2$ derivatives of smoothness compared to the initial state. 

\medskip

Enlightening a genuine improvement in the propagation of surface water waves, the result is far more physically significant than the gain of a fractional derivative which the local smoothing effect in \cite{CHS} or \cite{ABZ-water} offers. Furthermore the instantaneous smoothing effect -- the solution at any time after evolution becomes smoother than the initial data -- imparts an acute understanding of wave motions at the surface of water. In the local smoothing effect, in stark contrast, integration in time regularizes the solution.  

\subsection*{Formulation}The present account is based upon the formulation of the water wave problem under the influence of surface tension as the system 
\begin{align}
(\conv)^2\k-\frac{1}{2\We}\H \partial_\a^3\k+g\H\p_\a\k-2\k\p_\a^2\k&=R^\k(\k,u), \label{E:kappa1} \\
(\conv)^2u-\frac{1}{2\We}\H \partial_\a^3u+g\H\p_\a u-2\k\p_\a^2u&=R^u(\k,u). \label{E:u1}
\end{align} 
Here $t \in \R$ denotes the temporal variable and $\a\in \R$ is the arclength parametrization of the fluid surface, which serves as the spatial variable; $\k=\k(\a,t)$ measures the curvature of the fluid surface, while $q=q(\a,t)$ pertains to the velocity of a fluid particle at the surface and $u=q_\a$. Throughout we represent partial differentiation either using the symbol $\p$ or by a subscript. The Hilbert transform is given as
\begin{equation*}\label{D:H}
\H f(\a)=\frac{1}{\pi} PV\int^\infty_{-\infty} \frac{f(\b)}{\a-\b}~d\b \qquad (\a\in\R),\end{equation*}
where $PV$ stands for the Cauchy principal value; alternatively $\widehat{\H f}(\xi)=-i\text{sgn}(\xi)\hat{f}(\xi)$ defines the operator in the Fou\-ri\-er space. Moreover $\We$ is inversely proportional to surface tension and $g$ describes gravitational acceleration. Lastly $R^\k$ and $R^u$ are made up of ``smooth" remainders in the Sobolev space setting. Details are discussed in Section \ref{SS:formulation} and Section \ref{SS:reformulation}, Section \ref{SS:R2}. 

\medskip

Observe that \eqref{E:kappa1} and \eqref{E:u1} are dispersive, characterized by the operator $\p_t^2-\frac{1}{2\We}\H\p_\a^3+g\H\p_\a$. The dispersion relation  (see \cite{CHS}, for instance)
\begin{equation}\label{E:dispersion} 
c(\xi)=\Big(\frac{|\xi|}{2\We}+\frac{g}{|\xi|}\Big)^{1/2}\frac{\xi}{|\xi|}\end{equation}
in particular suggests a regularizing effect under the influence of surface tension, i.e. $\We<\infty$, where $c(\xi)$ denotes the phase velocity of simple harmonic oscillations with the frequency $\xi$. Colloquially speaking, high frequency waves propagate faster than low frequency waves, broadening out the wave profile. 

But \eqref{E:kappa1} and \eqref{E:u1} are severely nonlinear, though. Notably their left sides contain $q\k_{\a t}$, $q^2\k_{\a\a}$ and $qu_{\a t}, q^2u_{\a\a}$ whereas the local smoothing effect (see \cite{CHS} and references therein) for the equation 
\[\p_t^2\k-\frac{1}{2\We}\H\p_\a^3\k+g\H\p_\a\k=R(\a,t),\] 
where $\We<\infty$, controls merely up to $7/4$ spatial derivatives in the inhomogeneity in the $L^2$-space setting. The energy method, to compare, handles maximally $3/2$ derivatives.

\subsection*{Local smoothing revisited}The proof in \cite{CHS} or \cite{ABZ-water} of the local smoothing effect for the water wave problem with surface tension is rooted in Kato's method of positive commutators (see \cite{Kato}, for instance). Specifically, if $\k$ solves \eqref{E:kappa1}\footnote{Equation \eqref{E:kappa1} solo embodies the problem; see Section \ref{SS:reformulation}.} then (see \cite{CHS}, for instance)
\begin{multline}\label{E:LScomm}
\|\lll\a \rrr^{-\rho}\p_\a^{K+1}\k\|_{L^2_\a}^2+\frac{d}{dt}\int^\infty_{-\infty} \lll\a\rrr^{-2\rho+1}
\big(\p_\a^{K}\k\,\H\p_\a^{K}\k_t -\p_\a^{K}\k_t\,\H \p_\a^{K}\k\big)~d\a \\ 
\leq C\big(\|\k\|_{H^{K+3/4}_\a}, \|\k_t\|_{H^{K-3/4}_\a}\big)\end{multline} 
for $K\geq 1$ a sufficiently large\footnote{Differentiated many times, \eqref{E:kappa1} becomes nearly linear whereupon the computation is made.} integer and for $\rho>1/2$ a real, where $\lll\a\rrr=(1+\a^2)^{1/2}$ is to describe weighted Sobolev spaces. Here and in the sequel $C(f_1, f_2, ...)$ means a positive but polynomial expression in its arguments. Integrated over an interval of time, \eqref{E:LScomm} implies that 
\begin{equation}\label{E:local-smoothing} 
\|\lll\a \rrr^{-1/2-}\p_\a^{K+1}\k\|_{L^2_{t, \text{loc}}L^2_\a}\leq 
C\big(\|\k\|_{H^{K+3/4}}(0), \|\k_t\|_{H^{K-3/4}}(0)\big).\end{equation} 
To interpret, the solution gains a $1/4$ derivative of smoothness relative to the initial data at the expense of a $1/2$ power of $\a$. Indeed 
\begin{equation}\label{E:comm+} 
\int \big(\lll\a\rrr^{-\rho+1/2} \p_\a^{K}\k\,\lll\a\rrr^{-\rho+1/2}\H \p_\a^{K}\k_t
- \lll\a\rrr^{-\rho+1/2} \p_\a^{K}\k_t\,\lll\a\rrr^{-\rho+1/2}\H \p_\a^{K}\k\big)~d\a
\hspace*{-.2in}\end{equation}
at each time is bounded by Sobolev norms of the initial data via the energy method. Does the increase by $k/2$ powers in spatial weight then give rise to $k/4$ many more derivatives of smoothness?

\medskip

Perhaps one attempts to use induction in the preceding argument to attain high regularities. It necessitates that \eqref{E:comm+} be nonnegative\footnote{It amounts to the method of positive commutators.}, which incidentally is more singular than a time integral of the right side of \eqref{E:LScomm}. But the Hilbert transform does not commute with functions as a rule of thumb. Hence even the principal part of \eqref{E:comm+} is unlikely to carry a sign. 

For a broad class of Korteweg-de Vries type\footnote{The Airy operator $\p_t+\p_\a^3$ dominates the linear part.} partial {\em differential} equations, allowedly fully nonlinear, on the other hand, an appropriate weight, possibly nonlinear, may control singular $L^\infty$-in time contributions in the course of the method of positive commutators. As a matter of fact the solution was shown in \cite{CKS}, for instance, to become infinitely smooth if the initial datum asymptotically vanishes faster than polynomially and if in addition it possesses a minimal regularity, namely the infinite gain of regularity.

\subsection*{Method of invariant vector fields}Instead we exploit invariant vector fields of the associated linear operator to understand regularizing effects for the problem. To illustrate, let's put forward the linear equation 
\begin{equation}\label{E:linear}
\partial_t^2\k-\H\partial_\a^3\k=0\end{equation} 
of \eqref{E:kappa1} (and \eqref{E:u1}) after normalization of parameters. It is readily verified that
\begin{equation}\label{E:L2}
\|\Lambda^{1/2}\p_\a\k\|_{L^2}^2(t)+\|\partial_t\k\|_{L^2}^2(t)=
\|\Lambda^{1/2}\p_\a\k\|_{L^2}^2(0)+\|\partial_t\k\|_{L^2}^2(0),
\end{equation}
where $\Lambda=(-\p_\a^2)^{1/2}$, or equivalently $\L=\H\p_\a$, is to represent fractional derivatives. Moreover \eqref{E:linear} is invariant under 
\begin{equation}\label{D:G2}
\k \mapsto \Big(\frac12t\partial_t+\frac13\a\partial_\a\Big)\k=:\G_2\k.\end{equation}
A straightforward calculation then reveals that 
\begin{align}\label{E:GL2}
\Big\|&\frac12t\Lambda^{1/2}\k_{\a t}+\frac13\a\Lambda^{1/2}\k_{\a\a}
+\frac16\Lambda^{1/2}\k_\a\Big\|_{L^2}^2(t)
+\Big\|\frac12t\k_{tt}-\frac12\k_t-\frac13\a\k_{\a t}\Big\|_{L^2}^2(t)\hspace*{-.2in}\\
&=\|\Lambda^{1/2}\p_\a \G_2\k\|_{L^2}^2(t)+\|\p_t\G_2\k\|_{L^2}^2(t) \notag \\
&=\|\Lambda^{1/2}\p_\a \G_2\k\|_{L^2}^2(0)+\|\partial_t\G_2\k\|_{L^2}^2(0) \notag \\
&=\Big\|\frac13\a \Lambda^{1/2}\k_{\a\a}+\frac16\Lambda^{1/2}\k_\a\Big\|_{L^2}^2(0)
+\Big\|\frac12\k_t-\frac13\a\k_{\a t}\Big\|_{L^2}^2(0).\notag
\end{align}
That is to say, the solution at any time $t\neq 0$ gains a $1/2$ derivative of smoothness relative to the initial data at the expense of one power of $\a$. Note that $\partial_t \sim \Lambda^{1/2}\p_\a$ (see \eqref{E:linear}). Taking $\G_2^k$ in \eqref{E:GL2}, $k\geq 1$ an integer, furthermore, leads to that the solution acquires $k/2$ derivatives of smoothness at the expense of $k$ powers of $\a$, explaining the gain of high regularities due to the effects of surface tension in the linear motions of water waves. 

\medskip

The present purpose is to make rigorous that the gain of high regularities for \eqref{E:linear} carries over to the system of water waves under the influence of surface tension, which is severely nonlinear. The main result (see Theorem \ref{T:main}) states at heart that if $\k$ solves \eqref{E:kappa1} then 
\begin{multline}\label{E:main1}
t^k\|\lll\a\rrr^{-k}\Lambda^{k/2+1/2}\p_\a^{k+3}\k\|_{L^2}(t) \\ \leq  
C\Big(\sum_{k'=0}^k\|(\a\p_\a)^{k'}\k\|_{H^{k-k'+7/2}}(0), 
\sum_{k'=0}^k\|(\a\p_\a)^{k'}\k_t\|_{H^{k-k'+2}}(0)\Big)\end{multline}
at any time $t$ within the interval of existence for $k\geq 1$ an integer. 

Reinforcing the argument leading to \eqref{E:GL2} and the discussion following it, the proof relies upon energy estimates for the system \eqref{E:kappa1}-\eqref{E:u1} under $\p_\a$ as well as~$\G_2$. Energy expressions are ``nonlinear", however, to overcome the strength of the nonlinearity versus the weakness of dispersion (see Remark \ref{R:energy}).

\medskip

The present strategy seems not to promote \eqref{E:main1} to infinite gain of regularity, as opposed to in \cite{CKS}, for instance, for Korteweg-de Vries type partial differential equations, although it is desirable in particular in link\footnote{If the water wave problem with surface tension were to enjoy infinite gain of regularity then one might be able to cook up smooth initial data which would develop an $H^k$ singularity in finite time, since the problem is time reversible.} with ``breaking" of surface water waves. Indeed weights enter initial data in the form of $\a\p_\a$ and its powers, instructing that one cannot separate them from smoothness.

\subsection*{Remarks upon gravity waves}In the gravity wave setting, where $\We=\infty$ and $g>0$ in the system \eqref{E:kappa1}-\eqref{E:u1}, no regularizing effects are expected. The associated linear equation 
\begin{equation}\label{E:linearg}
\p_t^2\k+\H\p_\a\k=0\end{equation}
after normalization of parameters enjoys conservation of $\|\L^{1/2}\k\|_{L^2}^2(t)+\|\p_t\k\|_{L^2}^2(t)$, obviously, analogously to \eqref{E:L2} in the case of $\We<\infty$. Moreover \eqref{E:linearg} is invariant under $\k\mapsto (\frac12t\p_t+\a\p_\a)\k=: \G_g\k$. But $\|\L^{1/2}\G_g\k\|_{L^2}^2(t)+\|\p_t\G_g\k\|_{L^2}^2(t)$ at any time~$t$ remains merely as smooth as at $t=0$, irrelevant to improvements in the solution. Note that $\p_t\sim \H\L^{1/2}$. To the contrary, $\p_t \sim \L^{1/2}\p_\a$ in the capillary wave setting, i.e. $\We<\infty$ (see \eqref{E:linear}). In other words, gain of regularities for water waves is an attribute of the effects of surface tension.

\medskip

The water wave problem under the influence of gravity is dispersive, nevertheless, in the sense that waves with different frequencies propagate at different velocities in the linear dynamics (see \eqref{E:dispersion}). In particular the method of stationary phase applies to \eqref{E:linearg} to yield that the amplitude of the solution decays in time like $t^{-1/2}$ as $t\to \infty$. Recently in \cite{Wu3} dispersive estimates plus a profound understanding of the nonlinearity shed light toward global existence for small data. 

Formulating the problem as a system of second-order in time nonlinear dispersive equations and establishing energy estimates under invariant vector fields of the associated linear operator, the proof in \cite{Wu3} is related to ours. Accordingly the present treatment is potentially useful in studies of the long-time dynamics of water waves under the influence of surface tension.

\newpage

\section{Formulation and preliminaries}\label{S:preliminaries}

A concise account is given of the approach taken in \cite{Am}, \cite{AM1} and \cite{CHS} to formulate the water wave problem under the influence of surface tension. Various operators and commutators are discussed. Preliminary estimates of nonlinearities are established. 

Throughout, the two-dimensional space is identified with the complex plane and $z\cdot w=\Re(\bar{z}w)$ expresses the inner product on $\R^2$ (or $\mathbb{C}$), where $\bar{z}$ denotes the~complex conjugate of $z$. 

\subsection{Formulations}\label{SS:formulation}
The water wave problem under the influence of surface tension in the simplest form concerns wave motions at the interface separating in two spatial dimensions an infinitely-deep irrotational flow of an incompressible inviscid fluid from vacuum, and subject to the Laplace-Young jump condition\footnote{The jump of the pressure across the fluid surface is proportional to its curvature.} for the pressure. The effects of gravity are neglected. The fluid surface is assumed to be as\-ymp\-tot\-i\-cal\-ly flat and the flow in the far field to be nearly at rest. 

Let the parametric curve $z(\alpha, t)$, $\a\in\R$, in the complex plane represent the fluid surface at time $t$ and let the Birkhoff-Rott integral
\begin{equation}\label{D:BR}
\overline{\W}(\alpha,t)=\frac{1}{2\pi i}PV\int^\infty_{-\infty}\frac{\gamma(\b,t)}{z(\a,t)-z(\b,t)}~d\b 
\qquad (\alpha \in \R)\end{equation}
pertain to the curve's velocity, where $\gamma$ denotes the vortex sheet strength\footnote{The vorticity is null in the bulk of the fluid, but it is concentrated at the fluid surface, though. The vortex sheet strength measures the amplitude of the vorticity at the fluid surface.}. We take Birkhoff's approach to vortex sheets but in the presence of the effects of surface tension (see \cite{Am} and references therein) to formulate the problem.

Let's write the evolution equation of $z$ as
\[ z_t=\U\,\frac{iz_\alpha}{|z_\alpha|}+\V\, \frac{z_\alpha}{|z_\alpha|}.\]
Namely, $\U$ denotes the normal velocity of the fluid surface and $\V$ is its tangential velocity. Introducing
\[ s_\a=|z_\a| \quad \text{and}\quad \theta=\arg(z_\a)\quad\]
we make a straightforward calculation to obtain that 
\[s_{\alpha t}= \V_\alpha-\U \theta_\alpha \quad \text{and}\quad
\theta_t =\frac{1}{s_\alpha} (\U_\alpha + \V \theta_\alpha).\]
To interpret, $s$ measures the curve's arclength and $\theta$ is the tangent angle that the curve forms with the horizontal. In studies of surface water waves, the Biot-Savart law dictates that $\U=\W\cdot \frac{iz_\a}{|z_\a|}$ whereas $\V$ merely reparametrizes the fluid surface (see \cite[Section 2.1]{AM1} and references therein). Therefore we may choose $\V$ to be anything we wish. Following \cite{CHS} we require that $s_{\alpha}=1$ everywhere on $\R$ at each time. Accordingly
\begin{equation}\label{E:SSD}
z_\a=e^{i\theta} \quad \text{and}\quad \V=\p_\a^{-1}(\theta_\a \W\cdot iz_\a).
\end{equation}
If the fluid surface is initially parametrized so that $|z_\a|=1$ everywhere on $\R$ then it will remain likewise at subsequent times. To recapitulate,  
\begin{equation}\label{E:theta}
\theta_t =(\W\cdot iz_\a)_\a + \V \theta_\alpha\end{equation} 
serves as the equation of motion for the fluid surface, where $z_\a$ and $\V$ are specified upon solving equations in \eqref{E:SSD}. 

Bernoulli's equation for potential flows and the Laplace-Young jump condition for the pressure moreover yield that
\begin{equation}\label{E:gamma}
\gamma_t=\frac{1}{\We} \theta_{\alpha\alpha}+((\V-\W \cdot z_\alpha)\gamma)_\alpha-2\W_t 
\cdot z_\alpha-\frac{1}{2}\gamma \gamma_\alpha+2(\V-\W\cdot z_\alpha) \W_\alpha \cdot z_\alpha,
\end{equation}
where $\frac{1}{\We}\in (0,\infty)$ is a dimensionless parameter representing surface tension. Details are discussed in \cite[Appendix B]{AM1} and \cite[Section 2.1]{CHS}, for instance. 

To summarize, the water wave problem under the influence of surface tension is formulated as the system consisting of equations \eqref{E:theta}, \eqref{E:gamma} and supplemented with equations in \eqref{E:SSD}. In what follows $\frac{1}{\We}=2$ to simplify the exposition.

\medskip

As in \cite{CHS} and \cite{AM1} we promptly revise \eqref{E:theta} and \eqref{E:gamma} through better understanding $\W$ and $\V-\mathbf{W}\cdot z_\a$. 

Specifically we approximate the Birkhoff-Rott integral (see \eqref{D:BR}) by the Hilbert transform and write 
\begin{equation}\label{E:W}
\overline{\W}=\frac{1}{2i}\H \Big(\frac{\gamma}{z_\alpha}\Big)+\B \gamma,\end{equation}
where  
\begin{equation}\label{D:B}
\B f(\alpha,t)=\frac{1}{2\pi i}\int^\infty_{-\infty} \Big( \frac{1}{z(\alpha,t)-z(\beta,t)}
-\frac{1}{z_\beta(\beta,t)(\alpha-\beta)}\Big) f(\beta,t)~d\beta.\end{equation}
Then
\begin{gather}
\W_\alpha \cdot iz_\alpha =\frac12\H \gamma_\alpha +\mathbf{m}\cdot iz_\alpha\quad\text{and}\quad
\W_\alpha \cdot z_\alpha =-\frac12\H(\gamma \theta_\alpha) +\mathbf{m}\cdot z_\alpha,\label{E:W_a} 
\intertext{where}
\overline{\mathbf{m}}=z_\a \B\Big(\frac{\gamma_\a}{z_\a}-\frac{\gamma z_{\a\a}}{z_\a^2}\Big)
+\frac{z_\a}{2i}\Big[\H, \frac{1}{z_\a^{2}}\Big]\Big(\gamma_\a-\frac{\gamma z_{\a\a}}{z_\a}\Big).\label{D:V} 
\end{gather} 
A full calculation leading to \eqref{E:W_a} and \eqref{D:V} is found in \cite[Section 2.2]{Am}, albeit in the periodic wave setting. The idea is to differentiate $\W\approx\frac12\H(i\gamma z_\a)$ (see \eqref{E:W}) to learn that $\W_\a\approx\frac12\H(\gamma_\a) iz_\a-\frac12\H(\gamma\theta_\a)z_\a$. 

To proceed, let
\begin{equation}\label{D:q}q=\frac12 \gamma-(\V-\W\cdot z_\a)\end{equation}
measure the difference between the Lagrangian tangential velocity $\W\cdot z_\alpha+\frac12 \gamma$ of a fluid particle at the surface and the curve's tangential velocity $\V$ (see \cite[Sec\-tion 2.3]{AM1} and \cite[Sec\-tion 2.2]{CHS}, for instance). A lengthy yet explicit calculation then reveals that 
\begin{equation}\label{D:u}
\k:=\theta_\a\quad\text{and}\quad u:=q_\a=\frac12\gamma_\a+\W_\a\cdot z_\a,\end{equation}
thanks to equations in \eqref{E:SSD}, satisfy that
\begin{equation}\label{E:kappa-u}
(\conv)\k=\H \p_\a u+u\k+r^\k_{\a} \quad\text{and}\quad 
(\conv)u=\p_\a^2\k-p\k+r^u,
\end{equation}
where  
\begin{gather}
r^\k=-\H(\mathbf{m}\cdot z_\alpha)+\mathbf{m}\cdot iz_\a
\quad\text{and}\quad r^u=-u^2+(\H u+r^\k)^2, \label{D:R}\\
p=(\conv)\W \cdot iz_\alpha+\frac12 \gamma(\conv)\theta. \label{D:p}
\end{gather}
Note that $\k$ measures the curvature of the fluid surface. Moreover
\begin{equation}\label{E:theta'}
(\conv)\theta=\H u+r^\k.\end{equation}
A thorough derivation of equations in \eqref{E:kappa-u}, \eqref{D:R}, \eqref{D:p} and \eqref{E:theta'} in the periodic wave setting is found in \cite[Section 2.4]{AM1}. We pause to remark that $-p$ defines the outward normal derivative of the pressure at the fluid surface. In the absence of the effects of surface tension, $p(\a,0)\geq p>0$ pointwise in $\a \in \R$ for some constant~$p$, namely Taylor's sign condition, ensures well-posedness of the associated initial value problem.

\medskip

The system formed by equations in \eqref{E:kappa-u} and supplemented with equations in  \eqref{D:R}, \eqref{D:p}, \eqref{D:q} is equivalent to the vortex sheet formulation \eqref{E:theta}-\eqref{E:gamma}, \eqref{E:SSD} of the water wave problem under the influence of surface tension. Indeed, if the fluid surface is determined upon integrating the former equation in \eqref{E:SSD} such that $z(\a,t) -\a \to 0$ as $|\a| \to \infty$ at each time $t$ then the mapping $q \mapsto \gamma$ is one-to-one. We refer the reader to \cite[Section 2.3]{AM1} and references therein. 

\medskip

Moreover an explicit calculation manifests that the system \eqref{E:kappa-u}, \eqref{D:R}, \eqref{D:p}, \eqref{D:q} enjoys the scaling symmetry under 
\begin{equation}\label{E:scaling}
\k(\a,t)\mapsto \lambda\k(\l\a, \l^{3/2}t)\quad\text{and}\quad u(\a,t)\mapsto \l^{3/2}u(\l\a, \l^{3/2}t)
\end{equation}
and correspondingly $z(\a,t)\mapsto \l^{-1}\k(\l\a, \l^{3/2}t)$ and $\gamma(\a,t)\mapsto \l^{1/2}\gamma(\l\a, \l^{3/2}t)$ for any $\lambda>0$. 

\subsection{Assorted properties of operators}\label{SS:operators}
Gathered in this subsection are properties of various operators and commutators, which will be used throughout. 

\medskip

The operator $\p_t^2-\H\p_\a^3$ remains invariant under the vector fields\footnote{The operator is invariant under $\p_t$ and $\frac13\a\p_t+\frac12t\H\p_\a^2$ as well. The former unfortunately does not induce gain of regularities whereas the latter is equivalent to $\G_2$ for the present purpose; see the discussion in Section \ref{S:intro} following \eqref{E:linear} and \eqref{E:linearg}.}
\[\Gamma_1:=\p_\a\quad\text{and}\quad \G_2:=\frac12t\p_t+\frac13\a\p_\a\]
in the sense that $[\p_t^2-\H\p_\a^3,\G_1]=0$ and $[\p_t^2-\H\p_\a^3, \G_2]=\p_t^2-\H\p_\a^3$. By the way the linear system 
\[\p_t\k=\H\p_\a u\quad\text{and}\quad\p_t u=\p_\a^2\k\quad\] 
associated with \eqref{E:kappa-u} becomes 
\[\p_t^2\k-\H\p_\a^3\k=0\quad\text{and}\quad\p_t^2u-\H\p_\a^3u=0\] after differentiation in time. Since
\[ \qquad [\partial_\alpha,\G_2]=\frac13\partial_\alpha,\qquad [\partial_t,\G_2]=\frac12\partial_t,
\qquad [\partial_\alpha^{-1}, \G_2]=-\frac13\partial_\alpha^{-1}\]
it follows that
\begin{equation}\label{E:commLj}
\p_\a \G_2^j=\Big(\G_2+\frac13\Big)^j \p_\a, \quad
\partial_t\G_2^j=\Big(\G_2+\frac12\Big)^j \partial_t,\quad
\G_2^j\partial_\alpha^{-1}=\partial_\alpha^{-1}\Big(\G_2+\frac13\Big)^j,
\end{equation} 
respectively, for $j\geq0$ an integer. In general
\begin{equation}\label{E:commj}
[A^j, B]=\sum_{m=1}^j A^{j-m}[A,B]A^{m-1}\quad \text{for $j\geq 1$ an integer.}\end{equation}
Moreover
\begin{equation}\label{E:commq}
\begin{split}
[\partial_\alpha, \conv]=u\partial_\alpha\quad\text{and}\quad [\partial_t, \conv]=q_t\partial_\alpha,\\
[\G_2,\conv]=\Big(\G_2 q+\frac12 q\Big)\partial_\alpha-\frac32(\conv).\quad
\end{split}\end{equation}
If 
\[\K f(\alpha,t)=PV\int^\infty_{-\infty} K(\alpha, \beta;t)f(\beta,t)~d\beta\qquad (\a \in \R),\]
where $K$ is continuously differentiable except along the diagonal $\a=\b$,  then
\begin{equation}\label{E:commK}
\begin{split}
[\p_\a, \K]f(\alpha,t)=&\int (\p_\a+\p_\b)K(\a,\b;t)f(\b,t)~d\b, \\
[\p_t, \K]f(\alpha,t)=&\int \p_t K(\a,\b;t)f(\b,t)~d\b, \\
[\G_2,\K]f(\a,t)=&\int \Big(\frac12t\p_t+\frac13\a\p_\a+\frac13\b\p_\b\Big)K(\a,\b;t)f(\b,t)~d\b +\K f(\a,t).
\end{split}\end{equation}
The proof of \eqref{E:commLj} through \eqref{E:commK} is straightforward. Hence we omit the detail.

\medskip

In various proofs in the text we shall deal with integral operators of the form 
\begin{subequations}\begin{equation}
\mathcal{K}_1(A,a)f(\a,t)=PV\int^\infty_{-\infty}
A(Qa_0)(\a,\b;t)\prod_{n'=1}^n(Qa_{n'})(\a,\b;t)\, \frac{f(\b,t)}{\a-\b}~d\b \label{D:T}
\end{equation}or
\begin{equation}
\mathcal{K}_2(A,a)f(\alpha,t)=\int^\infty_{-\infty}
A(Qa_0)(\a,\b;t)\prod_{n'=1}^n (Qa_{n'})(\a,\b;t)\,\p_\beta f(\b,t)~d\b,\,\,\,\label{D:T'}
\end{equation}\end{subequations}
where 
\begin{equation}\label{D:Q}
(Qa)(\a,\b;t):=\frac{a(\a,t)-a(\b,t)}{\a-\b}=\int^1_0a_\a(\sigma\a+(1-\sigma)\b,t)~d\sigma \end{equation}
is shorthand for the divided difference. Examples include
\begin{equation}
\B f(\a,t)=\frac{1}{2\pi i}\int^\infty_{-\infty} \log(Qz)(\a,\b;t)\p_\b\Big(\frac{f}{z_\b}\Big)(\b,t)~d\b \label{D:B'} 
\end{equation}
(see \eqref{D:B}) and
\begin{equation}
[\H,a]\p_\a f(\a,t)=-\frac{1}{\pi}\int^\infty_{-\infty}(Qa)(\a,\b;t)\p_\b f(\b,t)~d\b.\label{D:[H,a]} 
\end{equation}
We promptly explore their boundedness, where $t$ is fixed and suppressed to simplify the exposition.

\begin{lemma}\label{L:T}
If $A$ is smooth in the range of $Qa_0$ then 
\begin{equation}\label{E:T'}
\|\mathcal{K}_1(A,a)f\|_{L^2}, \|\mathcal{K}_2(A,a)f\|_{L^2}\leq 
C_1\|a_{1\a}\|_{L^{r^*}}\|a_{2\a}\|_{L^\infty}\cdots \|a_{n\a}\|_{L^\infty}
\|f\|_{L^{(3-r)^*}}\hspace*{-.2in}
\end{equation}
for $r=1$ or $2$, where $r^*$ in the range $[1,\infty]$ is defined such that $1/r+1/r^*=1$ and $C_1>0$ a constant depends upon $A$ as well as $\|a_{0\a}\|_{L^\infty}$. 

\smallskip

If in particular $A(x)=\frac{1}{x^d}$ for $d\geq 1$ an integer and if ${\displaystyle \inf_{\a\neq \b}|(Qa_0)(\a,\b)|}\geq Q>0$ for some constant $Q$ then 
\begin{equation}\label{E:T} 
\|\mathcal{K}_1(A,a)f\|_{L^2}, \|\mathcal{K}_2(A,a)f\|_{L^2}\leq \frac{C_2}{Q^d}
\|a_{1\a}\|_{L^{r^*}}\|a_{2\a}\|_{L^\infty}\cdots \|a_{n\a}\|_{L^\infty}
\|f\|_{L^{(3-r)^*}}\hspace*{-.2in}\end{equation}
for $r=1$ or $2$, where $C_2>0$ a constant is independent of $A$, $a_0, \dots, a_n$ and $f$.  
\end{lemma}

If $r=1$ then $r^*=\infty$ and $(3-r)^*=2$; in the case of $r=2$, correspondingly, $r^*=2$ while $(3-r)^*=\infty$. Here and in the sequel $C_0$, $C_1$, $C_2, ...$ mean positive but generic constants unless otherwise specified. 

\begin{proof}
In the case of $r=1$, a celebrated result of Coifman, McIntosh and Meyer bears out \eqref{E:T'}, and further \eqref{E:T} through keeping track of how $C_1$ depends upon $A$ and $a_{0}$. We refer the reader to \cite[Proposition~3.2 and Proposition~3.3]{Wu3} and references therein. In the case of $r=2$, the proof based upon the T(b) theorem is found in \cite[Proposition~3.2 and Proposition~3.3]{Wu3}, for instance. 
\end{proof}

\begin{lemma}\label{L:Q}
For $j_1$, $j_2$, $k_1$, $k_2\geq 0$ integers and for $1\leq r \leq \infty$ a real
\begin{equation}\label{E:Q}
\|\partial_\a^{j_1}\p_\b^{j_2}\Gamma_{2\a}^{k_1}\Gamma_{2\b}^{k_2}Qa\|_{L^r_{\a,\b}}
\leq C_0\|\partial_\a^{j_1+j_2}\G_2^{k_1+k_2}a_\a\|_{L^r},\end{equation}
where $\G_{2\a}=\frac12t\p_t+\frac13\a\p_\a$ and $\G_{2\b}=\frac12t\p_t+\frac13\b\p_\b$ to emphasize the spatial variable.
\end{lemma}

\begin{proof}
If $1\leq r <\infty$ then we split the interval of integration on the right side of
\[ \p_\a^{j_1}\p_\b^{j_2}\Gamma_{2\a}^{k_1}\Gamma_{2\b}^{k_2}Qa(\a,\b)=
\int^1_0\sigma^{j_1+k_1}(1-\sigma)^{j_2+k_2}
\p_\a^{j_1+j_2}\G_2^{k_1+k_2} a_\a(\sigma\a+(1-\sigma)\b)~d\sigma\]
(see \eqref{D:Q}), using the change of variables $\a \mapsto \sigma\a+(1-\sigma)\b$ for $\sigma \in [0,1/2]$ and $\b \mapsto \sigma\a+(1-\sigma)\b$ for $\sigma \in [1/2,1]$, and we integrate the $r$-th power over both $\a$ and $\b$. The proof in the case of $r=\infty$ is trivial.
\end{proof}

\begin{remark}[Ramification of \eqref{E:T}]\label{R:ET}\rm
An integral operator either in \eqref{D:T} or in \eqref{D:T'}, where $\p_\a^{j_1}\p_\b^{j_2}\Gamma_{2\a}^{k_1}\Gamma_{2\b}^{k_2}Qa_{n'}$ substitutes $Qa_{n'}$ for $j_1$, $j_2$, $k_1$, $k_2\geq 0$ integers and for $n'$ an index, yet which fulfills hypotheses of Lemma \ref{L:T}, satisfies \eqref{E:T} (or \eqref{E:T'}), but for which $\|\p_\a^{j_1+j_2}\G_2^{k_1+k_2}a_{n'\a}\|_{L^r}$ is in place of $\|a_{n'\a}\|_{L^r}$, $r=2$ or $\infty$. 

Utilizing \eqref{E:Q} in the argument leading to \eqref{E:T} (or \eqref{E:T'}), the proof is straightforward. Hence we omit the detail. 
\end{remark} 

In particular 
\begin{equation}\label{E:[H,a]}
\|\partial_\a^j[\H,a]\partial_\a^kf\|_{L^2}\leq C_0\|\partial_\a^{j+k}a\|_{L^\infty}\|f\|_{L^2} 
\quad \text{for $j+k\geq 1$,}\end{equation}
$j\geq 0$ and $k\geq 0$ integers.

\medskip

Concluding the subsection we discuss properties of $\L=(-\p_\a^2)^{1/2}$, or equivalently $\L=\H\p_\a$, in the real-valued function setting, which will be relevant  in Section \ref{S:proof} to energy estimates.

Defined via the Fourier transform as $\widehat{\L f}(\xi)=|\xi|\hat{f}(\xi)$, it is self-adjoint, its fractional powers, too, and it is linked with half-integer Sobolev spaces. Specifically
\[\|f\|_{H^{1/2}}^2=\int^\infty_{-\infty} (f^2+f\L f)~d\a.\]
To compare, $\int (\H f_1)f_2=-\int f_1(\H f_2)$ and $\|\H f\|_{L^2}=\|f\|_{L^2}$. Moreover the commutator of a fractional power of $\Lambda$ is ``smoothing". 

\begin{lemma}\label{L:smoothing}It follows that
\begin{equation}\label{E:Lambda}
\int^\infty_{-\infty}  af\L f~d\a \leq C_0\|a\|_{H^2}\|f\|_{H^{1/2}}^2\text{ and }
\int^\infty_{-\infty} a f_\a \L f~d\a \leq C_0\|a_{\a\a}\|_{L^\infty}\|f\|_{L^2}^2.\hspace*{-.2in}
\end{equation}
\end{lemma}

\begin{proof}
Note that $\L^{1/2}$ is self-adjoint, and we calculate
\[\int af\L f~d\a=\int a(\Lambda^{1/2}f)^2~d\a +\int (\L^{1/2}[\L^{1/2}, a]f)f~d\a.\]
The first term on the right side is bounded by $\|a\|_{L^\infty}\|\L^{1/2}f\|_{L^2}^2$, obviously, while we claim that the second term on the right side is bounded by $\|\L a\|_{L^\infty}\|f\|_{H^{1/2}}^2$ up to multiplication by a constant. Indeed, since
\[ (\L^{1/2}[\Lambda^{1/2},a]f)^{\wedge}(\xi)=
\frac{1}{2\pi} \int^\infty_{-\infty} |\xi|^{1/2}(|\xi|^{1/2}-|\eta|^{1/2})\hat{a}(\xi-\eta)\hat{f}(\eta)~d\eta \] 
in the Fourier space and since $|\xi|^{1/2}||\xi|^{1/2}-|\eta|^{1/2}|\leq C_0|\xi-\eta|$ for all $\xi,\eta \in \R$ by brutal force (see the proof of \cite[Lemma~2.14]{Yos1}, for instance), Young's inequality and the Parseval theorem yield that\footnote{The commutator of $\Lambda^{1/2}$ therefore is ``smoothing"; compare it to \eqref{E:[H,a]}.} 
\[\|\L^{1/2}[\Lambda^{1/2},a]f\|_{L^2} \leq C_0\||\xi|\hat{a}\|_{L^1}\|f\|_{L^2}.\]
H\"older's inequality then proves the claim. The first inequality in \eqref{E:Lambda} follows by a Sobolev inequality. 

Note that the adjoint of $\H$ is $-\H$, on the other hand, and we calculate
\[\int a f_\a\L f~d\a=-\int a(\H f_\a)f_\a~d\a-\int ([\H,a]f_\a)f_\a~d\a=\frac12\int([\H, a]f_\a)_\a f~d\a.\]
The second inequality in \eqref{E:Lambda} then follows by \eqref{E:[H,a]} and by H\"ol\-der's inequality. \end{proof}

\subsection{Estimates of $r^\k$, $r^u$ and $p$}\label{SS:R1}
Let the plane curve $z(\a,t)$, $\a \in \R$, describe the fluid surface at time $t$ corresponding to a solution of the system \eqref{E:theta}-\eqref{E:gamma}, \eqref{E:SSD} for the water wave problem with surface tension, over the interval of time $[0,T]$ for some $T>0$. Throughout the subsection we assume that 
\begin{equation}\label{E:cord-arc}
\inf_{\a\neq\b}\Big|\frac{z(\a,t)-z(\b,t)}{\a-\b}\Big|\geq Q>0
\end{equation}
at each $t \in [0,T]$ for some constant $Q$. To interpret, the curve lacks self-intersections during the evolution over $[0,T]$. Similar qualifications were used in \cite{CHS}, \cite{Wu3} among others in studies of surface water waves and in \cite{GHS}, for instance, accounting for vortex patches.

This subsection concerns preliminary estimates of $r^\k$, $r^u$ and $p$ (see \eqref{D:R} and \eqref{D:p}, respectively) as well as their derivatives under $\p_\a$ or $\G_2$, where $t \in [0,T]$ is fixed and suppressed to simplify the exposition.

\medskip

Recall the notation that $C_0$ means a positive generic constant and $C(f_1,f_2,...)$ is a positive but polynomial expression in its arguments; $C_0$ and $C(f_1,f_2,...)$ which appear in different places in the text need not be the same.

\begin{lemma}[Estimates of $r^\k$]\label{L:R1}
For $j\geq 0$ an integer
\begin{align}
\|r^\k\|_{H^j}\leq &C(\|\k\|_{H^{\max(j-1,1)}})(1+ \|u\|_{H^1}+\|\gamma\|_{L^2});\label{E:R1j} 
\intertext{for $j\geq 0$ and $k\geq 1$ integers} 
\|\p_\a^j\G_2^kr^\k\|_{L^2} \leq &C \Big(\sum_{\ell=0,1}\sum_{k'=0}^{k-\ell}
\|\G_2^{k'}\k\|_{H^{\max(j-1,1)+\ell}}, \sum_{k'=0}^k\|\G_2^{k'}\theta\|_{L^2}\Big) \label{E:R1jk} \\
&\hspace*{1.25in}\cdot \Big(1+\sum_{k'=0}^k\|\G_2^{k'}u\|_{H^1}
+\sum_{k'=0}^k\|\G_2^{k'}\gamma\|_{L^2} \Big).\notag \end{align}
\end{lemma}

Colloquially speaking, $r^\k$ is smoother than $\k$ in the Sobolev space setting. Without depending upon $\theta$ explicitly, \eqref{E:R1j} refines the result in \cite[Proposition~2.1]{CHS}. 

\begin{proof}
In view of the former equation in \eqref{D:R} and those in \eqref{D:V}, the proof involves understanding the smoothness of $\mathcal{W}$ and the commutator of the Hilbert transform as well as their derivatives under $\p_\a$ or $\G_2$. It is readily verified from \eqref{D:B'}, \eqref{D:[H,a]} and from formulae in \eqref{E:commK} that
\begin{equation}\label{E:commB}
\begin{split}
[\partial_\a, \B]f(\a,t)=&\frac{1}{2\pi i} \int \Big(\frac{Qz_\a}{Qz}\Big)(\a,\b;t)
\partial_\b\Big(\frac{f}{z_\b}\Big)(\b,t)~d\b, \\
[\partial_t, \B]f(\a,t)=&\frac{1}{2\pi i} \int \Big(\frac{Qz_t}{Qz}\Big)(\a,\b;t)
\partial_\b\Big(\frac{f}{z_\b}\Big)(\b,t)~d\b,\\
[\G_2, \B]f(\a,t)=&\frac{1}{2\pi i} \int \Big(\frac{Q\G_2 z}{Qz}\Big)(\a,\b;t)
\partial_\b\Big(\frac{f}{z_\b}\Big)(\b,t)~d\b+\B f(\a,t) 
\end{split}\end{equation}
and 
\begin{equation}\label{E:commH}
[\p_t,[\H,a]]=[\H, a_t],\qquad [\G_2,[\H,a]]=[\H,\G_2 a]+\frac23[\H,a].\end{equation}
Moreover the former equation in \eqref{E:SSD} implies upon applications of the composition inequality (see Lemma \ref{L:comp} and references therein) that
\begin{equation}\label{E:z_a}\begin{aligned}
|z_\a|\equiv1&, \qquad
\|\p_\a^jz_\a\|_{L^2}\leq  C_0\|\k\|_{H^{j-1}}^j \quad &&\text{for $j\geq 1$ an integer,} \\
\|\p_\a^j\G_2^k&z_\a\|_{L^2}\leq C_0\sum_{k'=0}^k\|\G_2^{k'}\theta\|_{H^{j}}^{j+k}&& 
\text{for $j\geq 0$ and $k\geq 1$ integers.}
\end{aligned}\end{equation} 

\smallskip

We manipulate among \eqref{E:commj}, the first formula in \eqref{E:commB} and integration by parts to show that $\p_\a^j\B f$ and $\p_\a^j[\H,\frac{1}{z_\a^2}]f$, $j\geq 1$ an integer, are both made up of integral operators either in \eqref{D:T} or in \eqref{D:T'}, where $A(x)=\frac{1}{x^d}$ for $d\geq1$ an integer and $a_{n'}=z$ for all indices, allowing that $\p_\a^{j_1}\p_\b^{j_2}Qz$ substitutes $Qz$, where $j_1$, $j_2\geq0$ are integers such that $j_1+j_2\leq j$, and allowing that $\frac{f}{z_\a}$ is in place of $f$. We then handle each of the resulting integral operators by means of \eqref{E:T}, or its ramification in Remark~\ref{R:ET}, along with \eqref{E:Q} and the first two estimates in \eqref{E:z_a} to obtain that
\begin{equation}\label{E:Bj}
\|\p_\a^j\B f\|_{L^2}, \|\p_\a^j[\H, \tfrac{1}{z_\a^{2}}]f\|_{L^2} \leq C_0\|\k\|_{H^{j-1}}^{2j}\|f\|_{H^1}.
\end{equation}
Moreover $\big\|\gamma_\a-\frac{\gamma z_{\a\a}}{z_\a}\big\|_{H^1}\leq C_0(1+\|\k\|_{H^1})\|\gamma\|_{H^2}$ by brutal force. For $j=0$ Young's inequality ensures that 
$\big\|\big[\H, \frac{1}{z_\a^{2}}\big]\frac{\gamma z_{\a\a}}{z_\a}\big\|_{L^2}\leq C_0 
\big\|Q\big(\frac{1}{z_\a^2}\big)\big\|_{L^2}\|\gamma\|_{L^2}\big\|\frac{z_{\a\a}}{z_\a}\big\|_{L^2}$; 
otherwise \eqref{E:T} applies to integral operators on the right side of \eqref{D:V} in like manner, whence 
\begin{equation}\label{E:R10}
\|\mathbf{m}\cdot iz_\a\|_{L^2}, \|\mathbf{m}\cdot z_\a\|_{L^2}\leq C_0(1+\|\k\|_{H^1}^2)\|\gamma\|_{L^2} \end{equation}
by \eqref{E:Q} and the first two estimates in \eqref{E:z_a}. Collectively it follows after repeated use of the first two estimates in \eqref{E:z_a} and product inequalities (see Lemma \ref{L:prod} and references therein) that
\begin{equation}\label{E:R1j'}
\|\mathbf{m}\cdot iz_\a\|_{H^j},\|\mathbf{m}\cdot z_\a\|_{H^j},\|r^\k\|_{H^j}\leq C(\|\k\|_{H^{\max(j-1,1)}})\|\gamma\|_{H^2}
\end{equation}
for $j\geq 0$ an integer.

In order to promote \eqref{E:R1j'} to \eqref{E:R1j}, since 
\begin{equation}\label{E:gamma_a}
\gamma_\a=2u+\H(\gamma \k)-2\mathbf{m}\cdot z_\a\end{equation}
by the latter equation in \eqref{D:u} and by the former in \eqref{E:W_a} we infer from \eqref{E:R10}, \eqref{E:R1j'} and from product inequalities (see Lemma \ref{L:prod}), Young's inequality with~$\epsilon$ (see Lemma \ref{L:Young-e} and references therein) that $\|\gamma\|_{H^2}\leq C(\|\k\|_{H^1})+4\|u\|_{H^1}+\|\gamma\|_{L^2}$. Incidentally an induction argument leads to that
\begin{equation}\label{E:gammaj} 
\|\gamma\|_{H^j}\leq C(\|\k\|_{H^{\max(j-1,1)}})(1+\|u\|_{H^{j-1}}+\|\gamma\|_{L^2})\end{equation}
for $j\geq 1$ an integer. 

\medskip

The proof of \eqref{E:R1jk} proceeds similarly. For $j\geq 1$ and $k\geq 1$ integers we exploit \eqref{E:commj}, the first and the third formulae in \eqref{E:commB}, the second in \eqref{E:commH} and integration by parts to write $\partial_\a^j\G_2^k\B f$ and $\partial_\a^j\G_2^k[\H, \frac{1}{z_\a^2}] f$ both as sums of integral operators in \eqref{D:T} or \eqref{D:T'}, where $A(x)=\frac{1}{x^d}$ for $d\geq 1$ an integer and $a_{n'}=z$ for all indices, allowing that $\p_\a^{j_1}\p_\b^{j_2}\Gamma_{2\a}^{k_1}\Gamma_{2\b}^{k_2}Qz$ replaces $Qz$, where $j_1$, $j_2$, $k_1$, $k_2\geq 0$ are integers such that $j_1+j_2\leq j$ and $k_1+k_2\leq k$ (recall the notation that $\G_{2\a}=\frac12t\p_t+\frac13\a\p_\a$ and $\G_{2\b}=\frac12t\p_t+\frac13\b\p_\b$), and possibly that either $\G_2^{k'}f$ or $\G_2^{k'}\big(\frac{f}{z_\a}\big)$ is in lieu of $f$ for $0\leq k'\leq k$ an integer. It then follows by virtue of \eqref{E:T}, or the ramification in Remark~\ref{R:ET}, and by \eqref{E:Q} and estimates in \eqref{E:z_a} that 
\begin{multline}\label{E:Bjk}
\|\partial_\a^j\G_2^k\B f\|_{L^2}, \|\partial_\a^j\G_2^k[\H, \tfrac{1}{z_\a^2}] f\|_{L^2} \\
\leq C_0\Big(\sum_{\ell=0,1}\sum_{{k'}=0}^{k-\ell} \|\G_2^{k'}\theta\|_{H^{j+\ell}}^{2(j+k)}\Big)
\Big(\sum_{{k'}=0}^k \|\G_2^{k'}\theta\|_{H^1}^j\Big)\Big(\sum_{{k'}=0}^k \|\G_2^{k'}f\|_{H^1}\Big).
\end{multline}
Moreover
\[ \sum_{{k'}=0}^k \Big\|\G_2^{k'}\Big(\gamma_\a-\frac{\gamma z_{\a\a}}{z_\a}\Big)\Big\|_{H^1}
\leq C_0\Big(1+\sum_{{k'}=0}^k \|\G_2^{k'}\theta\|_{H^2}^k\Big)
\Big(\sum_{{k'}=0}^k\|\G_2^{k'}\gamma\|_{H^2}\Big)\]
for $k\geq 1$ an integer, using the first formula in \eqref{E:commLj}, \footnote{Throughout we tacitly exercise the first two formulae in \eqref{E:commLj} to swap $\p_\a$ or $\p_t$ with $\G_2$ up to a sum of smooth remainders whenever it is convenient to do so.}to interchange $\p_\a$ and $\G_2$ up to a sum of smooth remainders, and using estimates in \eqref{E:z_a}. For $j=0$ we appeal to \eqref{E:T}, Young's inequality and the last formulae in \eqref{E:commB} and in \eqref{E:commH} to deal with integral operators on the right side of \eqref{D:V} and their derivatives under $\G_2$; \eqref{E:Q}, estimates in \eqref{E:z_a} and product inequalities (see Lemma \ref{L:prod}) then yield~that  
\begin{equation}\label{E:R10k}
\|\G_2^k(\mathbf{m}\cdot iz_\a)\|_{L^2}, \|\G_2^k(\mathbf{m}\cdot z_\a)\|_{L^2}\leq C\Big(
\sum_{k'=0}^k\|\G_2^{k'}\theta\|_{H^2}\Big)\Big(\sum_{k'=0}^k\|\G_2^{k'}\gamma\|_{L^2}\Big)
\end{equation}
for $k\geq 1$ an integer. Together it follows after numerous applications of estimates in \eqref{E:z_a} and  product inequalities (see Lemma \ref{L:prod}) that
\begin{multline}\label{E:R1jk'} 
\|\p_\a^j\G_2^k(\mathbf{m}\cdot z_\a)\|_{L^2},\|\p_\a^j\G_2^k(\mathbf{m}\cdot iz_\a)\|_{L^2},
\|\p_\a^j\G_2^kr^\k\|_{L^2}\\ 
\leq C\Big(\sum_{\ell=0,1}\sum_{k'=0}^{k-\ell}\|\G_2^{k'}\theta\|_{H^{\max(j,2)+\ell}}\Big)\Big(\sum_{k'=0}^k \|\G_2^{k'}\gamma\|_{H^2}\Big)
\end{multline}
for $j\geq 0$ and $k\geq 1$ a integers.

It remains to relate high Sobolev norms of $\theta$, $\gamma$ and their derivatives under $\G_2$ to Sobolev norms of $\kappa$, $u$ and the $L^2$-norms of $\theta$,  $\gamma$, as well as their derivatives under $\G_2$. In view of \eqref{E:gamma_a} we infer from \eqref{E:R10k}, \eqref{E:R1jk'} and from product inequalities (see Lemma \ref{L:prod}), Young's inequality with~$\epsilon$ (see Lemma \ref{L:Young-e}) that
\begin{equation}\label{E:gammajk}
\|\G_2^k\gamma_\a\|_{H^j} \leq C\Big(\sum_{\ell=0,1}\sum_{k'=0}^{k-\ell}
\|\G_2^{k'}\theta\|_{H^{\max(j+1,2)+\ell}}\Big)\Big(1+\sum_{k'=0}^k \|\G_2^{k'}u\|_{H^j}
+\sum_{k'=0}^k\|\G_2^{k'}\gamma\|_{L^2}\Big)\hspace*{-.4in}
\end{equation}
for $j\geq 0$ and $k\geq 1$ integers. The proof closely resembles that of \eqref{E:gammaj}. Hence we omit the detail. Moreover the first formula in \eqref{E:commLj} manifests~that
\begin{equation}\label{E:thetajk}
\| \p_\a^j \G_2^k\theta\|_{L^2} \leq \|\p_\a^{j-1}\G_2^k\k\|_{L^2}
+C_0\sum_{k'=0}^{k-1}\|\G_2^{k'}\k\|_{H^{j-1}}\quad\text{for $j, k \geq 1$ integers.}
\end{equation}

At last, \eqref{E:gammajk} and \eqref{E:thetajk} upgrade \eqref{E:R1jk'} to \eqref{E:R1jk}. 
\end{proof}

\begin{corollary}[Estimates of $r^u$]\label{C:R2}
For $j\geq 0$ an integer
\begin{align}
\|r^u\|_{H^j}\leq &C(\|\k\|_{H^{\max(j-1, 1)}}, \|u\|_{H^{\max(j,1)}}, \|\gamma\|_{L^2});\label{E:R2j}
\intertext{for $j\geq 0$ and $k\geq 1$ integers}
\|\partial_\a^j\G_2^kr^u\|_{L^2} \leq &C \Big(\sum_{\ell=0,1}\sum_{k'=0}^{k-\ell}
\|\G_2^{k'}\k\|_{H^{\max(j-1,1)+\ell}}, \sum_{k'=0}^k\|\G_2^{k'}u\|_{H^{\max(j,1)}}, M_k\Big), \label{E:R2jk}
\end{align}where 
\vspace*{-.1in}
\begin{equation}\label{D:M}
M_k=\sum_{k'=0}^k\|\G_2^{k'}\theta\|_{L^2}+\sum_{k'=0}^k\|\G_2^{k'}\gamma\|_{L^2}.\end{equation}\end{corollary}

Colloquially speaking, $r^u$ is smoother than $\k$ or $u_\a$ in the Sobolev space setting.
 
\begin{proof}
Upon inspection of the latter equation in \eqref{D:R}, obviously, \eqref{E:R2j} and \eqref{E:R2jk} follow from \eqref{E:R1j} and \eqref{E:R1jk}, respectively.\end{proof}

Turning the attention to $p$, since
\begin{equation}\label{E:p'}
p=\W_t\cdot iz_\a+q\Big(\frac12\H\gamma_\a+\mathbf{m}\cdot iz_\a\Big)+\frac12\gamma(\H u+r^\k)
\end{equation}
by the latter equation in \eqref{E:W_a} and by \eqref{E:theta'} (see \eqref{D:R}) we shall make an effort to understand the smoothness of $\W_t$ and $q$ as well as their derivatives under $\p_\a$ or~$\G_2$. 

\subsection*{Estimates of $q$} 
We shall discuss $\G_2^kq$, $k\geq 0$ an integer, in the $L^\infty$-space setting. For, $\V$ on the right side of \eqref{D:q} involves anti-differentiation (see the latter equation in \eqref{E:SSD}) and hence it seems unwieldy in the $L^2$-space setting. Nevertheless we easily compute from the latter equation in \eqref{E:SSD} and from the last formula in \eqref{E:commLj} that
\begin{align}\label{E:Uk}
\|\G_2^k\V\|_{L^\infty} \leq &\int^\infty_{-\infty} 
\Big|\big(\G_2+\frac13\big)^k(\theta_\a \W\cdot iz_\a)\Big|~d\a\\ 
\leq &C_0\Big(\sum_{k'=0}^k\|\G_2^{k'}\k\|_{L^2}\Big)
\Big(\sum_{k'=0}^k\|\G_2^{k'}(\W\cdot iz_\a)\|_{L^2}\Big).\notag \end{align}
By the way $q$ enters equations in \eqref{E:kappa-u}, \eqref{D:R} and \eqref{D:p} merely as the ``velocity" of the convective derivative $\conv$. Accordingly one may dump it and its derivatives under the $L^\infty$-norm in the usual manner in the course of the energy method.

Moreover an explicit calculation against the inner product formula reveals that
\[\W\cdot iz_\a=\frac12\H\gamma
+\Re\Big(iz_\a\B\gamma+\frac{z_\a}{2}\Big[\H,\frac{1}{z_\a}\Big]\gamma\Big)\text{ and }
\W\cdot z_\a=\Re\Big(z_\a\B\gamma+\frac{z_\a}{2i}\Big[\H, \frac{1}{z_\a}\Big]\gamma\Big)\]
 (see \eqref{E:W}). We then run the argument in the proof of Lemma \ref{L:R1} leading to \eqref{E:Bj} and \eqref{E:Bjk} to bound the $L^2$-norm of $\p_\a^j\G_2^k[\H, \frac{1}{z_\a}]\gamma$, $j\geq 0$ and $k\geq 0$ integers, by Sobolev norms of $\k$, $u$ and the $L^2$-norms of $\theta$, $\gamma$, as well as their derivatives under~$\G_2$. Collectively it follows by virtue of \eqref{E:Bj}, \eqref{E:Bjk}, \eqref{E:gammaj}, \eqref{E:gammajk}, \eqref{E:z_a}, \eqref{E:thetajk} and by product inequalities (see Lemma \ref{L:prod})  that
\begin{equation}\label{E:Wj}
\|\W\cdot iz_\a\|_{H^j}, \|\W\cdot z_\a\|_{H^j}  \leq 
C(\|\k\|_{H^{\max(j-1,1)}})(1+\|u\|_{H^{\max(j-1,0)}}+\|\gamma\|_{L^2})
\end{equation}for $j\geq 0$ an integer and 
\begin{multline}\label{E:Wjk}
\|\p_\a^j\G_2^k (\W\cdot iz_\a)\|_{L^2},\|\p_\a^j\G_2^k(\W\cdot z_\a)\|_{L^2} \\ 
\leq  C\Big(\sum_{\ell=0,1} \sum_{k'=0}^{k-\ell}\|\G_2^{k'}\k\|_{H^{\max(j-1,1)+\ell}}, 
\sum_{k'=0}^k\|\G_2^{k'}\theta\|_{L^2}\Big) \\ \cdot \Big(1+\sum_{k'=0}^k \|\G_2^{k'}u\|_{H^{\max(j-1,0)}}
+\sum_{k'=0}^k\|\G_2^{k'}\gamma\|_{L^2}\Big)\hspace*{-.1in}
\end{multline}
for $j\geq 0$, $k\geq 1$ integers. The proof closely resembles those of Lemma \ref{L:R1}. Hence we omit the detail. 

Upon inspection of \eqref{D:q} we consequently deduce from \eqref{E:gammajk}, \eqref{E:Uk}, \eqref{E:Wj},~\eqref{E:Wjk} and from a So\-bol\-ev inequality that 
\begin{align} 
\|q\|_{L^\infty}\leq &C(\|\k\|_{H^1})(1+\|u\|_{L^2}+\|\gamma\|_{L^2}),\label{E:q0} \\
\|\G_2^k q\|_{L^\infty}\leq &C\Big(\sum_{k'=0}^{k}\|\G_2^{k'}\k\|_{H^1}, 
\sum_{k'=0}^k\|\G_2^{k'}\theta\|_{L^2}\Big)\Big(1+\sum_{k'=0}^k \|\G_2^{k'}u\|_{L^2}
+\sum_{k'=0}^k\|\G_2^{k'}\gamma\|_{L^2}\Big)\hspace*{-.5in}\label{E:qk}
\intertext{for $k\geq 1$ an integer. Moreover the latter equation in \eqref{D:u} and the first formula in \eqref{E:commLj} manifest that}
\|\p_\a^j\G_2^kq\|_{L^2}\leq &C_0\sum_{k'=0}^k\|\G_2^{k'}u\|_{H^{j-1}}
\quad\text{for $j\geq 1$ and $k\geq 0$ integers.} \label{E:qjk}
\end{align}

\smallskip

For future reference, note in view of the former equation in \eqref{E:kappa-u} that 
\begin{align}
\|u\|_{H^{j+1}} \leq &4\|\k_t\|_{H^j}+C(\|\k\|_{H^{j+1}})(1+\|u\|_{L^2}+\|\gamma\|_{L^2}) \label{E:uj} 
\intertext{for $j\geq 0$ an integer and}
\|\p_\a^j\G_2^ku_\a\|_{L^2} \leq &4 \sum_{k'=0}^k\|\G_2^{k'}\k_t\|_{H^j} +C\Big(\sum_{k'=0}^{k}\|\G_2^{k'}\k\|_{H^1}, \sum_{k'=0}^k\|\G_2^{k'}\theta\|_{L^2}\Big)\label{E:ujk} \\&\hspace*{1.3in}\cdot \Big(1+\sum_{k'=0}^k \|\G_2^{k'}u\|_{L^2}+\sum_{k'=0}^k\|\G_2^{k'}\gamma\|_{L^2}\Big) \notag
\end{align}
for $j\geq 0$, $k\geq 0$ integers. It makes rigorous that $u_\a$ has the same regularity as $\k_t$, $\k_\a$ (see the former equation in \eqref{E:kappa-u}). Combining \eqref{E:q0}-\eqref{E:qjk}, \eqref{E:R1j}, \eqref{E:R1jk} and product inequalities (see Lemma \ref{L:prod}), Young's inequality with~$\epsilon$ (see Lemma \ref{L:Young-e}), the proof is very similar to that of \eqref{E:gammaj} and \eqref{E:gammajk}. Hence we leave out the detail.

\subsection*{Estimates of $\W_t$}Differentiating \eqref{E:W} in time we arrange the result as 
\begin{equation}\label{E:W_t}
\overline{\W}_t=\mathbf{T}_1+\mathbf{T}_2+\Rr\gamma,\end{equation}
where \begin{gather}
\mathbf{T}_1=\frac{1}{2i}\H\Big(\frac{\gamma_t}{z_\a}\Big)+\B\gamma_t\quad\text{and}\quad
\mathbf{T}_2=\frac{1}{2i}\H\Big(\frac{\gamma z_{\a t}}{z_\a^2}\Big)+\B\Big(\frac{\gamma z_{\a t}}{z_\a}\Big),\label{D:W12}  \\
\Rr f(\a,t)=\frac{1}{2\pi i} \int^\infty_{-\infty}\Big(\frac{Qz_t}{Qz}\Big)(\a,\b;t)
\p_\b\Big(\frac{f}{z_\b}\Big)(\b,t)~d\b. \label{D:rr}\end{gather}
Details are discussed in \cite[Section~6]{Am} and in the proof of \cite[Lemma~2.4]{CHS}, for instance.

\medskip

Since $\frac{z_{\a t}}{z_\a}=i\theta_t$ (see the former equation in \eqref{E:SSD}) it follows by virtue of \eqref{E:Bj}, \eqref{E:Bjk}, \eqref{E:gammaj}, \eqref{E:gammajk}, \eqref{E:z_a}, \eqref{E:thetajk} and by product inequalities (see Lemma \ref{L:prod})~that 
\begin{align}
\|\mathbf{T}_2\|_{H^j}\leq &C(\|\k\|_{H^{\max(j-1,1)}}, \|u\|_{H^{\max(j-1,0)}},  
\|\theta_t\|_{H^{\max(j,1)}}, \|\gamma\|_{L^2})\label{E:W2j'}  
\intertext{for $j\geq 0$ an integer and} 
\|\p_\a^j\G_2^k\mathbf{T}_2\|_{L^2} \leq &C 
\Big(\sum_{\ell=0,1}\sum_{k'=0}^{k-\ell}\|\G_2^{k'}\k\|_{H^{\max(j-1,1)+\ell}},\label{E:W2jk'} \\
&\hspace*{.5in}\sum_{k'=0}^k\|\G_2^{k'}u\|_{H^{\max(j-1,0)}}, 
\sum_{k'=0}^k\|\G_2^{k'}\theta_t\|_{H^{\max(j,1)}}, M_k\Big) \notag
\end{align}
for $j\geq 0$, $k\geq 1$ integers, where $M_k$ is in \eqref{D:M}.

Moreover it is readily verified from \eqref{D:rr} and from formulae in \eqref{E:commK} that
\begin{align}\label{E:commR}
[\partial_\a, \Rr]f(\alpha,t)=&\frac{1}{2\pi i}\int \Big(\frac{Qz_{\a t}}{Qz}-\frac{(Qz_t)(Qz_\a)}{(Qz)^2}\Big)
(\a,\b;t)\partial_\b\Big(\frac{f}{z_\b}\Big) (\beta,t)~d\b,\notag \\
[\partial_t\,, \Rr]f(\alpha,t)=&\frac{1}{2\pi i}\int \Big(\frac{Qz_{tt}}{Qz}-\frac{(Qz_t)^2}{(Qz)^2}\Big)
(\a,\b;t)\partial_\b\Big(\frac{f}{z_\b}\Big) (\beta,t)~d\b, \\
[\G_2\,, \Rr]f(\alpha,t)=&\frac{1}{2\pi i}\int \Big(\frac{Q\G_2 z_t}{Qz}
-\frac{(Qz_t)(Q\G_2 z)}{(Qz)^2}\Big)(\a,\b;t)\partial_\b\Big(\frac{f}{z_\b}\Big) (\beta,t)~d\b \notag \\
&+\Rr f(\alpha,t). \notag
\end{align}
We then run the argument leading to \eqref{E:Bj} and \eqref{E:Bjk} mutatis mutandis to bound the $L^2$-norm of $\p_\a^j\G_2^k\Rr\gamma$, $j\geq 0$ and $k\geq 0$ integers, by the right side of either \eqref{E:W2j'} or \eqref{E:W2jk'}, but where $Qz_t$ is in lieu of $\theta_t$. The proof relies upon \eqref{E:Q}, \eqref{E:z_a}, \eqref{E:thetajk}, \eqref{E:gammaj}, \eqref{E:gammajk} and upon product inequalities (see Lemma \ref{L:prod}). It is straightforward. Hence we omit the detail.

To proceed, \eqref{E:theta'} trades Sobolev norms of $\theta_t$ and its derivatives under $\G_2$ with those of $\k$, $u$, $r^\k$, and $q$ in the $L^\infty$-space setting. Lemma~\ref{L:R1} and \eqref{E:q0}-\eqref{E:qjk} further render them in terms of Sobolev norms of $\k$, $u$ and the $L^2$-norms of $\theta$,~$\gamma$, as well as their derivatives under $\G_2$. The $L^2$-norm of $\p_\a^j\G_2^kQz_t$, $j\geq 0$ and $k\geq 0$ integers, is bounded likewise by virtue of \eqref{E:Q}, the former equation in \eqref{E:SSD}, estimates in \eqref{E:z_a} and by product inequalities (see Lemma \ref{L:prod}). To recapitulate,
\begin{align}
\|\theta_t\|_{H^j}, \|Qz_t\|_{H^j}\leq &C(\|\k\|_{H^{\max(j,1)}}, \|u\|_{H^{\max(j,1)}}, \|\gamma\|_{L^2})\label{E:Qztj} \\
\intertext{for $j\geq 0$ an integer and}
\|\p_\a^j\G_2^k\theta_t\|_{L^2}, \|\p_\a^j\G_2^kQz_t\|_{L^2}  \leq &C\Big(
\sum_{k'=0}^k\|\G_2^{k'}\k\|_{H^{\max(j,1)}}, \sum_{k'=0}^{k}\|\G_2^{k'}u\|_{H^{\max(j,1)}}, M_k\Big)\hspace*{-.2in}\label{E:Qztjk}  \end{align}
for $j\geq 0$, $k\geq 1$ integers, where $M_k$ is in \eqref{D:M}.

Upon reinforcing \eqref{E:W2j'} and \eqref{E:W2jk'} with \eqref{E:Qztj} and \eqref{E:Qztjk}, respectively, therefore
\begin{align} 
\qquad\|\mathbf{T}_2\|_{H^j}, \|\Rr\gamma\|_{H^j} \leq &C(\|\k\|_{H^{\max(j,1)}}, \|u\|_{H^{\max(j,1)}}, \|\gamma\|_{L^2})\label{E:W2j} \\
\intertext{for $j\geq 0$ an integer and}
\|\p_\a^j\G_2^k\mathbf{T}_2\|_{L^2}, \|\p_\a^j\G_2^k\Rr\gamma\|_{L^2} \leq &C \Big(
\sum_{k'=0}^{k}\|\G_2^{k'}\k\|_{H^{\max(j,1)}},\sum_{k'=0}^k\|\G_2^{k'}u\|_{H^{\max(j,1)}},  M_k\Big)
\hspace*{-.21in}\label{E:W2jk} \end{align}
for $j\geq 0$, $k\geq 1$ integers.

\begin{lemma}[Estimates of $\gamma_t$]\label{L:gamma_t}For $j\geq 0$ an integer
\begin{align}
\|\gamma_t\|_{H^j}\leq &C(\|\k\|_{H^{j+1}}, \|u\|_{H^{\max(j,1)}}, \|\gamma\|_{L^2});\label{E:gamma_tj}
\intertext{for $j\geq 0$ and $k\geq 1$ integers}
\|\G_2^{k} \gamma_t\|_{H^j} \leq &C\Big(\sum_{k'=0}^{k}\|\G_2^{k'} \k\|_{H^{j+1}},
\sum_{k'=0}^k \|\G_2^{k'} u\|_{H^{\max(j,1)}}, M_k\Big),\label{E:gamma_tjk}
\end{align}
where $M_k$ is in \eqref{D:M}.
\end{lemma}

\begin{proof}
Notice that \eqref{E:gamma} is an integral equation for $\gamma_t$. Indeed we calculate $\mathbf{W}_t\cdot z_\a$ using \eqref{E:W_t} and \eqref{D:W12}, \eqref{D:rr} to rewrite it as
\begin{multline}\label{E:gamma'}
(1+2\J)\gamma_t=2\theta_{\a\a} +(\V-\W\cdot z_\a)\gamma_\a-\gamma \W_\a\cdot z_\a \\
-\H(\gamma\theta_t)-2\mathbf{b}\cdot z_\a -\frac12\gamma\gamma_\a+2(\V-\W\cdot z_\a)\W_\a \cdot z_\a,
\end{multline}
where 
\begin{equation}\label{D:J}
\J f=\Re \Big( \frac{z_\a}{2}\B f+\frac{z_\a}{2i}\Big[\H,\frac{1}{z_\a}\Big]f\Big)\,\text{ and }\, 
\overline{\mathbf{b}}=\B(i\gamma\theta_t)
+\frac{1}{2i}\Big[\H, \frac{1}{z_\a}\Big](\gamma\theta_t)+\Rr \gamma.\hspace*{-.2in}\end{equation} 
Details are discussed in \cite[Section~6]{Am} and in the proof of \cite[Lemma~2.4]{CHS}, for instance. 

We then employ various results and arguments previously worked out and we make repeated use of product inequalities (see Lemma \ref{L:prod}) to bound the $L^2$-norm of $\p_\a^j\G_2^k(1+2\J)\gamma_t$, $j\geq 0$ and $k\geq 0$ integers, by the right side of either \eqref{E:gamma_tj} or \eqref{E:gamma_tjk}. The proof relies upon \eqref{E:Bj}, \eqref{E:Bjk}, \eqref{E:z_a}, \eqref{E:thetajk}, \eqref{E:gammaj}, \eqref{E:gammajk}, \eqref{E:Wj},~\eqref{E:Wjk}, \eqref{E:q0}-\eqref{E:qjk}, \eqref{E:Qztj}, \eqref{E:Qztjk}, \eqref{E:W2j}, \eqref{E:W2jk} and upon the argument leading to \eqref{E:Bj} and \eqref{E:Bjk} (or \eqref{E:Wj} and \eqref{E:Wjk}). It is straightforward. Hence we omit the detail. 

The argument leading to \eqref{E:Wj} and \eqref{E:Wjk} moreover furnishes that 
\begin{align*}
\|\J\gamma_t \|_{H^j}\leq &C(\|\k\|_{H^{j-1}})\|\gamma_t\|_{H^1}
\quad\text{for $j\geq 0$ an integer,} \\
\|\p_\a^j\G_2^k\J\gamma_t \|_{L^2} \leq &C\Big(\sum_{\ell=0,1}\sum_{k'=0}^{k-\ell}
\|\G_2^{k'}\k\|_{H^{\max(j-1,1)+\ell}}, \sum_{k'=0}^k\|\G_2^{k'}\theta\|_{L^2}\Big)
\Big(\sum_{k'=0}^k\|\G_2^{k'}\gamma_t\|_{H^1}\Big)
\end{align*}
for $j\geq 0$ and $k\geq 1$ integers.  

Since $\p_\a^j\G_2^k\gamma_t=\p_\a^j\G_2^k(1+2\J)\gamma_t-2\p_\a^j\G_2^k\J\gamma_t$, therefore, \eqref{E:gamma_tj} and \eqref{E:gamma_tjk} follow collectively after a routine application of Young's inequality with~$\epsilon$ (see Lemma~\ref{L:Young-e}). The proof is similar to that of \eqref{E:gammaj} and \eqref{E:gammajk}. Hence we leave out the detail. Instead we refer the reader to \cite[Section~6]{Am} or the proof of \cite[Lemma~2.4]{CHS}, for instance, for some detail relevant to \eqref{E:gamma_tj}.
\end{proof}

Upon inspection of the former equation in \eqref{D:W12} we consequently deduce from \eqref{E:Bj}, \eqref{E:Bjk}, \eqref{E:z_a}, \eqref{E:thetajk}, \eqref{E:gamma_tj}, \eqref{E:gamma_tjk} and from product inequalities (see~Lem\-ma~\ref{L:prod}) that the $L^2$-norm of $\p_\a^j\G_2^k\mathbf{T}_1$, $j\geq 0$ and $k\geq 0$ integers, is bounded by the right side of either \eqref{E:gamma_tj} or \eqref{E:gamma_tjk}. Together with \eqref{E:W2j} and \eqref{E:W2jk}, it implies that
\begin{align}
\|\W_t\|_{H^j}\leq &C(\|\kappa\|_{H^{j+1}}, \|u\|_{H^{\max(j,1)}}, \|\gamma\|_{L^2})\label{E:Wtj}  \\
\intertext{for $j\geq 0$ an integer and}
\|\p_\a^j \G_2^k \W_t\|_{L^2} \leq &C\Big(\sum_{k'=0}^{k}\|\G_2^{k'} \k\|_{H^{j+1}},
\sum_{k'=0}^k \|\G_2^{k'} u\|_{H^{\max(j,1)}}, M_k\Big) \label{E:Wtjk}\end{align}
for $j\geq 0$, $k\geq 1$ integers, where $M_k$ is in \eqref{D:M}.  

\medskip 

At last, the right side of \eqref{E:p'} as well as its derivatives under $\p_\a$ or $\G_2$ may be term\-wise handled by virtue of \eqref{E:z_a}, \eqref{E:thetajk}, \eqref{E:gammaj}, \eqref{E:gammajk}, \eqref{E:R1j}, \eqref{E:R1jk}, \eqref{E:q0}-\eqref{E:qjk}, \eqref{E:Wtj}, \eqref{E:Wtjk} and by product inequalities (see Lemma \ref{L:prod}). We summarize the conclusion.

\begin{lemma}[Estimates of $p$]\label{L:p}
For $j\geq 0$ an integer
\begin{align}
\|p\|_{H^j}\leq &C(\|\kappa\|_{H^{j+1}}, \|u\|_{H^{\max(j,1)}},\|\gamma\|_{L^2}); \label{E:pj} 
\intertext{for $j\geq 0$ and $k\geq 1$ integers}
\|\p_\a^j \G_2^k p\|_{L^2} \leq &C\Big(\sum_{k'=0}^{k}\|\G_2^{k'} \k\|_{H^{j+1}},
\sum_{k'=0}^k \|\G_2^{k'} u\|_{H^{\max(j,1)}}, M_k\Big), \label{E:pjk} \end{align}
where $M_k$ is in \eqref{D:M}.
\end{lemma}

As a matter of fact $p$ has the same regularity as $\k_\a$ in the Sobolev space setting; Lemma \ref{L:R3} will elaborate it. Moreover \eqref{E:pj} does not depend upon $\theta$ explicitly, refining the result in \cite[Proposition 2.1]{CHS}.

\newpage

\section{Reformulation and the main result}\label{S:formulation}

The water wave problem under the influence of surface tension is reformulated as a system of second-order in time nonlinear dispersive equations. The main result is stated. Estimates of smooth nonlinearities are established.

\subsection{Reformulation}\label{SS:reformulation}
Investigating regularizing effects encompassed by surface tension in the propagation of water waves we revamp equations in \eqref{E:kappa-u} so that their linear part is distinguished by the operator $\p_t^2-\H\p_\a^3$. For, the linear system 
\[\p_t\k=\H \p_\a u\quad \text{and}\quad \p_tu=\p_\a^2\k\quad\] 
associated with \eqref{E:kappa-u} is not invariant under $\k\mapsto \G_2\k$ and $u\mapsto \G_2 u$ although the system
\[ \p_t^2\k-\H\p_\a^3\k=0\quad \text{and}\quad \p_t^2u-\H\p_\a^3u=0\] 
is (see the discussion at the beginning of Section \ref{SS:operators}). As a matter of fact the latter equations were used in Section \ref{S:intro} to explain the gain of high regularities due to the effects of surface tension in the linear motions of water waves. 

Specifically we differentiate both equations in \eqref{E:kappa-u} in time, or better yet under $\conv$, and make an explicit calculation to obtain that
\begin{equation}\label{E:system0}
(\conv)^2\k-\H \partial_\a^3\k=G^{\k} \quad\text{and}\quad
(\conv)^2u-\H \partial_\a^3u=G^{u},\end{equation} where
\begin{align}
G^\k=&-[\H,q]u_{\a\a}-\H(uu_\a)+\H(-p\k+r^u)_\a+(\conv)(u\k+r^\k_\a),\label{D:gk} \\
G^u=&\k u_{\a\a}+\k_\a u_\a-u\k_{\a\a}+r^\k_{\a\a\a}+(\conv)(-p\k+r^u).\label{D:gu}\end{align}
Equations in \eqref{E:system0} give a prominence to dispersion, characterized by the operator $\p_t^2-\H\p_\a^3$, but they are severely\footnote{Techniques of oscillatory integral operators (see \cite[Section~1.3]{CHS} and references therein) bear out that a solution of $\p_t^2\k-\H\p_\a^3\k=G(\a,t)$ gains maximally $7/4$ spatial derivatives of smoothness over the inhomogeneity in the $L^2$-space setting; the energy method, to compare, controls up to $3/2$ derivatives. They are not sufficient to control nonlinearities containing more than $7/4$ spatial derivatives, e.g. $q\k_{\a t}$ and $q^2\k_{\a\a}$.} nonlinear, though. Notably the left sides contain $q\k_{\a t}$, $q^2\k_{\a\a}$ and $qu_{\a t}$, $q^2u_{\a\a}$. 

\medskip

We promptly look into the smoothness of the right sides of equations in \eqref{E:system0}. As we shall elucidate in the forthcoming section, nonlinearities on the left sides, despite high numbers of derivatives, accommodate the energy method.

Thanks to \eqref{E:[H,a]} the first term on the right side of \eqref{D:gk} is regarded smoother than $u_\a$, and in turn $\k_t$  in view of the former equation in \eqref{E:kappa-u}. Indeed \eqref{E:R1j} implies that $r^\k$ is smoother than $\k$. Since $r^u$ seems to have the same regularity as $u$ by \eqref{E:R2j} and since $p$ seems to have the same regularity as $\k_\a$ by \eqref{E:pj} we argue in view of the former equation in \eqref{E:kappa-u} that $-\H(p\k_\a)$, $\H r^u_{\a}$ and $(\conv)r^\k_{\a}$ enjoy the same regularity as $\k_t$ or $\k_\a$. Therefore the right side of \eqref{D:gk} minus $-\H(p_\a\k)$ and $u_t\k$ is expected smoother than $\k_{\a t}$ or $\k_{\a\a}$. 

Similarly \eqref{E:R1j}, \eqref{E:R2j}, \eqref{E:pj} and the latter equation in \eqref{E:kappa-u} suggest that the right side of \eqref{D:gu} but for $\k u_{\a\a}$ and $-p_t\k$ is smoother than $u_{\a t}$ or $u_{\a\a}$.

To the contrary, \eqref{E:pj} and the latter equation in \eqref{E:kappa-u} instruct that $-\H(p_\a\k)$ and $u_t\k$ on the right side of \eqref{D:gk} have the same regularity as $\k_{\a\a}$. Correspondingly $-p_t\k$ on the right side of \eqref{D:gu} seems to possess the same regularity as $\k_{\a t}$, and in turn $u_{\a\a}$ in view of the former equation in \eqref{E:kappa-u}. To aggravate, $p$ involves the Birk\-hoff-Rott integral in the principal part (see the proof of Lemma \ref{L:p}) and hence it may be unwieldy under integration by parts.

A critically important observation that was made in \cite{AM1},\cite{CHS} and which we shall describe, nevertheless, relates $p_\a$ to the Hilbert transform of $\k_{\a\a}$, ensuring that the most singular contributions of $G^\k$ as well as $G^u$ are differential.

\medskip

As in Section \ref{SS:R1} let $z(\cdot\,, t): \R\to\mathbb{C}$ represent the fluid surface at time $t$ corresponding to a solution of the system \eqref{E:theta}-\eqref{E:gamma}, \eqref{E:SSD} for the water wave problem with surface tension, over the interval of time $[0,T]$ for some $T>0$. The working assumption in the lemma below is that \eqref{E:cord-arc} holds during the time interval $[0,T]$ for some constant $Q$. Various estimates in the previous subsection are then at our command. Moreover $t\in[0,T]$ is fixed and ignored in the lemma for the sake of exposition. 

\begin{lemma}[Estimates of $p_\a$]\label{L:R3}Upon writing 
\begin{equation} \label{E:p_a}
p_\alpha=\H \k_{\a\a}+r^p,\end{equation}
\begin{align} 
\|r^p\|_{H^j}\leq &C(\|\kappa\|_{H^{j+1}}, \|u\|_{H^{j+1}}, \|\gamma\|_{L^2})\label{E:R3j}
\intertext{for $j\geq 0$ an integer and}
\|\p_\a^j \G_2^k r^p\|_{L^2} \leq &C\Big(\sum_{k'=0}^{k}\|\G_2^{k'} \k\|_{H^{j+1}},
\sum_{k'=0}^k \|\G_2^{k'} u\|_{H^{j+1}}, M_k\Big)\label{E:R3jk}
\end{align}
for $j\geq 0$, $k\geq 1$ an integer, where $M_k$ is in \eqref{D:M}.
\end{lemma}

Colloquially speaking, $p_\a=\H\k_{\a\a}$ plus smooth remainders. 

\begin{proof}
After differentiation of \eqref{D:p} with respect to $\alpha$, a rather lengthy yet explicit calculation reveals that
\begin{multline}\label{D:R3}r^p=\H(-p\k+r^u)-\frac12\gamma\theta_{\a t}
-\H(\mathbf{v}\cdot z_\a+\theta_t\W_\a\cdot iz_\a+q(\W_\a\cdot z_\a)_\a) \\
+q(\mathbf{m}\cdot iz_\a)_\a+q\theta_\a \W_\a\cdot z_\a+R,
\end{multline}where
\begin{align}\overline{\mathbf{v}}=
&\frac{1}{2i}\Big[\H, \frac{1}{z_\a}\Big]\gamma_{\a t}-\frac{1}{2i}\H\Big(\frac{\gamma_t z_{\a\a}}{z_\a^2}\Big)+(\B\gamma_t+\mathbf{T}_2+\Rr\gamma)_\a,\label{D:Y} \\
R=&-\theta_\a (\conv)\W \cdot z_\a+u\W_\a\cdot iz_\a \label{D:pi} \\
&\qquad\qquad+\frac12\gamma_\a(\H u+r^\k)+\frac12\gamma(\H u+u\k+r^\k_{\a})+\frac12\gamma u\theta_\a.\notag
\end{align}
Indeed 
\[p_\a=\W_{\a t} \cdot iz_\a+q\W_{\a\a} \cdot iz_\a+R,\] 
where $\W_{\alpha t}\approx \frac12\H(\gamma_{\a t})iz_\a-\frac12\H(\gamma\theta_{\a t})z_\a$ and $\W_{\a\a}\approx \frac12\H(\gamma_{\a\a})iz_\a-\frac12\H(\gamma\theta_{\a\a})z_\a$ by the same kind of argument following \eqref{E:W_a} and \eqref{D:V}; moreover $\frac12(\conv)\gamma_\a \approx \k_{\a\a}$ by the latter equations in \eqref{D:u} and in \eqref{E:kappa-u}. Details are discussed in \cite[Sec\-tion 2.6]{AM1}, for instance, albeit in the periodic wave setting. 

We then employ various results and arguments worked out in the previous subsection and we make repeated use of product inequalities (see Lemma \ref{L:prod}) to bound the $L^2$-norms of the right side of \eqref{D:R3} (and \eqref{D:Y}, \eqref{D:pi}) as well as derivatives under $\p_\a$ or $\G_2$ by the right side of either \eqref{E:R3j} or \eqref{E:R3jk}. The proof relies upon \eqref{E:R1j}, \eqref{E:R1jk},  \eqref{E:z_a}, \eqref{E:thetajk}, \eqref{E:gammaj}, \eqref{E:gammajk}, \eqref{E:Wj}, \eqref{E:Wjk}, \eqref{E:q0}-\eqref{E:qjk}, \eqref{E:W2j}, \eqref{E:W2jk}, \eqref{E:gamma_tj}, \eqref{E:gamma_tjk}, \eqref{E:Wtj}, \eqref{E:Wtjk}, \eqref{E:pj}, \eqref{E:pjk} and upon the argument leading to \eqref{E:Bj} and \eqref{E:Bjk} (or \eqref{E:Wj} and \eqref{E:Wjk}). It is straightforward. Hence we omit the detail. Instead we refer the reader to the proof of \cite[Proposition 2.4]{AM1} or \cite[Lemma~2.4]{CHS}, for instance, for some detail relevant to \eqref{E:R3j}. 
\end{proof}

To recapitulate, the principal parts of $-\H(\k p_\a)$ and $\k u_t$ on the right side of \eqref{D:gk} are $\k\k_{\a\a}$ thanks to Lemma \ref{L:R3} and the latter equation in \eqref{E:kappa-u}, respectively, and $\k p_{t\a}=\k u_{\a\a\a}$ plus smooth remainders in like manner. We then differentiate the latter equation in \eqref{E:system0} in the spatial variable and make another explicit, albeit lengthy, calculation to ultimately obtain that 
\begin{alignat}{2}
&(\conv)^2\k\,\,\,-\H \partial_\a^3\k\,\,\,-\,2\k\p_\a^2\k&&=R^\k(\k, u;\theta, \gamma),\label{E:kappa} \\
&(\conv)^2u_\a-\H \partial_\a^3 u_\a-2\k\p_\a^2u_\a&&=R^u(\k,u;\theta,\gamma), \label{E:u}
\end{alignat}where 
\begin{align}
R^\k=&\H[\H,\k]\k_{\a\a}-[\H,q]u_{\a\a}-\H(uu_\a)+\H(-\k r^p-p\k_\a+r^u_{\a}) \label{D:Gk} \\
&+\k(-p\k+r^u)+u(\conv)\k+(\conv)r^\k_{\a} \notag
\intertext{and $R^u=(R^{u,1})_\a+R^{u,2}$,}
R^{u,1}=&\k_\a u_\a-u\k_{\a\a}+r^\k_{\a\a\a}+p(\conv)\k+(\conv)r^u, \label{D:Gu1}\\
R^{u,2}=&-(\conv)(uu_\a)-u(\conv)u_\a-u^2u_\a \label{D:Gu2} \\
&+\k[\H,q]\k_{\a\a}-\k\H(\k u_{\a\a}+\k_\a u_\a -u\k_{\a\a}+r^\k_{\a\a\a}) \notag \\
&-\k(\conv)r^p-\k_\a(\conv)p - \k up_\a.\notag
\end{align}

The system \eqref{E:kappa}-\eqref{E:u}, \eqref{D:Gk}, \eqref{D:Gu1}, \eqref{D:Gu2} is equivalent to the system \eqref{E:system0}, \eqref{D:gk}, \eqref{D:gu}, obviously, provided that it is supplemented with  \eqref{E:p_a}, and in turn the system \eqref{E:kappa-u}. To see it let's break down both equations in \eqref{E:system0} into first-order in time equations and write 
\begin{alignat}{2} 
&(\conv)\k=\varphi\quad\text{and}\quad&&(\conv)\varphi=\H\p_\a^3\k+G^{\k}, \label{E:ky3}\\
&(\conv)u=v\quad\,\text{and}\quad&&(\conv)v\,=\H\p_\a^3u+G^{u}.\label{E:uv3}
\end{alignat}
Comparing the former equations in \eqref{E:ky3} and \eqref{E:uv3} to those in \eqref{E:kappa-u} dictates that $\varphi=\H u_\a+u\k+r^\k_{\a}$ and $v=\k_{\a\a}+p\k+r^u$. It is then straightforward to verify that the latter equations agree with the others in \eqref{E:kappa-u} up to constants of integration, which by the way are null if the corresponding fluid surface is asymptotically flat in an appropriate sense.

To summarize, the water wave problem under the influence of surface tension is reformulated as the system of equations \eqref{E:kappa} and \eqref{E:u}, where $R^\k$ and $R^u$ are specified \footnote{Since $\theta$ is recovered from $\k$ by quadrature and since $q\mapsto \gamma$ is one-to-one (see Section \ref{SS:formulation} and references therein), $R^\k$ and $R^u$ may be thought of depending merely upon $\k$ and $u$. But we do not explicitly practice it, though.}in terms of $\k$, $u$, $\theta$, $\gamma$ upon solving \eqref{D:Gk} and \eqref{D:Gu1}, \eqref{D:Gu2}, respectively, with the assistance of equations in \eqref{D:R}, \eqref{D:p}, \eqref{D:q} (see Section~\ref{SS:formulation}) and \eqref{E:p_a}, \eqref{D:R3}-\eqref{D:pi}, and where $q$ is defined by \eqref{D:q} in like manner. It serves as the basis to demonstrate the gain of high regularities for the problem. 

Note that \eqref{E:kappa} and \eqref{E:u} give a prominence to dispersion, inherited from \eqref{E:system0}. Furthermore their right sides contain less numbers of derivatives than nonlinearities on the left sides. Proposition \ref{P:G} will elaborate it. 

\medskip

Perhaps one learns from the preceding proofs of equivalence that \eqref{E:u} (and \eqref{D:Gu1}, \eqref{D:Gu2}), or its quadrature form, solo is equivalent to the system \eqref{E:theta}-\eqref{E:gamma}, \eqref{E:SSD}, provided that it is supplemented with equations in \eqref{D:R}, \eqref{D:p}, \eqref{D:q} and  \eqref{E:p_a}, and that in addition it is coupled with the latter equation in \eqref{E:kappa-u}. Indeed $\k$ may be determined upon solving the latter equation in \eqref{E:kappa-u} via elliptic PDE methods. In particular the water wave problem with surface tension was formulated in \cite{CHS} as a second-order in time nonlinear dispersive equation for $q$; it was then used to demonstrate the local smoothing effect and Strichartz estimates. 

The problem may alternatively be formulated as \eqref{E:kappa} (and \eqref{D:Gk}), for $\k$, where $u$ is determined upon solving the latter equation in \eqref{E:kappa-u} via the method of characteristics, for instance. 

\subsection{Statement of the main result}\label{SS:result}
Let's prepare the initial data for the vortex sheet formulation \eqref{E:theta}-\eqref{E:gamma}, \eqref{E:SSD} of the water wave problem with surface tension, as well as the reformulation \eqref{E:kappa}-\eqref{E:u}, \eqref{D:Gk}, \eqref{D:Gu1}, \eqref{D:Gu2} of the problem.

Let the parametric curve $z_0(\a)$, $\a \in \R$, in the complex plane represent the initial fluid surface such that $z_0(\a)-\a \to 0$ as $|\a|\to \infty$ and $|z_{0\a}|=1$ everywhere on $\R$. That is to say, the curve is asymptotically flat and parametrized by arclength. We assume that 
\begin{equation}\label{E:cord-arc0}
\inf_{\a\neq \b}\Big|\frac{z_0(\a)-z_0(\b)}{\a-\b}\Big| \geq Q>0\end{equation}
and $\|z_{0\a}\|_{C^1}\leq \frac{1}{Q}$ for some constant $Q$. To interpret, $z_0$ lacks self-inter\-sec\-tions and cusps. Accordingly it separates the plane into two simply-connected, un\-bounded $C^2$ regions. Let $\theta_0=\arg(z_{0\a})$ and let $\gamma_0:\R\to \R$ denote the initial vortex sheet strength. 

Suppose that $z(\a, t)$, $\a \in \R$, describes the fluid surface at time $t$ corresponding to the solution of the system \eqref{E:theta}-\eqref{E:gamma}, \eqref{E:SSD}, subject to the initial conditions 
\[\theta(\cdot,0)=\theta_0 \quad \text{and} \quad \gamma(\cdot,0)=\gamma_0.\]
Indeed $z$ is determined upon integrating the former equation in \eqref{E:SSD} and requiring that $z(\a,t)-\a \to 0$ as $|\a|\to \infty$ at each time $t$. Obviously $|z_\a|=1$ everywhere on $\R$ at each time. Concerning regular\footnote{$z \in H^k_{loc}(\R)$ at each time for some $k\geq 3$ suffices; see Theorem \ref{T:main}.} solutions, furthermore, \eqref{E:cord-arc} holds during an interval of time, say $[0,T_Q]$, possibly after replacing $Q$ by $\epsilon Q$ for any fixed but small $\epsilon>0$. Therefore $z$ at each time during the evolution over $[0,T_Q]$ separates the plane into two simply-connected, unbounded $C^2$ regions. 

To proceed, let 
\begin{equation}\label{E:k0u0} 
\k_0=\theta_{0\a}\quad\text{and}\quad u_0=\frac12\gamma_0+\W_{0\a}\cdot z_{0\a},
\end{equation} 
where $\W_0$ denotes the Birkhoff-Rott integral (see \eqref{D:BR}) of $\gamma_0$ along the curve $z_0$. They form an initial data pair for the system \eqref{E:kappa-u}, \eqref{D:R}, \eqref{D:p}, \eqref{D:q} (see \eqref{D:u}), and vice versa any initial data for the water wave problem with surface tension may be recast in terms of $\k_0$ and $u_0$, as deliberated in Section \ref{SS:formulation}. Moreover, let 
\begin{equation}\label{E:k1u1} 
\k_1=\H u_{0\a}-q_0\k_{0\a}+u_0\k_0+r^\k_{0\a} \quad\text{and}\quad
u_1=\k_{0\a\a}-q_0u_{0\a}-p_0\k_0+r^u_{0},\end{equation}
where $r^\k_{0}$, $r^u_{0}$ and $p_0$, $q_0$ are specified upon evaluating the right sides of equations in \eqref{D:R} and \eqref{D:p}, \eqref{D:q}, respectively, at $\theta_0$, $\gamma_0$ and $u_0$. Together with $\k_0$ and $u_0$, they  make an initial data quadruple for the system \eqref{E:kappa}-\eqref{E:u}, \eqref{D:Gk}, \eqref{D:Gu1}, \eqref{D:Gu2}. In light of equations in \eqref{E:kappa-u}, indeed, $\k_1=\k_t(\cdot, 0)$ and $u_1=u_t(\cdot, 0)$.

\medskip

If $(\a\p_\a)^{k'}\theta_0\in H^{k-k'+K+7/2}(\R)$ and $(\a\p_\a)^{k'}\gamma_0\in H^{k-k'+K+3/2}(\R)$ for $k, K\geq 0$ integers and for all $0\leq k'\leq k$ integers then 
\[((\a\p_\a)^{k'}\k_0, (\a\p_\a)^{k'}u_0) \in H^{k-k'+K+5/2}(\R)\times H^{k-k'+K+3/2}(\R)\] 
and $((\a\p_\a)^{k'}\k_1, (\a\p_\a)^{k'}u_1) \in H^{k-k'+K+1}(\R) \times H^{k-k'+K}(\R)$ for all $0\leq k'\leq k$ integers by virtue of \eqref{E:Wj}, \eqref{E:Wjk} and \eqref{E:R1j}, \eqref{E:R1jk}, \eqref{E:R2j}, \eqref{E:R2jk}, \eqref{E:pj}, \eqref{E:pjk}, \eqref{E:q0}-\eqref{E:qjk}.

\medskip

The main result of the article concerns gain of high regularities for the system \eqref{E:kappa}-\eqref{E:u}, \eqref{D:Gk}, \eqref{D:Gu1}, \eqref{D:Gu2} of water waves under the influence of surface~tension. 

\begin{theorem}[The main result]\label{T:main}
Assume that $z_0:\R \to \mathbb{C}$ such that $z_0(\a)-\a \to 0$ as $|\a| \to \infty$ and $|z_{0\a}|=1$ everywhere on $\R$ satisfies \eqref{E:cord-arc0} for some constant $Q$. Assume that $\theta_0=\arg(z_{0\a})$ satisfies $(\a\p_\a)^{k'}\theta_0\in H^{k-k'+9/2}(\R)$ for $k\geq 0$ an integer and for all $0\leq k'\leq k$ integers. Assume in addition that $\gamma_0:\R \to \R$ satisfies $(\a\p_\a)^{k'}\gamma_0\in H^{k-k'+5/2}(\R)$ for all $0\leq k'\leq~k$ integers. 

\medskip

Then the initial value problem consisting of equations \eqref{E:kappa} and \eqref{E:u}, supplemented with equations in \eqref{D:Gk}, \eqref{D:Gu1}, \eqref{D:Gu2} $($\eqref{D:R}, \eqref{D:p}, \eqref{D:q} and \eqref{D:R3}-\eqref{D:pi}$)$, subject to the initial conditions 
\[\k(\cdot,0)=\k_0,\,\,\k_t(\cdot,0)=\k_1 \quad\text{and}\quad u(\cdot,0)=u_0,\,\, u_t(\cdot,0)=u_1,\]
where $\k_0, u_0$, $\k_1, u_1$ are in \eqref{E:k0u0} and \eqref{E:k1u1}, respectively, supports the unique solution quadruple 
\[(\k,\k_t) \in C([0,T]; H^{k+7/2}(\R) \times H^{k+2}(\R)),\,\,
(u,u_t) \in C([0,T];H^{k+5/2}(\R) \times H^{k+1}(\R))\] 
satisfying \vspace*{-.1in}
\begin{multline}\label{E:mainL2} 
\sum_{k'=0}^k\big(\|\G_2^{k'}\k\|_{H^{k-k'+7/2}}(t) +\|\G_2^{k'}\k_t\|_{H^{k-k'+2}}(t) \\
+\|\G_2^{k'}u\|_{H^{k-k'+5/2}}(t) +\|\G_2^{k'}u_t\|_{H^{k-k'+1}}(t) \big)
<C\big(t, \Phi_{k}\big)\qquad \end{multline}
at each $t \in [0,T]$, where $T>0$ depends upon $Q$ as well as
\begin{align}\label{D:M0}\Phi_{k}:=\sum_{k'=0}^k\big(&\|(\a\p_\a)^{k'}\k_0\|_{H^{k-k'+7/2}}
+\|(\a\p_\a)^{k'}\k_1\|_{H^{k-k'+2}}+\|(\a\p_\a)^{k'}\theta_0\|_{L^2} \hspace*{-.4in}\\
&+\|(\a\p_\a)^{k'}u_0\|_{H^{k-k'+5/2}}+\|(\a\p_\a)^{k'}u_1\|_{H^{k-k'+1}}
+\|(\a\p_\a)^{k'}\gamma_0\|_{L^2}\big);\hspace*{-.2in} \notag 
\end{align}
$\G_2=\frac12t\p_t+\frac13\a\p_\a$ and $C(f_1, f_2,...)$ is a positive but polynomial expression in its~arguments. Furthermore
\begin{equation}\label{E:maint}
t^k \|\lll\a\rrr^{-k} \L^{k/2+1/2}\p_\a^{k+3}\k \|_{L^2}(t)<C\big(t, \Phi_{k}\big)\quad\text{at each $t \in [0,T]$,}
\end{equation}
where $\lll\a\rrr=(1+\a^2)^{1/2}$ and $\Lambda=(-\p_\a^2)^{1/2}$; $\k_t$ and $u, u_t$ obey analogous inequalities. 
\end{theorem}

\begin{remark}[Decay rates of initial data]\label{R:weight0}\rm
If $(\a\p_\a)^{k'}f \in H^{k-k'+K}(\R)$ for $k, K \geq~0$ integers and for all $0\leq k'\leq k$ integers and if in addition $f$ as well as its derivatives decay algebraically then, necessarily, $\p_\a^{k'+K'}f(\a) \to 0$ faster than $|\a|^{-k'-1/2}$ does as $|\a|\to\infty$ for all $0\leq k'\leq k$ and all $0\leq K'\leq K$ integers.
\end{remark}

Theorem \ref{T:main} quantifies gain of regularity in solution versus localization of initial data for the water wave problem under the influence of surface tension. In light of Remark~\ref{R:weight0}, in particular, if the curvature of the initial fluid surface is contained in $H^{k+7/2}(\R)$ and if its derivatives decay faster than $\a^{-k-1}$ does as $|\a|\to \infty$, if the initial velocity of a fluid particle at the surface is likewise, then the curvature of the fluid surface corresponding to the solution of the problem at any time $t\neq0$ within the interval of existence gains $k/2$ derivatives of smoothness compared to the initial state. In brief the solution acquires $k/2$ derivatives of smoothness relative to the initial data at the expense of $k$ powers of $\a$. 

If in addition the curvature of the initial fluid surface is in $H^{k+7/2+}(\R-\{0\})$ then such $H^{k+7/2}$ singularity at the origin disappears instantaneously in the solution up to order $k+7/2+k/2$. 

\medskip

Bearing directly upon that the governing equations of the problem are dispersive, Theorem \ref{T:main} contrasts against existence theories in \cite{Yos2}, \cite{Am}, \cite{AM1}, \cite{CS1}, \cite{SZ3} among others via the energy method, which merely deliver that the solution remains as smooth as the initial data (see Theorem \ref{T:LWP}, for instance). 

Enlightening a genuine improvement in the propagation of surface water waves, furthermore, Theorem \ref{T:main} is far more physically significant than the gain of a fractional derivative which the local smoothing effect  (see \eqref{E:local-smoothing}) in \cite{CHS} or \cite{ABZ-water} offers. Besides it illustrates that the solution at any time after evolution becomes smoother than the initial data, namely the instantaneous smoothing effect; in the local smoothing effect, integration in time regularizes the solution.

The present approach works in three spatial dimensions as well, which is under scrutiny. The proof in \cite{CHS} or \cite{ABZ-water}, on the other hand, is unlikely to extend to higher dimensions. 

\medskip

Either in allowance for the effects of gravity or in case the flow depth is finite, one may regardless take the approach in Section \ref{SS:formulation} and Section \ref{SS:reformulation} to formulate the water wave problem with surface tension as the system \eqref{E:kappa1}-\eqref{E:u1}, but the symbol of the Hilbert transform is $-i\tanh(d\xi)$ in the finite-depth case, where $d$ denotes the mean fluid depth (see \cite{Yos1}, \cite{Yos2} or \cite{ABZ-water}, for instance); to compare, the symbol of the Hilbert transform is $-i\text{sgn}(\xi)$ in the infinite-depth case. In a more comprehensive description of the problem, therefore, the governing equations may lose the scaling symmetry (see \eqref{E:scaling}) of the capillary wave problem, i.e. $\We<~\infty$ and $g=0$ in \eqref{E:kappa1}-\eqref{E:u1}. Incidentally a broad class of interfacial fluids problems was formulated in \cite{SZ3}, for instance, as a second-order in time equation for the interface curvature, analogously to \eqref{E:kappa} (and \eqref{D:Gk}). 

Theorem \ref{T:main} is valid, nevertheless, at least for small\footnote{Smallness of the interface curvature must be imposed, unless proven, so that ``nonlinear" energy expressions for the problem be equivalent to Sobolev norms of the solution. In the present setting, scaling arguments guarantee it; see the discussion in Section \ref{SS:scaling}.} data because the added parameters merely behave like ``smooth" remainders. Specifically the third terms on the left sides of \eqref{E:kappa1} and \eqref{E:u1} due to the effects of gravity enjoy the same regularity as nonlinearities on the right sides; moreover $i\tanh(d\xi)\approx i\text{sgn}(\xi)$ for high frequencies\footnote{Low frequencies of the solution may be handled via the energy method; see \cite{CHS}, for instance.}. 

To conclude, surface tension acts to regularize slight disturbances of water waves from the quiescent state. 

\medskip 

Theorem \ref{T:main} necessitates that the curvature of the fluid surface be in $H^{7/2}(\R)$ or smoother and correspondingly the fluid surface in the $H^{11/2+}$ class locally in space. It is unlikely to be optimal. But we do not intend to achieve the lowest exponent, though. 

\medskip

The present strategy seems not to promote \eqref{E:maint} to the infinite gain of regularity. For, weights enter initial data in the form of $\a\p_\a$ and its powers. Hence one cannot separate them from smoothness. For a broad class of Korteweg-de Vries type partial differential equations, allowedly fully nonlinear, on the other hand, the solution was shown in \cite{CKS}, for instance, to become infinitely smooth if the initial datum as\-ymp\-tot\-i\-cal\-ly vanishes faster than polynomially and if in addition it possesses a minimal regularity.  

\medskip

Since the water wave problem is time reversible, Theorem \ref{T:main} may be used to explain singularity formation in finite time. In particular one might be able to cook up the curvature of the initial fluid surface in the $H^{5}$ class, for instance, with weight which would continuously evolve in $H^{9/2}(\R)$ but depart the weighted $H^{5}$ class in finite time. In other words, the curvature of the fluid surface would develop an $H^{5}$ singularity in finite time. The regularity condition of Theorem~\ref{T:main}, however, is too demanding to describe a blowup scenario physically realistic. 

\subsection{Estimates of $R^\k$ and $R^u$}\label{SS:R2}
As in Section \ref{SS:R1} and Lemma \ref{L:R3} let the plane curve $z(\a, t)$, $\a\in\mathbb{R}$, describe the fluid surface at time $t$ corresponding to a solution of the system \eqref{E:theta}-\eqref{E:gamma}, \eqref{E:SSD} for the water wave problem with surface tension, over the interval of time $[0,T]$ for some $T>0$. Throughout the subsection we assume that \eqref{E:cord-arc} holds during the time interval $[0,T]$ for some constant $Q$. It offers recourse to various estimates previously worked out. 

This subsection concerns routine estimates of $R^\k$ and $R^u$ (see \eqref{D:Gk} and  \eqref{D:Gu1}, \eqref{D:Gu2}, respectively) as well as their derivatives under $\p_\a$ or $\G_2$, where $t\in [0,T]$ is fixed and suppressed to simplify the exposition.

\medskip

Recall that $C_0$ means a positive generic constant and $C(f_1, f_2, ...)$ is a positive but polynomial expression in its arguments; $C(f_1, f_2, ...)$ which appears in different places in the text needs not be the same. 

\begin{proposition}[Estimates of $R^\k$ and $R^u$]\label{P:G}
For $j\geq 0$ an integer
\begin{align}
\|R^\k\|_{H^j} \leq &C(
\|\k\|_{H^{\max(j+1,3)}}, \|\k_t\|_{H^{\max(j,1)}}, \|u\|_{L^2}, \|\gamma\|_{L^2}),\label{E:Gkj} \\
\|R^u\|_{H^j} \leq &C(
\|\k\|_{H^{j+3}},\|\k_t\|_{H^{j+1}},\|u\|_{H^{j+2}},\|u_t\|_{H^{j+1}}, \|\gamma\|_{L^2}).\label{E:Guj}
\intertext{For $j\geq 0$ and $k\geq 1$ integers}
\|\p_\a^j\G_2^kR^\k\|_{L^2} \leq &C\Big(\sum_{k'=0}^k\|\G_2^{k'}\k\|_{H^{\max(j+1,3)}}, 
\sum_{k'=0}^k\|\G_2^{k'}\k_t\|_{H^{\max(j,1)}}, \sum_{k'=0}^k\|\G_2^{k'}u\|_{L^2}, M_k\Big),
\hspace*{-.5in} \label{E:Gkjk} \\
\|\p_\a^j\G_2^kR^u\|_{L^2} \leq &C\Big(\sum_{k'=0}^k\|\G_2^{k'}\k\|_{H^{j+3}}, 
\sum_{k'=0}^k\|\G_2^{k'}\k_t\|_{H^{j+1}},\label{E:Gujk}\\ 
&\hspace*{1.17in}\sum_{k'=0}^k\|\G_2^{k'}u\|_{H^{j+2}}, \sum_{k'=0}^k\|\G_2^{k'}u_t\|_{H^{j+1}},M_k \Big),
\notag \end{align}
where $M_k$ is in \eqref{D:M}.
\end{proposition}

Colloquially speaking, $R^\k$ and $R^u$ behave like $\k_{t}$, $\k_{\a}$ and $u_{\a t}$, $u_{\a\a}$ ($\k_{\a t}$, $\k_{\a\a\a}$), respectively, in the Sobolev space setting. Moreover \eqref{E:Gkj} and \eqref{E:Gkjk} do not  depend upon high Sobolev norms of $u$ or $u_t$ explicitly, as well as their derivatives under $\G_2$.

\medskip

The proof of Proposition \ref{P:G} involves understanding the smoothness of various nonlinearities and their derivatives. Rewriting \eqref{D:Gk} as 
\[R^\k=[\H,\k]\H\k_{\a\a}-[\H,q]u_{\a\a}+r^\k_{\a t}+R^{\k,r}\] 
we deduce from \eqref{E:R1j}, \eqref{E:R1jk}, \eqref{E:R2j}, \eqref{E:R2jk}, \eqref{E:q0}-\eqref{E:qjk}, \eqref{E:pj}, \eqref{E:pjk} and \eqref{E:R3j}, \eqref{E:R3jk} and from product inequalities (see Lemma \ref{L:prod}) that 
\begin{align} 
\|R^{\k,r}\|_{H^j}\leq &C( \|\k\|_{H^{j+1}},\|\k_t\|_{H^j}, \|u\|_{H^{j+1}}, \|\gamma\|_{L^2})\label{E:Rkj}
\intertext{for $j\geq 0$ an integer and}
\|\p_\a^j\G_2^kR^{\k,r}\|_{L^2}\leq &C\Big(\sum_{k'=0}^k\|\G_2^{k'}\k\|_{H^{j+1}}, 
\sum_{k'=0}^k\|\G_2^{k'}\k_t\|_{H^j}, \sum_{k'=0}^k\|\G_2^{k'}u\|_{H^{j+1}}, M_k\Big)\label{E:Rkjk} 
\end{align}
for $j\geq 0$, $k\geq 1$ integers. The proof is straightforward. Hence we omit the detail. Moreover \eqref{E:[H,a]}, or its ramification in Remark~\ref{R:ET}, and a Sobolev inequality mani\-fest that
\begin{multline}\label{E:commRk}
\|\p_\a^j\G_2^k[\H, \k]\H\k_{\a\a}\|_{L^2}, \|\p_\a^j\G_2^k[\H, q]u_{\a\a}\|_{L^2} \\
\leq C\Big(\sum_{k'=0}^k\|\G_2^{k'}\k\|_{H^{\max(j+1,2)}},\sum_{k'=0}^k\|\G_2^{k'}u\|_{H^{\max(j,2)}}\Big)
\end{multline} 
for $j\geq 0$ and $k\geq 0$ integers.

Upon inspection of \eqref{D:Gu1} and \eqref{D:Gu2}, similarly, various accomplishments in previous subsections yield after repeated use of product inequalities (see Lemma~\ref{L:prod}) that $R^u=r^u_{\a t}-\k_\a p_t-\k r^p_t+R^{u,r}$,~where
\begin{align}
\|R^{u,r}\|_{H^j}\leq &C( \|\k\|_{H^{j+3}},\|\k_t\|_{H^{j+1}},\|u\|_{H^{j+2}},\|u_t\|_{H^{j+1}},\|\gamma\|_{L^2}) \label{E:Ruj}  \\
\intertext{for $j\geq 0$ an integer and} 
\|\p_\a^j\G_2^kR^{u,r}\|_{L^2}\leq &C\Big(
\sum_{\ell=0,1}\sum_{k'=0}^{k-\ell}\|\G_2^{k'}\k\|_{H^{j+3+\ell}},
\sum_{k'=0}^k\|\G_2^{k'}\k_t\|_{H^{j+1}}, \label{E:Rujk}  \\ 
&\hspace*{1.1in} \sum_{k'=0}^k\|\G_2^{k'}u\|_{H^{j+2}}, 
\sum_{k'=0}^k\|\G_2^{k'}u_t\|_{H^{j+1}}, M_k\Big)\notag \end{align}
for $j\geq 0$, $k\geq 1$ integers. The proof, too, is straightforward. 

\medskip

It remains to understand the smoothness of $r^\k_t$, $r^u_t$, $p_t$ and $r^p_t$ as well as their derivatives under $\p_\a$ or $\G_2$. 

\begin{lemma}[Estimates of $r^\k_t$, $r^u_t$, $p_t$ and $r^p_t$]\label{L:Rt}For $j\geq 0$ an integer 
\begin{align}
\|r^\k_t\|_{H^j}&\leq C(\|\k\|_{H^{\max(j,3)}}, \|u\|_{H^{\max(j,2)}},\|\gamma\|_{L^2}), \label{E:R1tj} \\
\|r^u_t\|_{H^j}&\leq C(\|\k\|_{H^{\max(j,3)}}, \|u\|_{H^{\max(j,2)}}, \|u_t\|_{H^j},\|\gamma\|_{L^2}), 
\label{E:R2tj} \\
\|p_t\|_{H^j} &\leq C(\|\k\|_{H^{\max(j+2,3)}},\|\k_t\|_{H^{j+1}},\|u\|_{H^{\max(j+1,2)}}, \|u_t\|_{H^{j}},\|\gamma\|_{L^2}), \label{E:ptj} \\
\|r^p_t\|_{H^j} &\leq C(\|\k\|_{H^{\max(j+2,3)}},\|\k_t\|_{H^{j+1}},\|u\|_{H^{\max(j+1,2)}}, \|u_t\|_{H^{j+1}},\|\gamma\|_{L^2}).\label{E:R3tj}
\end{align}
For $j\geq 0$ and $k\geq 1$ integers
\begin{equation}\label{E:R1tjk}
\|\p_\a^j\G_2^kr^\k_t\|_{L^2} \leq C\Big( 
\sum_{\ell=0,1}\sum_{k'=0}^{k-\ell} \|\G_2^{k'}\k\|_{H^{\max(j,3)+\ell}}, 
\sum_{k'=0}^k\|\G_2^{k'}u\|_{H^{\max(j,2)}}, M_k \Big),
\end{equation}
\vspace*{-.1in}
\begin{multline}\label{E:R2tjk}
\,\,\,\|\p_\a^j\G_2^kr^u_t\|_{L^2}\leq C\Big( 
\sum_{\ell=0,1}\sum_{k'=0}^{k-\ell} \|\G_2^{k'}\k\|_{H^{\max(j,3)+\ell}}, \\
\sum_{k'=0}^k\|\G_2^{k'}u\|_{H^{\max(j,2)}},\sum_{k'=0}^k\|\G_2^{k'}u_t\|_{H^j}, M_k \Big),
\end{multline}
\vspace*{-.1in}
\begin{multline}\label{E:ptjk}
\|\p_\a^j\G_2^k p_t\|_{L^2} \leq C\Big(
\sum_{k'=0}^k \|\G_2^{k'}\k\|_{H^{\max(j+2,3)}}, \sum_{k'=0}^{k} \|\G_2^{k'}\k_t\|_{H^{j+1}}, \\ 
\sum_{k'=0}^k\|\G_2^{k'}u\|_{H^{\max(j+1,2)}},\sum_{k'=0}^k\|\G_2^{k'}u_t\|_{H^j}, M_k\Big), 
\end{multline}\vspace*{-.1in}
\begin{multline}\label{E:R3tjk}
\|\p_\a^j\G_2^k r^p_t\|_{L^2} \leq C\Big(
\sum_{k'=0}^k \|\G_2^{k'}\k\|_{H^{\max(j+2,3)}}, \sum_{k'=0}^{k} \|\G_2^{k'}\k_t\|_{H^{j+1}}, \\ 
\sum_{k'=0}^k\|\G_2^{k'}u\|_{H^{\max(j+1,2)}},\sum_{k'=0}^k\|\G_2^{k'}u_t\|_{H^{j+1}}, M_k\Big), 
\end{multline}where $M_k$ is in \eqref{D:M}. 
\end{lemma}

\begin{proof}[Proof of Proposition \ref{P:G}]
Once results of Lemma \ref{L:Rt} are at the ready, \eqref{E:Gkj} and \eqref{E:Gkjk} follow from \eqref{E:Rkj}, \eqref{E:Rkjk}, \eqref{E:commRk} and \eqref{E:R1tj}, \eqref{E:R1tjk} and from \eqref{E:uj}, \eqref{E:ujk}; \eqref{E:Guj} and \eqref{E:Gujk} follow from \eqref{E:Ruj}, \eqref{E:Rujk} and \eqref{E:R2tj}, \eqref{E:R2tjk}, \eqref{E:ptj}, \eqref{E:ptjk}, \eqref{E:R3tj}, \eqref{E:R3tjk}.\end{proof}

\begin{proof}[Proof of Lemma \ref{L:Rt}]{\bf The proof of \eqref{E:R1tj} and \eqref{E:R1tjk}} entails computing the time derivatives of $\B$ and the commutator of the Hilbert transform. Indeed we differentiate the former equation in \eqref{D:R} and write in light of the former in \eqref{E:SSD}~that
\begin{equation}\label{E:R1t}
r^\k_t=-\H(\mathbf{m}_t\cdot z_\a)-\theta_t\H(\mathbf{m}\cdot iz_\a)
+\mathbf{m}_t\cdot iz_\a-\theta_t \mathbf{m}\cdot z_\a,\end{equation}
where $\mathbf{m}$ is in \eqref{D:V}. Moreover 
\begin{equation}\label{E:WHt} 
\p_t\mathcal{W}f=\mathcal{W}f_t+[\p_t, \mathcal{W}]f \quad\text{and}\quad 
\p_t[\H, a]f=[\H,a]f_t+[\H,a_t]f.\end{equation}

\smallskip 

We manipulate among \eqref{E:WHt}, \eqref{E:commj}, the first two formulae in \eqref{E:commB} and integration by parts, as in the proof of \eqref{E:Bj}, to show that $\p_\a^j\p_t\B f$ and $\p_\a^j\p_t[\H,\frac{1}{z_\a^2}]f$, $j\geq 1$ an integer, are both made up of integral operators either in \eqref{D:T} or in \eqref{D:T'}, where $A(x)=\frac{1}{x^d}$ for $d\geq 1$ an integer and $a_{n'}=z$ for all indices, allowing that $\p_\a^{j_1}\p_\b^{j_2}Qz$ or $\p_\a^{j_1}\p_\b^{j_2}Qz_t$ substitutes $Qz$, where $j_1$, $j_2\geq 0$ are integers such that $j_1+j_2\leq j$, and allowing that $\frac{f}{z_\a}$ or $\big(\frac{f}{z_\a}\big)_t$ is in place of $f$. We then handle each of the resulting integral operators by means of \eqref{E:T}, or its ramification in Remark~\ref{R:ET}, along with \eqref{E:Q}, the first two estimates in \eqref{E:z_a}, and \eqref{E:Qztj} to obtain that
\begin{equation}\label{E:Btj} 
\|\p_\a^j\p_t\B f\|_{L^2}, \|\p_\a^j\p_t[\H, \tfrac{1}{z_\a^2}]f\|_{L^2} 
\leq C(\|\k\|_{H^{j}}, \|u\|_{H^j}, \|\gamma\|_{L^2})(\|f\|_{H^1}+\|f_t\|_{H^1}).\hspace*{-.1in}\end{equation}
Moreover $\big\|\p_t^\ell\big(\gamma_\a-\frac{\gamma z_{\a\a}}{z_\a}\big)\big\|_{H^1}\leq C(\|\k\|_{H^3}, \|u\|_{H^2}, \|\gamma\|_{L^2})$ for $\ell=0$ or $1$ using \eqref{E:gammaj}, \eqref{E:gamma_tj}, the first two estimates in \eqref{E:z_a}, and \eqref{E:Qztj}. For $j=0$ we resort to Young's inequality and the first formula in \eqref{E:commH} to deal with the $L^2$-norm of $\p_t[\H, \frac{1}{z_\a^2}]\frac{\gamma z_{\a\a}}{z_\a}$; otherwise we apply \eqref{E:T} and the second formula in \eqref{E:commB} to the time derivatives of integral operators on the right side of \eqref{D:V} in like manner. 

Upon inspection of \eqref{E:R1t}, therefore, \eqref{E:R1tj} follows collectively and from \eqref{E:R1j'} after repeated use of  the first two estimates in \eqref{E:z_a}, \eqref{E:Qztj} and product inequalities (see Lemma~\ref{L:prod}). The proof closely resembles that of \eqref{E:R1j}. Hence we omit the detail. Instead refer to \cite[Appendix B]{CHS}, for instance, for some detail. 

\medskip

Similarly we exploit \eqref{E:WHt}, \eqref{E:commj}, formulae in \eqref{E:commB} and in \eqref{E:commH}, and integration by parts, as in the proof of \eqref{E:Bjk}, to write $\p_\a^j\G_2^k\p_t\B f$ and $\p_\a^j\G_2^k\p_t[\H,\frac{1}{z_\a^2}]f$, $j\geq 1$ and $k\geq 1$ integers, both as sums of integral operators in \eqref{D:T} or \eqref{D:T'}, where $A(x)=\frac{1}{x^d}$ for $d\geq 1$ an integer and $a_{n'}=z$ for all indices, allowing that $\p_\a^{j_1}\p_\a^{j_2}\Gamma_{2\a}^{k_1}\Gamma_{2\b}^{k_2}Qz$ or $\p_\a^{j_1}\p_\a^{j_2}\Gamma_{2\a}^{k_1}\Gamma_{2\b}^{k_2}Qz_t$ replaces $Qz$, where $j_1, j_2$, $k_1, k_2\geq 0$ are integers such that $j_1+j_2\leq j$ and $k_1+k_2\leq k$ (recall the notation that $\G_{2\a}=\frac12t\p_t+\frac13\a\p_\a$ and $\G_{2\b}=\frac12t\p_t+\frac13\b\p_\b$), and possibly that $\G_2^{k'}(\frac{f}{z_\a})$ or $\G_2^{k'}(\frac{f}{z_\a})_t$ is in lieu of $f$ for $0\leq k'\leq k$ an integer. It then follows by virtue of \eqref{E:T}, or the ramification in Remark~\ref{R:ET}, and by \eqref{E:Q}, estimates in \eqref{E:z_a}, and \eqref{E:thetajk}, \eqref{E:Qztjk} that 
\begin{multline}\label{E:Btjk}
\| \partial_\a^j\G_2^k \p_t\B f\|_{L^2}, \|\partial_\a^j\G_2^k \p_t[\H, \tfrac{1}{z_\a^2}]f\|_{L^2}  \\
\leq C\Big(\sum_{\ell=0,1}\sum_{k'=0}^{k-\ell} \|\G_2^{k'}\k\|_{H^{\max(j,1)+\ell}}, 
\sum_{k'=0}^k \|\G_2^{k'}u\|_{H^{\max(j,1)}}, M_k\Big)\\
\cdot\Big(\sum_{k'=0}^k \|\G_2^{k'}f\|_{H^1}+\sum_{k'=0}^k \|\G_2^{k'}f_t\|_{H^1}\Big),
\end{multline}
where $M_k$ is in \eqref{D:M}. Moreover  
\[ \sum_{k'=0}^k \Big\|\G_2^{k'}\p_t^\ell\Big(\gamma_\a-\frac{\gamma z_{\a\a}}{z_\a}\Big)\Big\|_{H^1} \leq C\Big(\sum_{k'=0}^k \|\G_2^{k'}\k\|_{H^3}, \sum_{k'=0}^k\|\G_2^{k'}u\|_{H^2}, M_k\Big)\]
for $k\geq 1$ an integer and for $\ell=0$ or $1$ using the first two formulae in \eqref{E:commLj}, to inter\-change either $\p_\a$ or $\p_t$ with $\G_2$ up to a sum of smooth remainders, and using \eqref{E:gammajk}, \eqref{E:gamma_tjk}, \eqref{E:thetajk}, \eqref{E:Qztjk}. For $j=0$ we appeal to \eqref{E:T}, Young's inequality and the second formula in \eqref{E:commB}, the first in \eqref{E:commH} to deal with the time derivatives of integral operators on the right side of \eqref{D:V} and their derivatives under $\G_2$. 

Consequently \eqref{E:R1tjk} follows together and from \eqref{E:R1jk'} after numerous applications of \eqref{E:z_a}, \eqref{E:Qztjk} and product inequalities (see Lemma \ref{L:prod}). The proof closely resembles that of \eqref{E:R1tj} or \eqref{E:R1jk}. Hence we leave out the detail. 

\subsection*{Estimates of $r^u_t$} Upon inspection of the time derivative of the latter equation in \eqref{D:R}, obviously, \eqref{E:R2tj} and \eqref{E:R2tjk} follow from \eqref{E:R1j}, \eqref{E:R1tj} and from \eqref{E:R1jk}, \eqref{E:R1tjk}, respectively.

\subsection*{Estimates of $p_t$}Examining term\-wise the time derivative of the right side of \eqref{E:p'} (and \eqref{E:W_t}, \eqref{D:W12}, \eqref{D:rr}) as well as its derivatives under $\p_\a$ or $\G_2$ by means of various results and arguments  worked out in previous subsections and by \eqref{E:R1tj}, \eqref{E:R1tjk} we learn that we have yet to understand the smoothness of $q_t$, $\theta_{tt}$ and $Qz_{tt}$, $\gamma_{tt}$ as well as their derivatives under $\p_\a$ or $\G_2$. 

\medskip

We shall discuss $\G_2^kq_t$, $k\geq 0$ an integer, in the $L^\infty$-space setting, as in the proof of \eqref{E:q0} and \eqref{E:qk}. For, $\V_t$ on the right side of the time derivative of \eqref{D:q} seems unwieldy in the $L^2$-space setting (see the latter equation in \eqref{E:SSD}) although
\begin{align}\label{E:Utk}
\|\G_2^k\V_t\|_{L^\infty}\leq 
\int^\infty_{-\infty}\Big|\Big(\G_2+\frac13\Big)^k&\p_t(\theta_\a \W\cdot iz_\a)\Big|~d\a \\
\leq &C\Big(\sum_{k'=0}^k\|\G_2^{k'}\k\|_{H^1},\sum_{k'=0}^k\|\G_2^{k'}u\|_{H^1}, M_k\Big),\notag
\end{align} 
where $M_k$ is in \eqref{D:M}. The first inequality utilizes the latter equation in \eqref{E:SSD} and the last formula in \eqref{E:commLj}, while the second inequality relies upon the first and the third estimates in \eqref{E:z_a}, \eqref{E:Qztjk} and \eqref{E:Wjk}, \eqref{E:Wtjk}. Upon inspection of the time derivative of \eqref{D:q} we then deduce from \eqref{E:z_a}, \eqref{E:Qztjk}, \eqref{E:gammajk}, \eqref{E:gamma_tjk}, \eqref{E:Uk}, \eqref{E:Utk}, \eqref{E:Wtj}, \eqref{E:Wtjk} and from product inequalities (Lemma~\ref{L:prod}), a Sobolev inequality~that 
\begin{align}
\|\G_2^kq_t\|_{L^\infty} 
\leq &C\Big(\sum_{k'=0}^k\|\G_2^{k'}\k\|_{H^2},\sum_{k'=0}^k\|\G_2^{k'}u\|_{H^1}, M_k\Big)
\quad\text{for $k\geq 0$ an integer,}\hspace*{-.1in}\label{E:qtk} \\
\intertext{ where $M_k$ is in \eqref{D:M}. The proof closely resembles that of \eqref{E:q0} and \eqref{E:qk}. Hence we omit the detail. Moreover the latter equation in \eqref{D:u} and the first formula in \eqref{E:commLj} manifest that}
\|\p_\a^j\G_2^kq_t\|_{L^2}\leq &C_0\sum_{k'=0}^k\|\G_2^{k'}u_t\|_{H^{j-1}}
\quad \text{for $j\geq 1$ and $k\geq 0$ integers.}\label{E:qtjk}\end{align}

To proceed, \eqref{E:theta'} trades Sobolev norms of $\theta_{tt}$ and its derivatives under $\G_2$ with those of $\k$, $\k_t$, $u_t$,  $r^\k_t$ and $q$, $q_t$ in the $L^\infty$-space setting. Further \eqref{E:R1tj}, \eqref{E:R1tjk} and \eqref{E:q0}-\eqref{E:qjk}, \eqref{E:qtk}, \eqref{E:qtjk} render them in terms of Sobolev norms of $\k$, $\k_t$, $u$, $u_t$ and the $L^2$-norms of $\theta$, $\gamma$, as well as their derivatives under $\G_2$. The $L^2$-norm of $\p_\a^j\G_2^kQz_{tt}$, $j\geq 0$ and $k\geq 0$ integers, is bounded likewise by virtue of the former equation in \eqref{E:SSD}, \eqref{E:Q}, \eqref{E:Qztj}, \eqref{E:Qztjk} and by the composition inequality (see Lemma \ref{L:comp}), product inequalities (see Lemma \ref{L:prod}). Therefore
\begin{equation}\label{E:Qzttj}
\|\theta_{tt}\|_{H^j}, \|Qz_{tt}\|_{H^j} \leq C(\|\k\|_{H^{\max(j,3)}}, \|\k_t\|_{H^j},
\|u\|_{H^{\max(j,2)}}, \|u_t\|_{H^j},\|\gamma\|_{L^2}) \end{equation}
for $j\geq 0$ an integer and 
\begin{align}\label{E:Qzttjk} 
\|\p_\a^j\G_2^k\theta_{tt}\|_{L^2}, \hspace*{-.02in}
\|\p_\a^j\G_2^kQz_{tt}\|_{L^2}\hspace*{-.02in}\leq\hspace*{-.03in} C\Big(&
\sum_{k'=0}^k\|\G_2^{k'}\k\|_{H^{\max(j,3)}},\sum_{k'=0}^k\|\G_2^{k'}\k_t\|_{H^j}, \\ &
\sum_{k'=0}^k\|\G_2^{k'}u\|_{H^{\max(j,2)}}, \sum_{k'=0}^k\|\G_2^{k'}u_t\|_{H^j}, M_k\Big) \notag
\end{align}
for $j\geq 0$, $k\geq 1$ integers, where $M_k$ is in \eqref{D:M}. The proof is similar to that of \eqref{E:Qztj} and \eqref{E:Qztjk}. Hence we omit the detail. 

In view of the latter equation in \eqref{D:W12} we consequently infer from \eqref{E:gammaj}, \eqref{E:gammajk}, \eqref{E:gamma_tj}, \eqref{E:gamma_tjk}, \eqref{E:z_a}, \eqref{E:thetajk}, \eqref{E:Qztj}, \eqref{E:Qztjk} and \eqref{E:Btj}, \eqref{E:Btjk}, \eqref{E:Qzttj}, \eqref{E:Qzttjk} and from product inequalities (see Lemma~\ref{L:prod}) that the $L^2$-norm of $\p_\a^j\G_2^k\mathbf{T}_{2t}$, $j\geq0$ and $k\geq 0$ integers, is bounded by the right side of either \eqref{E:Qzttj} or \eqref{E:Qzttjk}. In view of \eqref{D:rr} and  formulae in \eqref{E:commR}, moreover, we run the argument leading to \eqref{E:Btj} and \eqref{E:Btjk} to bound likewise the $L^2$-norm of $\p_\a^j\G_2^k\p_t\Rr\gamma$, $j\geq0$ and $k\geq 0$ integers. The proof relies upon \eqref{E:gammaj}, \eqref{E:gammajk}, \eqref{E:gamma_tj}, \eqref{E:gamma_tjk}, \eqref{E:z_a}, \eqref{E:thetajk}, \eqref{E:Qztj}, \eqref{E:Qztjk} and \eqref{E:Qzttj}, \eqref{E:Qzttjk} and upon product inequalities (see Lemma~\ref{L:prod}). It is straightforward. Hence we leave out the detail. 

\medskip

Differentiating \eqref{E:gamma'} in time we arrange the result as   
\begin{multline*}
(1+2\J)\gamma_{tt}=-\text{Re} 
\Big( z_\a[\p_t,\B]\gamma_t +\frac{z_\a}{2i}\Big[\H, \Big(\frac{1}{z_\a}\Big)_t\Big]\gamma_t  
+ z_{\a t}\B\gamma_t +\frac{z_{\a t}}{2i}\Big[\H,\frac{1}{z_\a}\Big]\gamma_t\Big)\qquad\\
+\p_t\Big(2\theta_{\a\a}+(\V-\W\cdot z_\a)\gamma_\a-\gamma \W_\a\cdot z_\a \\
-\H(\gamma \theta_t) -2\mathbf{b}\cdot z_\a-\frac12\gamma\gamma_\a+2(\V-\W\cdot z_\a)\W_\a \cdot z_\a\Big),\end{multline*}
where $\J$ and $\mathbf{b}$ are in \eqref{D:J}. We then employ various results and arguments previously worked out and we make repeated use of product inequalities (see Lemma~\ref{L:prod}) to bound the $L^2$-norm of $\p_\a^j\G_2^k(1+2\J)\gamma_{tt}$, $j\geq 0$ and $k\geq 0$ integers, by Sobolev norms of $\k$, $\k_t$, $u$, $u_t$ and the $L^2$-norms of $\theta$, $\gamma$, as well as their derivatives under~$\G_2$. The proof relies upon \eqref{E:Bj}, \eqref{E:Bjk}, \eqref{E:z_a}, \eqref{E:thetajk}, \eqref{E:gammaj}, \eqref{E:gammajk}, \eqref{E:Wj}, \eqref{E:Wjk}, \eqref{E:q0}-\eqref{E:qjk}, \eqref{E:Qztj}, \eqref{E:Qztjk}, \eqref{E:W2j}, \eqref{E:W2jk}, \eqref{E:gamma_tj}, \eqref{E:gamma_tjk}, \eqref{E:Wtj}, \eqref{E:Wtjk} and \eqref{E:Btj}, \eqref{E:Btjk}, \eqref{E:qtk}, \eqref{E:qtjk}, \eqref{E:Qzttj}, \eqref{E:Qzttjk} and upon the arguments leading to \eqref{E:Wj}, \eqref{E:Wjk} and \eqref{E:Btj}, \eqref{E:Btjk}. It is straightforward. Hence we omit the detail.

The argument leading to \eqref{E:Wj} and \eqref{E:Wjk} moreover furnishes that 
\begin{align*}
\|\J\gamma_{tt}\|_{H^j}\leq &C(\|\k\|_{H^{j-1}})\|\gamma_{tt}\|_{H^1}\quad\text{for $j\geq 0$ an integer,}\\  
\|\p_\a^j\G_2^k\J \gamma_{tt}\|_{L^2}\leq &C\Big( \sum_{\ell=0,1}\sum_{k'=0}^{k-\ell} 
\|\G_2^{k'}\k\|_{H^{\max(j-1,1)+\ell}}, \sum_{k'=0}^k\|\G_2^{k'}\theta\|_{L^2}\Big)
\Big(  \sum_{k'=0}^k\|\G_2^{k'}\gamma_{tt}\|_{H^1}\Big)
\end{align*}
for $j\geq 0$, $k\geq 1$ integers. 

Since $\p_\a^j\G_2^k\gamma_{tt}=\p_\a^j\G_2^k(1+2\J)\gamma_{tt}-2\p_\a^j\G_2^k\J\gamma_{tt}$, therefore, it follows collectively after an application of Young's inequality with~$\epsilon$ (see Lemma \ref{L:Young-e}) that 
\begin{align}
\|\gamma_{tt}\|_{H^j}\leq &C(\|\k\|_{H^{\max(j+2,3)}}, \|\k_t\|_{H^{j+1}}, 
\|u\|_{H^{\max(j+1,2)}}, \|u_t\|_{H^j}, \|\gamma\|_{L^2})\label{E:gammattj}
\intertext{for $j\geq 0$ an integer and}
\|\p_\a^j\G_2^k\gamma_{tt}\|_{L^2} \leq &C\Big(\sum_{k'=0}^k\|\G_2^{k'}\k\|_{H^{\max(j+2,3)}}, 
\sum_{k'=0}^k\|\G_2^{k'}\k_t\|_{H^{j+1}}, \label{E:gammattjk} \\ 
&\hspace*{.8in}\sum_{k'=0}^k\|\G_2^{k'}u\|_{H^{\max(j+1,2)}}, 
\sum_{k'=0}^k\|\G_2^{k'}u_t\|_{H^j}, M_k\Big)\notag \end{align}
for $j\geq 0, k\geq 1$ integers, where $M_k$ is in \eqref{D:M}. The proof is similar to that of Lem\-ma \ref{L:gamma_t}. Hence we leave out the detail. Instead we refer the reader to \cite[Appendix B]{CHS}, for instance, for some detail relevant to \eqref{E:gammattj}.

\medskip 

In view of the former equation in \eqref{D:W12} and \eqref{E:W_t} we consequently infer from \eqref{E:z_a}, \eqref{E:thetajk}, \eqref{E:Qztj}, \eqref{E:Qztjk}, \eqref{E:gamma_tj}, \eqref{E:gamma_tjk} and \eqref{E:Btj}, \eqref{E:Btjk}, \eqref{E:Qzttj}, \eqref{E:Qzttjk}, \eqref{E:gammattj}, \eqref{E:gammattjk} and from product inequalities (see Lemma \ref{L:prod}) that the $L^2$-norms of $\p_\a^j\G_2^k\mathbf{T}_{1t}$ and $\p_\a^j\G_2^k\mathbf{W}_{tt}$, $j\geq 0$ and $k\geq 0$ integers, are bounded by the right side of either \eqref{E:gammattj} or \eqref{E:gammattjk}. 

In view of \eqref{E:p'}, furthermore, \eqref{E:ptj} and \eqref{E:ptjk} follow from \eqref{E:R1j}, \eqref{E:R1jk}, \eqref{E:z_a}, \eqref{E:thetajk}, \eqref{E:gammaj}, \eqref{E:gammajk}, \eqref{E:q0}-\eqref{E:qjk},  \eqref{E:Qztj}, \eqref{E:Qztjk}, \eqref{E:Wtj}, \eqref{E:Wtjk}, \eqref{E:gamma_tj}, \eqref{E:gamma_tjk} and from \eqref{E:R1tj}, \eqref{E:R1tjk}, \eqref{E:qtk}, \eqref{E:qtjk}, \eqref{E:Qzttj}, \eqref{E:Qzttjk}, \eqref{E:gammattj}, \eqref{E:gammattjk} after repeated use ofs product inequalities (see Lem\-ma~\ref{L:prod}).

\subsection*{Estimates of $r^p_t$}We employ various results and arguments previously worked out and we make repeated use of the product inequalities (see Lemma \ref{L:prod}) to bound the $L^2$-norms of the time derivative of the right side of \eqref{D:R3} (and \eqref{D:Y}, \eqref{D:pi}) as well as derivatives under $\p_\a$ or $\G_2$ by the right side of either \eqref{E:R3tj} or \eqref{E:R3tjk}. The proof relies upon \eqref{E:R1j}, \eqref{E:R1jk}, \eqref{E:Bj}, \eqref{E:Bjk}, \eqref{E:z_a}, \eqref{E:thetajk}, \eqref{E:gammaj}, \eqref{E:gammajk}, \eqref{E:Wj}, \eqref{E:Wjk}, \eqref{E:q0}-\eqref{E:qjk}, \eqref{E:Qztj}, \eqref{E:Qztjk}, \eqref{E:W2j}, \eqref{E:W2jk}, \eqref{E:gamma_tj}, \eqref{E:gamma_tjk}, \eqref{E:Wtj}, \eqref{E:Wtjk}, \eqref{E:pj}, \eqref{E:pjk} and \eqref{E:R1tj}, \eqref{E:R1tjk}, \eqref{E:qtk}, \eqref{E:qtjk}, \eqref{E:Qzttj}, \eqref{E:Qzttjk}, \eqref{E:gammattj}, \eqref{E:gammattjk}, \eqref{E:ptj}, \eqref{E:ptjk} and upon the arguments leading to  \eqref{E:Wj}, \eqref{E:Wjk} and \eqref{E:Btj}, \eqref{E:Btjk}. It is straightforward. Hence we omit the detail.
\end{proof}

Concluding the subsection we discuss unrefined estimates of nonlinearities. 

\begin{corollary}[Estimates of $G^\k$ and $G^u$]\label{C:g}
Upon writing \begin{equation}\label{E:system0'}
\p_t^2\k-\H \partial_\a^3\k=G^{\k,q} \quad\text{and}\quad
\p_t^2u-\H \partial_\a^3u=G^{u,q}, \end{equation}
\begin{align}
\|\p_\a^j\G_2^kG^{\k,q}\|_{L^2} \leq &C\Big(\sum_{k'=0}^k\|\G_2^{k'}\k\|_{H^{\max(j+2,3)}}, 
\sum_{k'=0}^k\|\G_2^{k'}\k_t\|_{H^{j+1}}, \sum_{k'=0}^k\|\G_2^{k'}u\|_{L^2}, M_k\Big), \hspace*{-.4in}\label{E:gkjk} \\
\|\p_\a^j\G_2^kG^{u,q}\|_{L^2} \leq &C\Big(\sum_{k'=0}^k\|\G_2^{k'}\k\|_{H^{\max(j+2,3)}}, 
\sum_{k'=0}^k\|\G_2^{k'}\k_t\|_{H^j}, \label{E:gujk} \\ &\hspace*{1.1in}
\sum_{k'=0}^k\|\G_2^{k'}u\|_{H^{j+2}}, \sum_{k'=0}^k\|\G_2^{k'}u_t\|_{H^{j+1}}, M_k\Big) \notag
\end{align}
for $j\geq 0$ and $k\geq 0$ integers, where $M_k$ is in \eqref{D:M}.
\end{corollary}

Colloquially speaking, $G^{\k,q}$ and $G^{u,q}$ behave like $\k_{\a t}$, $\k_{\a\a}$ and $u_{\a t}$, $u_{\a\a}$, respectively.

\begin{proof}An explicit calculation reveals that  
\[G^{\k,q}=G^\k-\p_t(q\k_\a)-q\p_\a(\conv)\k\,\text{ and }\,
G^{u,q}=G^u-\p_t(qu_\a)-q\p_\a(\conv)u,\]
where $G^\k$ and $G^u$ are in \eqref{D:gk} and \eqref{D:gu}, respectively. Therefore \eqref{E:gkjk} follows from \eqref{E:R1j}, \eqref{E:R1jk},  \eqref{E:R2j}, \eqref{E:R2jk}, \eqref{E:q0}-\eqref{E:qjk}, \eqref{E:pj}, \eqref{E:pjk} and \eqref{E:commRk}, \eqref{E:R1tj}, \eqref{E:R1tjk} after repeated use of product inequalities (see Lemma \ref{L:prod}); \eqref{E:gujk} follows from \eqref{E:R1j}, \eqref{E:R1jk}, \eqref{E:R2j}, \eqref{E:R2jk}, \eqref{E:q0}-\eqref{E:qjk}, \eqref{E:pj}, \eqref{E:pjk} and \eqref{E:R2tj}, \eqref{E:R2tjk}, \eqref{E:ptj}, \eqref{E:ptjk} in like manner.  \end{proof}

\newpage

\section{Gain of regularity}\label{S:proof}
Gain of regularities is demonstrated for water waves under the influence of surface tension. Energy estimates are set out for the system \eqref{E:kappa}-\eqref{E:u}, \eqref{D:Gk}, \eqref{D:Gu1}, \eqref{D:Gu2} under $\p_\a$ as well as $\G_2$. Estimates of weighted Sobolev norms are established.

\subsection{Energy estimates under $\p_\a$}\label{SS:scaling}
This subsection concerns preliminary energy estimates for \eqref{E:kappa} (and \eqref{D:Gk}) and the system \eqref{E:kappa}-\eqref{E:u}, \eqref{D:Gk}, \eqref{D:Gu1}, \eqref{D:Gu2}, which are equivalent to the vortex sheet formulation \eqref{E:theta}-\eqref{E:gamma}, \eqref{E:SSD} of the water wave problem with surface tension, as deliberated in Section \ref{SS:formulation} and Section \ref{SS:reformulation}. Drawing upon a well-posedness proof, the result is of independent interest. 

\medskip

Let the plane curve $z(\a, t)$, $\a \in \R$, describe the fluid surface at time $t$ corresponding to a solution of the system \eqref{E:theta}-\eqref{E:gamma}, \eqref{E:SSD} over the interval of time $[0,T]$ for some $T>0$. In the proposition below we assume that \eqref{E:cord-arc} holds during the time interval $[0,T]$ for some constant $Q$. It offers recourse to various estimates in Section~\ref{SS:R1} and Lemma \ref{L:R3}, Section \ref{SS:R2}. 

Recall that $C_0$, $C_1, C_2, ...$ mean positive constants and $C(f_1,f_2, ...)$ is a positive but polynomial expression in its arguments; $C_0, C_1, C_2, ...$ and $C(f_1,f_2, ...)$ which appear in different places in the text need not be the same. Recall that $\L=\H\p_\a$ and $\|f\|_{H^{1/2}}^2=\int(f^2+f\L f)$.

\medskip

Following \cite{CHS} let's write \eqref{E:kappa} as the system of first-order in time equations
\begin{equation}\label{E:ky}
\p_t \k=\varphi-q\p_\a\k\quad\text{and}\quad 
\p_t\varphi=\H\p_\a^3\k+2\k\p_\a^2\k-q\p_\a\varphi+R^\k(\k, u;\theta,\gamma),\end{equation} 
where \eqref{E:Gkj} implies in light of \eqref{E:q0}, \eqref{E:qjk} and \eqref{E:uj} that
\begin{equation}\label{E:Gkj'}
\|R^\k\|_{H^j}(t) \leq C(\|\k\|_{H^{\max(j+1,3)}}(t), \|\varphi\|_{H^{\max(j,1)}}(t), \|u\|_{L^2}(t),\|\gamma\|_{L^2}(t))\end{equation}
at each $t \in [0,T]$ for $j\geq 0$ an integer. Indeed
\begin{align}\label{E:kt-yj}
\|\k_t\|_{H^j}(t)\leq &\|\pp\|_{H^j}(t)+\|q\k_\a\|_{H^j}(t) \\
\leq &\|\pp\|_{H^j}(t)+C_0(\|q\|_{L^\infty}(t)+\|u\|_{H^j}(t))\|\k\|_{H^{j+1}}(t) \notag \\
\leq &C(\|\k\|_{H^{j+1}}(t), \|\pp\|_{H^j}(t), \|u\|_{L^2}(t), \|\gamma\|_{L^2}(t)) \notag
\end{align}
and vice versa $\|\phi\|_{H^j}(t) \leq C(\|\k\|_{H^{j+1}}(t), \|\k_t\|_{H^j}(t), \|u\|_{L^2}(t), \|\gamma\|_{L^2}(t))$. The first inequality in \eqref{E:kt-yj} utilizes the former equation in \eqref{E:ky}. The second inequality relies upon \eqref{E:q0}, \eqref{E:qjk} and product inequalities (see Lemma \ref{L:prod}). The last inequality relies upon \eqref{E:uj} and Young's inequality with~$\epsilon$ (see Lemma \ref{L:Young-e}).

Let's define the energy norm for \eqref{E:ky} (and \eqref{D:Gk}) of order $k$, $k\geq 0$ an integer,~as 
\begin{equation}\label{D:E1k} 
E^0_k(t)=\sum_{j=0}^k e^0_j(t)+\|u\|_{L^2}^2(t)+\|\gamma\|_{L^2}^2(t),\end{equation}
where $e^0_0(t)=\|\k\|_{L^2}^2(t)+\|\varphi\|_{L^2}^2(t)$ and
\begin{equation}\label{D:e1j}
e^0_j(t)=\frac12 \int^\infty_{-\infty} \big(\p_\a^{j+1}\k\L\p_\a^{j+1}\k+(\p_\a^{j}\pp)^2
+2\k(\p_\a^{j+1}\k)^2\big)(\a,t)~d\a
\end{equation}for $j\geq 1$ an integer. An explicit calculation reveals that the system \eqref{E:ky}, \eqref{D:Gk} enjoys the scaling symmetry under
\[ \k(\a,t) \mapsto \l\k(\l\a, \l^{3/2}t),\qquad \pp(\a,t) \mapsto \l^{5/2}\pp(\l\a, \l^{3/2}t)\]
and correspondingly $z(\a,t) \mapsto \l^{-1}z(\l\a, \l^{3/2}t)$, $\gamma(\a,t) \mapsto \l^{1/2}\gamma(\l\a, \l^{3/2}t)$ for any $\lambda>0$ (see \eqref{E:scaling}). Therefore we may assume that $\|\k\|_{L^\infty}(t)$ is sufficiently small during the time interval $[0,T]$ so that $E^0_k(t)$ is nonnegative and accordingly it is equivalent to $\|\k\|_{H^{k+3/2}}^2(t)+\|\pp\|_{H^k}^2(t)+\|u\|_{L^2}^2(t)+\|\gamma\|_{L^2}^2(t)$ at each $t\in [0,T]$ for $k\geq 1$ an~integer. In light of \eqref{E:kt-yj}, furthermore,
\begin{gather}
\|\k\|_{H^{k+3/2}}(t)+\|\k_t\|_{H^{k}}(t)+\|u\|_{L^2}(t)+\|\gamma\|_{L^2}(t)\leq C(E^0_k(t)),\label{E:k-E} \\
E^0_k(t)\leq C(\|\k\|_{H^{k+3/2}}(t), \|\k_t\|_{H^{k}}(t), \|u\|_{L^2}(t), \|\gamma\|_{L^2}(t))\label{E:E-k}
\end{gather}
at each $t \in [0,T]$ for $k\geq 1$ an integer. That is to say, $E^0_k(t)$ amounts to $\|\k\|_{H^{k+3/2}}(t)+\|\k_t\|_{H^{k}}(t)$. In a more comprehensive description of the problem, e.g. in allowance for the effects of gravity, incidentally, the scaling symmetry may be lost and smallness of the interface curvature must be imposed (see the discussion in Section~\ref{SS:result}).

\begin{proposition}[Energy estimates under $\p_\a$]\label{P:energy1}
If $\k\in H^{k+3/2}(\R), \pp\in H^{k}(\R)$ and $u\in L^2(\R), \gamma \in L^2(\R)$, $k\geq 2$ an integer, solve the system consisting of equations in \eqref{E:ky} and supplemented with \eqref{D:Gk} over the interval of time $[0,T]$ then 
\begin{equation}\label{E:E1k} 
E^0_k (t) \leq -\frac{1}{C_2}\log\left(\exp(-C_2E^0_k(0))-C_1C_2t\right)
\end{equation}
at each $t\in [0,T_*]$, where $T_*$ in the range $(0,T]$ depends upon $E^0_k(0)$. Furthermore 
\begin{equation}\label{E:energyk}
\|\k\|_{H^{k+3/2}}(t)+\|\k_t\|_{H^{k}}(t)
\leq C(t, \|\k\|_{H^{k+3/2}}(0), \|\k_t\|_{H^{k}}(0), \|u\|_{L^2}(0), \|\gamma\|_{L^2}(0)) \hspace*{-.2in}
\end{equation}at each $t \in [0,T_*]$.
\end{proposition}

\begin{remark}[Energy expressions are ``nonlinear"]\label{R:energy}\rm
In view of the linear part of \eqref{E:kappa}, perhaps one takes
\begin{equation}\label{E:badE}
\frac12\int^\infty_{-\infty} \big(\p_\a^{j+1}\k\L\p_\a^{j+1}\k+(\p_\a^{j}\k_t)^2\big)(\a,t)~d\a
\end{equation}
for $e^0_j$, instead of \eqref{D:e1j}. But the multi-derivative nonlinearities $q^2\k_{\a\a}$ and $2\k\k_{\a\a}$ on the left side of \eqref{E:kappa} are unwieldy in the course of the associated energy method. Notably, differentiating \eqref{E:badE} in time and substituting $\k_{tt}$ by \eqref{E:kappa} one arrives at an expression containing 
\[-\int q^2\p_\a^{j+2}\k\, \p_\a^j\k_t~d\a+\int 2\k\p_\a^{j+2}\k\, \p_\a^j\k_t~d\a,\] which cannot be controlled by \eqref{E:badE}. 

To overcome the setback and attain energy estimates we introduce extra terms into \eqref{E:badE}, possibly cubic or higher order but smoother than $\L^{1/2}\p_\a^{j+1}\k$ and $\p_\a^j\k_t$, which offset the awkward nonlinearities, e.g. $- \frac12\int q^2(\p_\a^{j+1}\k)^2~d\a+\int \k(\p_\a^{j+1}\k)^2~d\a$. Energy expressions therefore become ``nonlinear"..  
\end{remark}

\begin{proof}[Proof of Proposition \ref{P:energy1}]
For $j\geq 1$ an integer we differentiate \eqref{D:e1j} in time and split the integral as 
\begin{align}\frac{d}{dt}e^0_j=&
\int \p_\a^{j+1}\k_t\,\L\p_\a^{j+1}\k~d\a+\int \p_\a^{j}\pp_t\,\p_\a^{j}\pp~d\a\label{E:e1j} \\
&+\int \k_t(\p_\a^{j+1}\k)^2~d\a+2\int \k\p_\a^{j+1}\k_t\,\p_\a^{j+1}\k~d\a
=:I_{1}^0+I_{2}^0+I^0_3+I^0_4.\notag \end{align}
We then use the former equation in \eqref{E:ky} to obtain after integration by parts that 
\begin{align}\label{E:e11}
I_1^0=&\int\p_\a^{j+1}(\varphi-q\k_\a)\L\p_\a^{j+1}\k~d\a \\
=&\int \p_\a^{j+1}\pp\H\p_\a^{j+2}\k~d\a-\int q\p_\a^{j+2}\k\H\p_\a^{j+2}\k~d\a \notag \\
&\hspace*{.8in}-(j+1)\int u\p_\a^{j+1}\k\L\p_\a^{j+1}\k~d\a+\text{(lower order terms)}, \notag
\end{align}
where (lower order terms) collects those which may be bounded by virtue of \eqref{E:uj} and \eqref{E:kt-yj} by $C(\|\k\|_{H^{j+1}}, \|\varphi\|_{H^{j}}, \|u\|_{L^2}, \|\gamma\|_{L^2})$. Similarly the latter equation in \eqref{E:ky} yields after integration by parts that
\begin{align}\label{E:e12}
I^0_2=&\int \p_\a^j(\H\k_{\a\a\a}+2\k\k_{\a\a}-q\varphi_\a+R^\k)\p_\a^j\varphi~d\a \\ 
=&\int \H\p_\a^{j+3}\k\,\p_\a^{j}\pp~d\a+2\int \k\p_\a^{j+2}\k\,\p_\a^j\k_t~d\a
-\int (\k q)_\a(\p_\a^{j+1}\k)^2~d\a \notag \\ &-\big(j-\frac12\big)\int u(\p_\a^{j}\pp)^2~d\a
+\int \p_\a^jR^\k\,\p_\a^j\pp~d\a+\text{(lower order terms)},\notag 
\end{align}
where (lower order terms) is bounded by $C(\|\k\|_{H^{\max(j+1,3)}}, \|\varphi\|_{H^{j}}, \|u\|_{L^2}, \|\gamma\|_{L^2})$ in the routine manner. Note that the second and the third terms on the right side utilizes the former equation in \eqref{E:ky}. 

The first term on the right side of \eqref{E:e11} and the first term on the right side of \eqref{E:e12} cancel each other when added together after integration by parts. The second term on the right side of \eqref{E:e12} and the last term on the right side of \eqref{E:e1j} cancel each other in like manner. The second and the third terms on the right side of \eqref{E:e11} are bounded by $C_0\|u\|_{H^2}\|\p_\a^{j+1}\k\|_{H^{1/2}}^2$ by virtue of inequalities in \eqref{E:Lambda}. The third term on the right side of \eqref{E:e12} is bounded by $(\|q\|_{L^\infty}+\|u\|_{H^1})\|\k\|_{H^2}\|\p_\a^{j+1}\k\|_{L^2}^2$, obviously, and the fourth term on the right side of \eqref{E:e12} is by $(j-\frac12)\|u\|_{L^\infty}\|\p_\a^{j}\pp\|_{L^2}^2$. Moreover \eqref{E:Gkj'} manifests that the fifth term on the right side of \eqref{E:e12} is governed by $C(\|\k\|_{H^{\max(j+1,3)}}, \|\varphi\|_{H^{j}},\|u\|_{L^2},$ $ \|\gamma\|_{L^2})\|\p_\a^j\pp\|_{L^2}$. The third term on the right side of \eqref{E:e1j} is bounded by $\|\k_t\|_{L^\infty}\|\p_\a^{j+1}\k\|_{L^2}^2$. To recapitulate, 
\begin{align}\label{I:e1j}
\frac{d}{dt}e^0_j \leq &
C(\|\k\|_{H^2}, \|\k_t\|_{H^1}, \|u\|_{H^2}, \|\gamma\|_{L^2})(\|\k\|_{H^{j+3/2}}^2+\|\pp\|_{H^{j}}^2) \\ 
&\hspace*{.5in} 
+C(\|\k\|_{H^{\max(j+1,3)}}, \|\pp\|_{H^{j}}, \|u\|_{L^2}, \|\gamma\|_{L^2})(1+\|\pp\|_{H^{j}}) \notag \\ 
\leq & C(\|\k\|_{H^{\max(j+3/2,3)}}, \|\varphi\|_{H^{j}}, \|u\|_{L^2}, \|\gamma\|_{L^2}) \notag
\end{align}
for $j\geq 1$ an integer. The first inequality utilizes \eqref{E:q0} and the second inequality relies upon \eqref{E:uj} and \eqref{E:kt-yj}.

\medskip

It remains to compute temporal growths of $e_0^0$ and $\|u\|_{L^2}^2$, $\|\gamma\|_{L^2}^2$. We invoke equations in \eqref{E:ky} to obtain after integration by parts that
\begin{align}
\frac{d}{dt}\|\k\|_{L^2}^2=&2\int (\varphi-q\k_\a)\k~d\a 
\leq (2\|\varphi\|_{L^2}+ \|u\|_{L^\infty}\|\k\|_{L^2})\|\k\|_{L^2} \label{I:k0}  \\
\leq &C(\|\k\|_{H^1}, \|\varphi\|_{L^2}, \|u\|_{L^2}, \|\gamma\|_{L^2}),\notag \\
\frac{d}{dt}\|\varphi\|_{L^2}^2=&2\int (\H\k_{\a\a\a}-q\varphi_\a+G^\k)\varphi~d\a
\leq C(\|\k\|_{H^3}, \|\varphi\|_{H^1}, \|u\|_{L^2}, \|\gamma\|_{L^2}).\hspace*{-.1in}\label{I:phi0} 
\end{align}
The latter inequalities utilize \eqref{E:uj} and \eqref{E:Gkj'}. Correspondingly the latter equation in \eqref{E:kappa-u} yield by virtue of \eqref{E:R2j}, \eqref{E:pj}, \eqref{E:uj} and integration by parts that 
\begin{equation}\label{I:u0}
\frac{d}{dt}\|u\|_{L^2}^2=2\int (\k_{\a\a}-qu_\a+p\k+r^u)u~d\a 
\leq C(\|\k\|_{H^2}, \|\varphi\|_{L^2}, \|u\|_{L^2}, \|\gamma\|_{L^2}).\end{equation}
Moreover \eqref{E:gamma_tj} and \eqref{E:uj} imply that 
\begin{equation}\label{I:gamma0} 
\frac{d}{dt}\|\gamma\|_{L^2}^2\leq \|\gamma_t\|_{L^2}\|\gamma\|_{L^2} \leq C(\|\k\|_{H^1}, \|\varphi\|_{L^2}, \|u\|_{L^2}, \|\gamma\|_{L^2}).
\end{equation}

\medskip

Adding \eqref{I:e1j} through \eqref{I:gamma0}, 
\begin{equation}\label{I:E1k}
\frac{d}{dt}E^0_k \leq C(\|\k\|_{H^{\max(k+3/2,3)}}, \|\varphi\|_{H^{k}}, \|u\|_{L^2}, \|\gamma\|_{L^2}) \leq C_1\exp(C_2E^0_k) \end{equation}
for $k \geq 1$ an integer for some constants $C_1,C_2>0$. Therefore \eqref{E:E1k} follows since \eqref{I:E1k} supports a solution until the time $T_*=\frac{\exp(-C_2E^0_k(0))}{C_1C_2}$. Moreover  \eqref{E:energyk} follows in light of \eqref{E:k-E} and \eqref{E:E-k}.\end{proof}

We may repeat the preceding argument for \eqref{E:u} (and \eqref{D:Gu1}, \eqref{D:Gu2}) mutatis mutandis to obtain that 
\begin{multline}\label{E:energyu} 
\|u\|_{H^{k+5/2}}(t)+\|(\conv)u_\a\|_{H^k}(t)+\|u_t\|_{L^2}(t) \\ \leq C(t, 
\|\k\|_{H^{k+7/2}}(0), \|\k_t\|_{H^{k+2}}(0), \|u\|_{H^{k+5/2}}(0), \|u_t\|_{H^{k+1}}(0),\|\gamma\|_{L^2}(0))
\end{multline}
at each $t$ in an interval of time, say $[0,T_*]$ abusing the notation, for $k\geq 0$ an integer, where $T_*\in(0,T]$ depends upon Sobolev norms of $\k$, $\k_t$, $u$, $u_t$ and the $L^2$-norm of $\gamma$ at $t=0$. The proof is nearly identical to that of \eqref{E:energyk}. Hence we omit the detail. Instead we refer the reader to the proof of \cite[Proposition 6.1]{CHS}, for instance, for some detail. 

\medskip

Concluding the subsection we address well-posedness of the initial value problem associated with the system \eqref{E:kappa}-\eqref{E:u}, \eqref{D:Gk}, \eqref{D:Gu1}, \eqref{D:Gu2}. 

\begin{theorem}[Local well-posedness]\label{T:LWP}
Assume that a plane curve $z_0(\a)$, $\a\in\R$, such that $z_0(\a)-\a \to 0$ as $|\a| \to \infty$ and $|z_{0\a}|=1$ everywhere on $\R$, satisfies \eqref{E:cord-arc0} for some constant $Q$. Assume that $\theta_0=\arg(z_{0\a})\in H^{k+9/2}(\R)$ and $\gamma_0\in H^{k+5/2}(\R)$ for $k\geq 0$ an integer. 

\medskip

Then the initial value problem consisting of equations \eqref{E:kappa} and \eqref{E:u}, supplemented with equations \eqref{D:Gk}, \eqref{D:Gu1}, \eqref{D:Gu2} $($and \eqref{D:R}, \eqref{D:p}, \eqref{D:q}, \eqref{D:R3}-\eqref{D:Y}$)$, and subject to the initial conditions 
\[\k(\cdot,0)=\k_0,\,\,\k_t(\cdot,0)=\k_1 \quad\text{and}\quad u(\cdot,0)=u_0,\,\, u_t(\cdot,0)=u_1,\]
where $\k_0, u_0$, $\k_1, u_1$ are in \eqref{E:k0u0} and \eqref{E:k1u1}, respectively, supports the unique solution quadruple 
\[ (\k, \k_t) \in C^0([0,T_*]; H^{k+7/2}(\R) \times H^{k+2}(\R)), \,\,
(u, u_t) \in C^0([0,T_*]; H^{k+5/2}(\R) \times H^{k+1}(\R))\] 
satisfying
\begin{multline}\label{E:energy1}
\|\k\|_{H^{k+7/2}}(t)+\|\k_t\|_{H^{k+2}}(t)+\|u\|_{H^{k+5/2}}(t)+\|u_t\|_{H^{k+1}}(t) \\  \leq 
C(t, \|\k_0\|_{H^{k+7/2}},\|\k_1\|_{H^{k+2}},\|u_0\|_{H^{k+5/2}},\|u_1\|_{H^{k+1}},\|\gamma_0\|_{L^2})
\end{multline}
at each $t \in [0,T_*]$, where $T_*>0$ depends upon $Q$ as well as Sobolev norms of $\k_0, \k_1$, $u_0, u_1$ and the $L^2$-norm of $\gamma_0$. 

Moreover \eqref{E:cord-arc} holds throughout the time interval $[0,T_*]$ $($possibly after replacing $Q$ by $\epsilon Q$ for any fixed but small $\epsilon>0)$, where $z(\a,t)$, $\a \in \R$, describes the fluid surface at time $t$ corresponding to the solution of the problem.
\end{theorem}
Solving ``mollified" equations via Picard's iteration, for instance, and taking the limit with the help of \eqref{E:energyk} and \eqref{E:energyu}, the proof is based upon the energy method. Details are discussed in \cite[Section 6]{CHS} or \cite[Section 5]{Am}, for~instance. 

\medskip

Alternative proofs of local in time well-posedness are found in \cite{Yos2}, \cite{Am}, \cite{AM1}, \cite{CS1}, \cite{SZ3}, \cite{CHS}, \cite{ABZ-water} among others for the water wave problem with surface tension and in \cite{Wu1}, \cite{Wu2}, \cite{Lan}, \cite{CL}, \cite{Lin}, for instance, for a broad class of interfacial fluids problems permitting gravity, higher dimensions, vorticity, compressibility, uneven bottom, etc. 

\subsection{Energy estimates under $\p_\a$ and $\G_2$}\label{SS:energy}
This subsection concerns energy estimates for the system \eqref{E:kappa}-\eqref{E:u}, \eqref{D:Gk}, \eqref{D:Gu1}, \eqref{D:Gu2} under $\p_\a$ as well as $\G_2$,~and the proof of Theorem \ref{T:main}.

\medskip

Let the parametric curve $z_0(\a)$, $\a \in \R$, in the complex plane represent the initial fluid surface such that $z_0(\a)-\a \to 0$ as $|\a| \to \infty$, $|z_{0\a}|=1$ everywhere on $\R$ and that \eqref{E:cord-arc0} holds for some constant $Q$. Let $\theta_0=\arg(z_{0\a})$ and let $\gamma_0:\R \to \R$ denote the initial vortex sheet strength, satisfying
\[(\a\p_\a)^{k'}\theta_0 \in H^{k+9/2-k'}(\R)\quad\text{and}\quad
(\a\p_\a)^{k'}\gamma_0 \in H^{k+5/2-k'}(\R)\] 
for $k\geq 0$ an integer and for all $0\leq k'\leq k$ integers. Then 
\begin{align*}
((\a\p_\a)^{k'}\k_0, (\a\p_\a)^{k'}u_0) &\in H^{k+7/2-k'}(\R) \times H^{k+5/2-k'}(\R), \\
((\a\p_\a)^{k'}\k_1, (\a\p_\a)^{k'}u_1) &\in H^{k+2-k'}(\R) \times H^{k+1-k'}(\R)\end{align*} 
for all $0\leq k'\leq k$ integer, where $\k_0$, $u_0$, $\k_1$, $u_1$ are in \eqref{E:k0u0} and \eqref{E:k1u1}, respectively (see the discussion in Section \ref{SS:result}). In particular  
$(\k_0, u_0)\in H^{k+7/2}(\R) \times H^{k+5/2}(\R)$ and 
$(\k_1, u_1)\in H^{k+2}(\R) \times H^{k+1}(\R)$. In light of Theorem \ref{T:LWP}, therefore, the system \eqref{E:kappa}-\eqref{E:u}, \eqref{D:Gk}, \eqref{D:Gu1}, \eqref{D:Gu2}, subject to the initial conditions 
\[\k(\cdot,0)=\k_0,\,\,\k_t(\cdot,0)=\k_1 \quad\text{and}\quad u(\cdot,0)=u_0,\,\, u_t(\cdot,0)=u_1,\]
supports the unique solution quadruple 
\[ (\k, \k_t) \in C^0([0,T_*]; H^{k+3}(\R) \times H^{k+3/2}(\R)), \,\,
(u, u_t) \in C^0([0,T_*]; H^{k+5/2}(\R) \times H^{k+1}(\R)),\] 
where $T_*>0$ is as in the proof of Theorem \ref{T:LWP}. Moreover \eqref{E:cord-arc} holds throughout the time interval $[0,T_*]$, possibly after replacing $Q$ by $\epsilon Q$ for any fixed but small $\epsilon>~0$, where $z(\a, t)$, $\a\in \R$, describes the fluid surface at time $t$ corresponding to the solution of the problem. It offers recourse to various estimates previously worked out during the time interval $[0,T_*]$. 

\medskip

Recall that $C_0, C_1, C_2, ...$ mean positive constants and $C(f_1,f_2, ...)$ is a positive but polynomial expression in its arguments; $C_0, C_1, C_2, ...$ and $C(f_1,f_2, ...)$ which appear in different places in the text need not be the same. Recall that $\L=\H\p_\a$ and $\|f\|_{H^{1/2}}^2=\int(f^2+f\L f)$.

Throughout the subsection we tacitly exercise the first two formulae in \eqref{E:commLj} to interchange either $\p_\a$ or $\p_t$ with $\G_2$ up to a sum of smooth remainders whenever it is convenient to do so.

\medskip

Recall from Section \ref{SS:operators} the notation 
\[ \Gamma_1=\p_\a \quad\text{and}\quad \G_2=\frac12t\p_t+\frac13\a\p_\a.\]
For $j=(j_1,j_2)$ a multi-index, $j_1\geq 0$ and $j_2\geq 0$ integers, let \[\Gamma^j=\Gamma_1^{j_1}\G_2^{j_2}.\] At times we write $\Gamma^j$, $j\geq 0$ an integer, for a $j$-product of $\Gamma$, where  $\Gamma=\Gamma_1$ or $\G_2$. 

\medskip

Taking $\Gamma^j$ of \eqref{E:kappa}, $j\geq 0$ an integer, we make an explicit calculation to write 
\begin{equation}\label{E:kGj}
(\conv)^2\G^j\k-\H \partial_\a^3\G^j\k-2\k\p_\a^2\G^j\k=(\G^jR^\k+F_j^\k)(\k,u;\theta,\gamma)=:R_j^{\k}(\k,u;\theta,\gamma),\hspace*{-.2in}\end{equation}
where $R^\k$ is in \eqref{D:Gk} and 
\begin{equation}\label{D:Fk}
F^\k_j=[(\conv)^2, \G^j]\k-2[\k\p_\a^2, \G^j]\k.\end{equation}
Note that $[\H\p_\a^3, \G_1]=[\H\p_\a^3, \G_2]=0$. Taking $\G^j$ of \eqref{E:u}, similarly,
\begin{equation}\label{E:uGj}
(\conv)^2\G^ju_\a-\H \partial_\a^3\G^ju_\a-2\k\p_\a^2\G^ju_\a=R_j^u(\k,u;\theta,\gamma),
\end{equation}
where $R^u_j=\G^jR^u+F^u_j$, $R^u$ is in \eqref{D:Gu1}, \eqref{D:Gu2} and $F^u_j=[(\conv)^2, \G^j]u_\a-2[\k\p_\a^2, \G^j]u_\a$.

Proposition \ref{P:G} states in the present notation that 
\begin{multline}\label{E:Gk1j} 
\|\G^jR^\k\|_{H^2}(t) \leq C\Big(\sum_{|j'|\leq j}\|\G^{j'}\k\|_{H^3}(t), \\
\sum_{|j'|\leq j}\|\G^{j'}\k_t\|_{H^2}(t), \sum_{|j'|\leq j}\|\G^{j'}u\|_{L^2}(t), M_j(t)\Big),
\end{multline}
\vspace*{-.1in}
\begin{multline}\label{E:Gujj} 
\|\G^jR^u\|_{L^2}(t)\leq C\Big(\sum_{|j'|\leq j}\|\G^{j'}\k\|_{H^3}(t), \sum_{|j'|\leq j}\|\G^{j'}\k_t\|_{H^1}(t) \\
\sum_{|j'|\leq j}\|\G^{j'}u\|_{H^2}(t),  \sum_{|j'|\leq j}\|\G^{j'}u_t\|_{H^1}(t), M_{j}(t)\Big)\end{multline}
at each $t\in [0,T_*]$ for $j\geq 0$ an integer, where $M_{j}$ is in \eqref{D:M}. 

\begin{lemma}[Estimates of $F^\k_j$ and $F^u_j$]\label{L:F}
For $j\geq 0$ an integer
\begin{multline}\label{E:Fk}
\|F^{\k}_j\|_{H^2}(t)\leq C\Big(
\sum_{|j'|\leq j}\|\G^{j'}\k\|_{H^3}(t), \sum_{|j'|\leq j} \|\G^{j'}\k_t\|_{H^2}(t), \\
\sum_{|j'|\leq j}\|\G^{j'}u\|_{H^2}(t), \sum_{|j'|\leq j} \|\G^{j'}u_t\|_{H^1}(t), M_{j}(t)\Big),\end{multline}
\begin{multline} \label{E:Fu}
\|F^u_j\|_{L^2}(t)\leq C\Big( 
\sum_{|j'|\leq j}\|\G^{j'}\k\|_{H^2}(t), \sum_{|j'|\leq j}\|\G^{j'}\k_t\|_{L^2}(t), \\
\sum_{|j'|\leq j} \|\G^{j'}u\|_{H^2}(t), \sum_{|j'|\leq j}\|\G^{j'}u_t\|_{H^1}(t), M_{j}(t)\Big)
\end{multline}
at each $t\in [0,T_*]$, where $M_{j}$ is in \eqref{D:M}.
\end{lemma}

\begin{proof}Let $t\in [0,T_*]$ be fixed and suppressed for the sake of exposition.

We may calculate from \eqref{E:commj} and from the first formula in \eqref{E:commLj} that 
\vspace*{-.1in}
\begin{align*}
F^{\k}_j=&\sum_{m=1}^j\Gamma^{j-m}[(\conv)^2, \Gamma]\Gamma^{m-1}\k \hspace*{-.3in} \\
&-2\sum_{m=1}^j\Gamma^{j-m}[\k, \Gamma]\Gamma^{m-1}\k_{\a\a}
-2\k \sum_{m=1}^j \Big(\begin{matrix} m \\ j \end{matrix}\Big)\big(\frac23\big)^{j-m}\G^{m-1}\k_{\a\a}
\notag =:F_1+F_2, \hspace*{-.3in}\notag
\end{align*}
where the first and the third formulae in \eqref{E:commq} manifest that 
\begin{align*}
[(\conv)^2, \G_1]=&-(\G_1q)\p_\a(\conv)-(\conv)(\G_1q)\p_\a,   \\
[(\conv)^2, \G_2]=&-\Big(\G_2q+\frac12q\Big)\p_\a(\conv)-(\conv)\Big(\G_2q+\frac12q\Big)\p_\a \\ &+3(\conv)^2.
\end{align*}
A straightforward calculation then reveals that $F_1$ is made up of
\[(\G^{j_1}q)\G^{j_2}(\conv)\k, \quad (\G^{j_1}\k)\G^{j_2}(\conv)q\quad\text{and}
\quad(\G^{j_1}q)(\G^{j_2}q)(\G^{j_3}\k),\]
where $j_1$, $j_2$, $j_3$ are integers such that $0\leq j_1, j_2\leq j$ and $1\leq j_3\leq j$. Indeed we exercise \eqref{E:commj} and the first and third formulae in \eqref{E:commq} to interchange $\conv$ and $\G$. Consequently we infer from \eqref{E:q0}-\eqref{E:qjk} and \eqref{E:qtk}, \eqref{E:qtjk} and from the first formula in \eqref{E:commLj}, to interchange $\p_\a$ and $\G_2$ up to a sum of smooth remainders, and product inequalities (see Lemma \ref{L:prod}) that 
\begin{multline*}
\|F_{1}\|_{H^{2}}\leq C\Big(\sum_{|j'|\leq j}\|\G^{j'}\k\|_{H^2},
\sum_{|j'|\leq j}\|\G^{j'}u\|_{H^2}, \sum_{|j'|\leq j}\|\G^{j'}u_t\|_{H^1}, M_j\Big) \\ 
\cdot \Big(\sum_{|j'|\leq j}\|\G^{j'}\k_t\|_{H^{2}}+\sum_{|j'|\leq j}\|\G^{j'}\k\|_{H^3}\Big).
\end{multline*}
Moreover \[ \|F_2\|_{H^{2}} \leq C_0\Big(1+ \sum_{|j'|\leq j}\|\G^{j'}\k\|_{H^2}\Big)
\Big(\sum_{|j'|\leq j}\|\G^{j'}\k\|_{H^3}\Big).\] 
Therefore \eqref{E:Fk} follows. The proof of \eqref{E:Fu} is nearly identical. Hence we omit the detail. 
\end{proof}

\begin{corollary}[Estimates of $R^\k_j$ and $R^u_j$]\label{C:Gj}For $j\geq 0$ an integer
\begin{multline}\label{E:GkGj}
\|R^{\k}_j\|_{H^2}(t)\leq C\Big(
\sum_{|j'|\leq j}\|\G^{j'}\k\|_{H^3}(t), \sum_{|j'|\leq j} \|\G^{j'}\k_t\|_{H^2}(t), \\
\sum_{|j'|\leq j}\|\G^{j'}u\|_{H^2}(t), \sum_{|j'|\leq j} \|\G^{j'}u_t\|_{H^1}(t), M_{j}(t)\Big), \end{multline}
\vspace*{-.1in}
\begin{multline}\label{E:GuGj}
\|R^u_j\|_{L^2}(t)\leq C\Big( 
\sum_{|j'|\leq j}\|\G^{j'}\k\|_{H^3}(t), \sum_{|j'|\leq j}\|\G^{j'}\k_t\|_{H^1}(t),  \\
\sum_{|j'|\leq j} \|\G^{j'}u\|_{H^2}(t), \sum_{|j'|\leq j}\|\G^{j'}u_t\|_{H^1}(t), M_{j}(t)\Big)
\end{multline}
at each $t \in [0,T_*]$, where $M_{j}$ is in \eqref{D:M}.
\end{corollary}

Taking $\G^j$ of equations in \eqref{E:system0}, $j\geq 0$ an integer, incidentally
\begin{equation}\label{systemG0}
\G^j\p_t^2\k-\H\p_\a^3\G^j\k=\G^jG^{\k,q}\quad\text{and}\quad
\G^j\p_t^2u-\H\p_\a^3\G^ju=\G^jG^{u,q},
\end{equation}
where Corollary \ref{C:g} states that 
\begin{multline} \label{E:GkG0}
\|\G^jG^{\k,q}\|_{H^2}(t)\leq C\Big(\sum_{|j'|\leq j}\|\G^{j'}\k\|_{H^3}(t), \\ 
\sum_{|j'|\leq j} \|\G^{j'}\k_t\|_{H^1}(t), \sum_{|j'|\leq j}\|\G^{j'}u\|_{L^2}(t), M_{j}(t)\Big), \end{multline}
\vspace*{-.1in}
\begin{multline}\label{E:GuG0}
\|\G^jG^{u,q}\|_{L^2}(t)\leq C\Big(\sum_{|j'|\leq j}\|\G^{j'}\k\|_{H^3}(t), \sum_{|j'|\leq j}\|\G^{j'}\k_t\|_{L^2}(t), \\
\sum_{|j'|\leq j} \|\G^{j'}u\|_{H^2}(t), \sum_{|j'|\leq j}\|\G^{j'}u_t\|_{H^1}(t), M_{j}(t)\Big),
\end{multline}
at each $t \in [0, T_*]$, where $M_{j}$ is in \eqref{D:M}.

\medskip

We promotly set forth energy estimates for the system \eqref{E:kappa}-\eqref{E:u}, \eqref{D:Gk},~\eqref{D:Gu1}, \eqref{D:Gu2} under $\p_\a$ as well as $\G_2$, and the proof of \eqref{E:mainL2}. 

As in the previous subsection we may write \eqref{E:kGj} and \eqref{E:uGj} both as systems of first-order in time equations as 
\begin{alignat}{2}
&(\conv)\Gamma^j\k=\varphi_j,\qquad&&(\conv)\varphi_j
=\H\partial_\a^3\G^j\k+2\k\p_\a^2\G^j\k+R^\k_j,\label{E:system-k} \\
&(\conv)\Gamma^ju_\a=v_j, \qquad&&(\conv)v_j\,
=\H\partial_\a^3\G^ju_\a+2\k\p_\a^2\G^ju_\a+R^u_j.\label{E:system-u}
\end{alignat}
Let's define the energy norm for \eqref{E:system-k}-\eqref{E:system-u} of order $k$, $k\geq 0$ an integer, as
\begin{align}\label{D:Ek}
E_k(t)=\sum_{|j|=0}^k\big(e^\k_j(t)+e^u_j(t)\big)
+\sum_{|j|=0}^k\Big(&\|\G^j\k\|_{L^2}^2(t)+\|\G^j\k_t\|_{L^2}^2(t)\hspace*{-.1in} \\
&+\|\G^ju\|_{L^2}^2(t)+\|\G^ju_t\|_{L^2}(t)\Big)+M_k(t),\hspace*{-.1in}  \notag
\end{align}
where 
\begin{align}
e_j^\k=&\frac12 \int^\infty_{-\infty} \big(\p_\a^3\G^j\k\L\p_\a^3\G^j\k+(\p_\a^2\pp_j)^2
+2\k(\p_\a^3\G^j\k)^2\big)(\a,t)~d\a,\label{D:ekj} \\
e_j^u=&\frac12\int^\infty_{-\infty} \big(\p_\a\G^ju_\a\L\p_\a\G^ju_\a+v_j^2
+2\k(\p_\a\G^ju_\a)^2\big)(\a,t)~d\a\label{D:euj} 
\end{align} 
and $M_k$ is in \eqref{D:M}. The former equations in \eqref{E:system-k} and \eqref{E:system-u} suggest that $E_k(t)$ amounts to 
${\displaystyle \sum_{|k'|\leq k}\big(\|\G^{k'}\k\|_{H^3}(t)+\|\G^{k'}\k_t\|_{H^2}(t)+\|\G^{k'}u\|_{H^2}(t)+\|\G^{k'}u_t\|_{H^1}(t)\big)}$. In fact
\begin{align}\label{E:Gkt-pp1}
\|\G^j\k_t\|_{H^2}(t) \leq &\sum_{|j'|\leq j}\|\pp_j\|_{H^2}(t) +\sum_{|k'|\leq k}\|q\p_\a\G^{k'}\k\|_{H^2}(t) \\
\leq  &C\Big(\sum_{|j'|\leq j}\|\G^{j'}\k\|_{H^3}(t),\sum_{|j'|\leq j}\|\pp_j\|_{H^2}(t), \|u\|_{L^2}(t),\|\gamma\|_{L^2}(t)\Big) \notag \end{align} and vice versa
${\displaystyle \|\pp_j\|_{H^2}(t)\leq C\big(\sum_{|j'|\leq j}\|\G^{j'}\k\|_{H^3}(t),\sum_{|j'|\leq j}\|\G^{j'}\k_t\|_{H^2}(t), \|u\|_{L^2}(t),\|\gamma\|_{L^2}(t)\big)}$
at each $t\in [0,T_*]$ for $j\geq 0$ an integer. The first inequality in \eqref{E:Gkt-pp1} utilizes the former equation in \eqref{E:system-k} and the second formula in \eqref{E:commLj}, to interchange $\p_t$ and $\G^j$ up to a sum of smooth remainders. The second inequality relies upon \eqref{E:q0}, \eqref{E:uj} and Young's inequality with~$\epsilon$ (see Lemma \ref{L:Young-e}). Similarly 
\begin{align}\label{E:Gut-v1}
\|\G^ju_t\|_{H^1}(t)\leq C\Big(\sum_{|j'|\leq j}\|\G^{j'}u\|_{H^2}(t), &\sum_{|j'|\leq j}\|v_j\|_{L^2}(t), \\
&\sum_{|j'|\leq j}\|\G^{j'}u_t\|_{L^2}(t), \|\k\|_{H^1}(t), \|\gamma\|_{L^2} \Big), \notag \\
\|v_j\|_{L^2}(t)\leq C\Big(\sum_{|j'|\leq j}\|\G^{j'}u\|_{H^2}(t),& 
\sum_{|j'|\leq j}\|\G^{j'}u_t\|_{L^2}(t), \|\k\|_{H^1}(t), \|\gamma\|_{L^2} \Big)\notag\end{align}
at each $t \in [0,T_*]$ for $j\geq 0$ an integer. 

Thanks to the scaling symmetry\footnote{An explicit calculation manifests that the \eqref{E:system-k}, \eqref{E:system-u} remains invariant under \eqref{E:scaling} and 
\[ \pp_j(\a,t)\mapsto \l^{5/2}\pp_j(\l\a,\l^{3/2}t), \qquad \pp_j(\a,t)\mapsto \l^{4}\pp_j(\l\a,\l^{3/2}t),\]
and $z(\a,t)\mapsto \l^{-1}z(\l\a, \l^{3/2}t)$, $\gamma(\a,t)\mapsto \l^{1/2}\gamma(\l\a, \l^{3/2}t)$ for any $\lambda>0$.} (see \eqref{E:scaling}), furthermore, we may assume that $\|\k\|_{L^\infty}(t)$ is sufficiently small throughout the time interval $[0,T_*]$ so that the third integrals on the right sides of \eqref{D:ekj} and \eqref{D:euj} are nonnegative. Accordingly
\begin{multline}\label{E:ku-E}\sum_{|k'|\leq k}\Big(\|\G^{k'}\k\|_{H^{7/2}}(t)+\|\G^{k'}\k_t\|_{H^{2}}(t) \\
+\|\G^{k'}u\|_{H^{5/2}}(t)+\|\G^{k'}u_t\|_{H^1}(t)\Big)\leq C(E_k(t)),\quad\end{multline}
\vspace*{-.1in}
\begin{multline}\label{E:E-ku}
E_k(t)\leq C\Big(\sum_{|k'|\leq k}\|\G^{k'}\k\|_{H^{7/2}}(t), \sum_{|k'|\leq k}\|\G^{k'}\k_t\|_{H^{2}}(t), \\ 
\sum_{|k'|\leq k}\|\G^{k'}u\|_{H^{5/2}}(t), \sum_{|k'|\leq k}\|\G^{k'}u_t\|_{H^1}(t), M_k(t)\Big)
\quad\end{multline}
at each $t\in [0,T_*]$ for $k\geq 1$ an integer.

\medskip

We differentiate \eqref{D:ekj} in time and split the integral, as in \eqref{E:e1j}, as 
\begin{align}\label{E:ekj}\frac{d}{dt}e_j^\k=&
\int \p_\a^3\p_t\G^j\k\L\p_\a^3\G^j\k~d\a+\int \p_\a^2\p_t\varphi_j\,\p_\a^2\varphi_j~d\a \\
&+\int \k_t(\p_\a^3\G^j\k)^2~d\a+2\int \k\p_\a^3\p_t\G^j\k\p_\a^3\G^j\k~d\a=:I^\k_1+I^\k_2+I^\k_3+I^\k_4. \notag \hspace*{-.1in}\end{align}
We then use the former equation in \eqref{E:system-k} to obtain after integration by parts that 
\begin{align}\label{E:ek1}
I_1^\k=&\int\p_\a^{3}(\varphi_j-q\p_\a\G^j\k)\L\p_\a^{3}\G^j\k~d\a \\
=&\int \p_\a^{3}\pp_j\H\p_\a^{4}\G^j\k~d\a-\int q\p_\a^{4}\G^j\k\H\p_\a^{4}\G^j\k~d\a \notag \\
&\hspace*{1.2in}-3\int u\p_\a^{3}\G^j\k\L\p_\a^{3}\G^j\k~d\a+\text{(lower order terms)}, \notag
\end{align}
where (lower order terms) collects those which may be bounded by virtue of \eqref{E:ujk} and \eqref{E:Gkt-pp1} by ${\displaystyle C\big(\sum_{|j'|\leq j}\|\G^{j'}\k\|_{H^3}, \sum_{|j'|\leq j}\|\varphi_{j'}\|_{H^{2}}, \sum_{|j'|\leq j}\|\G^{j'}u\|_{L^2}, M_j\big)}$. Similarly the latter equation in \eqref{E:ky} yields after integration by parts that
\begin{align}\label{E:ek2}
I^\k_2=&\int \p_\a^2(\H\p_\a^3\G^j\k+2\k\p_\a^2\G^j\k-q\p_\a\varphi_j+R^\k_j)\p_\a^2\varphi_j~d\a \\ 
=&\int \H\p_\a^{5}\G^j\k\,\p_\a^{2}\pp_j~d\a+2\int \k\p_\a^{4}\G^j\k\,\p_\a^2\k_t~d\a
-\int (\k q)_\a(\p_\a^{3}\G^j\k)^2~d\a \notag \\ &-\frac32\int u(\p_\a^2\pp_j)^2~d\a
+\int \p_\a^2R^\k_j\,\p_\a^2\pp_j~d\a+\text{(lower order terms)},\notag 
\end{align}
where (lower order terms) is bounded by ${\displaystyle C\big(\sum_{|j'|\leq j}\|\G^{j'}\k\|_{H^3}, \sum_{|j'|\leq j}\|\varphi_{j'}\|_{H^{2}}, \|u\|_{L^2}, \|\gamma\|_{L^2}\big)}$ in the routine manner. Note that the second and the third terms on the right side relies upon the former equation in \eqref{E:ky}. 

The first term on the right side of \eqref{E:ek1} and the first term on the right side of \eqref{E:ek2} cancel each other when added together after integration by parts. The second term on the right side of \eqref{E:ek2} and the last term on the right side of \eqref{E:ekj} cancel each other in like manner. The second and the third terms on the right side of \eqref{E:ek1} are bounded by $C_0\|u\|_{H^2}\|\p_\a^{3}\G^j\k\|_{H^{1/2}}^2$ by virtue of inequalities in \eqref{E:Lambda}. The third term on the right side of \eqref{E:e12} is bounded by $(\|q\|_{L^\infty}+\|u\|_{H^1})\|\k\|_{H^2}\|\p_\a^{3}\G^j\k\|_{L^2}^2$, obviously, and the fourth term on the right side of \eqref{E:e12} is bounded by $\frac32\|u\|_{L^\infty}\|\p_\a^{2}\pp_j\|_{L^2}^2$. Moreover \eqref{E:GkGj} manifests with the assistance of \eqref{E:Gkt-pp1} and \eqref{E:Gut-v1} that the fifth term on the right side of \eqref{E:e12} is governed by ${\displaystyle C\big(\sum_{|j'|\leq j}\|\G^{j'}\k\|_{H^3}, \sum_{|j'|\leq j}\|\varphi_{j'}\|_{H^{2}}, \sum_{|j'|\leq j}\|\G^{j'}u\|_{H^2}, \sum_{|j'|\leq j}\|v_{j'}\|_{L^2}, M_j\big)\|\p_\a^2\pp_j\|_{L^2}^2}$. The third term on the right side of \eqref{E:e1j} is bounded by $\|\k_t\|_{L^\infty}\|\p_\a^{3}\G^j\k\|_{L^2}^2$. To recapitulate, 
\begin{align}\label{I:ekj}
\frac{d}{dt}e^\k_j \leq &
C(\|\k\|_{H^2}, \|\k_t\|_{H^1}, \|u\|_{H^2}, \|\gamma\|_{L^2})(\|\p_\a^3\G^j\k\|_{H^{1/2}}^2+\|\p_\a^2\pp_j\|_{L^2}^2) \\ 
&+C\Big(\sum_{|j'|\leq j}\|\G^{j'}\k\|_{H^3}, \sum_{|j'|\leq j}\|\varphi_{j'}\|_{H^{2}}, \sum_{|j'|\leq j}\|\G^{j'}u\|_{H^2}, \sum_{|j'|\leq j}\|v_{j'}\|_{L^2}, M_j\Big)\hspace*{-.2in} \notag \\ 
\leq & C\Big(\sum_{|j'|\leq j}\|\G^{j'}\k\|_{H^{7/2}}, \sum_{|j'|\leq j}\|\varphi_{j'}\|_{H^{2}}, \sum_{|j'|\leq j}\|\G^{j'}u\|_{H^2}, \sum_{|j'|\leq j}\|v_{j'}\|_{L^2}, M_j\Big)\hspace*{-.2in} \notag
\end{align}
for $j\geq 0$ an integer. The first inequality utilizes \eqref{E:q0} and the second inequality relies upon  \eqref{E:Gkt-pp1}.

\medskip

We then repeat the argument for \eqref{E:system-u} mutatis mutandis to obtain that 
\begin{equation}\label{I:euj}
\frac{d}{dt}e^u_j \leq C\Big(\sum_{|j'|\leq j}\|\G^{j'}\k\|_{H^{7/2}}, \sum_{|j'|\leq j}\|\varphi_{j'}\|_{H^{2}}, \sum_{|j'|\leq j}\|\G^{j'}u\|_{H^2}, \sum_{|j'|\leq j}\|v_{j'}\|_{L^2}, M_j\Big)\hspace*{-.2in}
\end{equation}
for $j\geq 0$ an integer. The proof is nearly identical to that of \eqref{I:ekj}. Hence we omit the detail.

\medskip

It remains to compute temporal growths of $\|\G^j\k\|_{L^2}^2+\|\G^ju\|_{L^2}^2$, $\|\G^j\k_t\|_{L^2}^2$, $\|\G^ju_t\|_{L^2}^2$ as well as $M_k$.  
It is readily seen from the second formula in \eqref{E:commLj} that 
\begin{equation}\label{I:ku0}
\frac{d}{dt}(\|\G^j\k\|_{L^2}+\|\G^ju\|_{L^2})
\leq C_0\sum_{|j'|\leq j}\Big(\|\G^{j'}\k_t\|_{L^2}+\|\G^{j'}u_t\|_{L^2}\Big)
\end{equation}
for $j\geq 0$ an integer. We then use the second formula in\eqref{E:commLj} and \eqref{systemG0}, \eqref{E:GkG0},~\eqref{E:GuG0} to obtain that 
\begin{align}\label{I:k0}
\frac{d}{dt}\|\G^j\k_t\|_{L^2}^2\leq &C_0\|\G^j\k_t\|_{L^2}\Big(\sum_{|j'|\leq j}\|\G^{j'}\p_t^2\k\|_{L^2}\Big) \\
\leq &C\Big(\sum_{|j'|\leq j}\|\G^{j'}\k\|_{H^3}, \sum_{|j'|\leq j}\|\varphi_{j'}\|_{H^{1}}, \sum_{|j'|\leq j}\|\G^{j'}u\|_{L^2}, M_j\Big),\notag \\
\frac{d}{dt}\|\G^ju_t\|_{L^2}^2\leq &C\Big(\sum_{|j'|\leq j}\|\G^{j'}\k\|_{H^3}, \sum_{|j'|\leq j}\|\varphi_{j'}\|_{L^2},\label{I:u0} \\ & \hspace*{1.15in}
\sum_{|j'|\leq j}\|\G^{j'}u\|_{H^2}, \sum_{|j'|\leq j}\|v_{j'}\|_{L^2}, M_j\Big) \notag
\end{align}for $j\geq 0$ an integer. Lastly \eqref{E:Qztj}, \eqref{E:Qztjk}, \eqref{E:gamma_tj}, \eqref{E:gamma_tjk} imply that 
\begin{equation}\label{I:M}
\frac{d}{dt}M_j \leq C_0\sum_{j'=0}^j\Big(\|\G_2^j\theta_t\|_{L^2}+\|\G_2^j\gamma_t\|_{L^2}\Big)
\leq C\Big(\sum_{j'=0}^j\|\G_2^j\k\|_{H^1}, \sum_{j'=0}^j\|\G_2^ju\|_{H^1}, M_j\Big)\hspace*{-.3in}
\end{equation}for $j\geq 0$ an integer.

\medskip

Adding \eqref{I:ekj} through \eqref{I:M}, 
\begin{align}\label{I:E} 
\frac{d}{dt}E_k \leq& C\Big(\sum_{|j'|\leq j}\|\G^{j'}\k\|_{H^{7/2}}, \sum_{|j'|\leq j}\|\varphi_{j'}\|_{H^2}, \sum_{|j'|\leq j}\|\G^{j'}u\|_{H^{5/2}}, \sum_{|j'|\leq j}\|v_{j'}\|_{H^1}, M_j\Big)\hspace*{-.3in} \\
\leq &C_1\exp(C_2 E_k)\notag \end{align}
for $k\geq 0$ an integer for some constants $C_1,C_2>0$. Therefore 
\[ E_k(t)\leq -\frac{1}{C_2} \log(\exp(-C_1E_k(0))-C_1C_2t)\]
since \eqref{I:E} supports a solution until time $T=\frac{\exp(-C_2E_k(0))}{C_1C_2}$. Moreover \eqref{E:mainL2} follows in light of \eqref{E:ku-E} and \eqref{E:E-ku}. 

\medskip

In order to promote \eqref{E:mainL2} to \eqref{E:maint} we claim that 
\begin{equation}\label{E:claim}
t^k\|\lll\a\rrr^{-k}\p_t^{k'}\L^{3/2(k+1-k')}\p_\a^2\k\|_{L^2}(t)\leq C(t, \Phi_{k}) \end{equation} 
at each $t \in [0,T]$ for all $0\leq k'\leq k$ integers, where $\Phi_k$ is in \eqref{D:M0}.

In the case of $k'=k$ note from \eqref{E:mainL2} that 
\[ \sum_{j'=0}^j\big(\|\p_\a^{j'}\G_2^{j-j'}\L^{3/2}\p_\a^2\H^{\ell}\k\|_{L^2}(t)
+\|\p_\a^{j'}\G_2^{j-j'}\p_t\p_\a^2\H^\ell\k\|_{L^2}(t)\big) \leq C(t, \Phi_k)\]
at each $t \in [0,T]$ for all $0\leq j\leq k$ integers, $\ell=0$ or $1$. Indeed $[\L^{3/2}, \G_2]=\a\L^{3/2}$ and the Hilbert transform commutes with $\p_\a$ as well as $\G_2$. Accordingly
\begin{multline}\label{E:k=k}
\| \lll\a\rrr^{-k}(t\p_t)^j\L^{3/2}\p_\a^2\H^\ell\k\|_{L^2}(t)
+\| \lll\a\rrr^{-k}(t\p_t)^j\p_t\p_\a^2\H^\ell\k\|_{L^2}(t) \\
\hspace*{.6in} \leq C_0\sum_{j'=0}^j\big(\|\lll\a\rrr^{-k}
(\a\p_\a)^{j'} \G_2^{j-j'}\L^{3/2}\p_\a^2\H^{\ell}\k\|_{L^2}(t) \\
\hspace*{.2in}+\|\lll\a\rrr^{-k}(\a\p_\a)^{j'} \G_2^{j-j'}\p_t\p_\a^2\H^{\ell}\k\|_{L^2}(t)\big)\leq C(t, \Phi_k)
\hspace*{-.2in} 
\end{multline}
for all $0\leq j\leq k$ integers for $\ell=0$ or $1$. The first inequality follows upon expansion of $\G_2^j$.
The claim then follows since 
\[t^k\|\lll\a\rrr^{-k}\p_t^{k}\L^{3/2}\p_\a^2\k\|_{L^2}(t)
\leq C_0\sum_{j=0}^k\|\lll\a\rrr^{-k}(t\p_t)^j\L^{3/2}\p_\a^2\k\|_{L^2}(t).\]
In the case of $j=k-1$ we use \eqref{E:system0'} to compute 
\begin{align*}
t^k\|\lll\a\rrr^{-k}\p_t^{k-1}\H\p_\a^5\k\|_{L^2}(t) \leq &t^k\|\lll\a\rrr^{-k}\p_t^{k+1}\p_\a^2\k\|_{L^2}(t) \\
&+t^k\|\lll\a\rrr^{-k}\p_t^{k-1}\p_\a^2G^{\k,q}\|_{L^2}(t)=:N_1+N_2,\end{align*}
where ${\displaystyle N_1\leq C_0\sum_{j=0}^k\| \lll\a\rrr^{-k}(t\p_t)^{j}\p_t\p_\a^2\H^\ell\k\|_{L^2}(t) \leq C(t,\Phi_k)}$ by the latter term on the left side of \eqref{E:k=k}. Moreover
\begin{align*}
N_2 \leq &C_0t \sum_{j=0}^{k-1}\|\lll\a\rrr^{-k}(t\p_t)^{j}\p_\a^2G^{\k,q}\|_{L^2}(t)  \\
\leq &C_0t \sum_{j=0}^{k-1}\|\lll\a\rrr^{-k}(\a\p_\a)^{j}\G_2^{k-1-j}\p_\a^2G^{\k,q}\|_{L^2}(t) \leq C(t,\Phi_k)
\end{align*}by virtue of \eqref{E:gkjk} and \eqref{E:mainL2}. 
The claim then follows inductively. In the case of $j=0$ in particular it reduces to \eqref{E:maint}. 

\begin{appendix}

\section{Assorted Sobolev estimates}\label{A:sobolev}
Gathered in this section are various Sobolev estimates, used throughout the exposition. 

Recall the notation that $C_0$ means a positive generic constant and $C(f_1,f_2,...)$ is a positive but polynomial expression in its arguments.

\begin{lemma}[Young's inequality with $\epsilon$]\label{L:Young-e}
For $\epsilon>0$ and for $0<r_1<r_2$
\[ \|f\|_{H^{r_1}}^2\leq \epsilon \|f\|_{H^{r_2}}^2+C(\epsilon)\|f\|_{L^2}.\]
\end{lemma}

The proof utilizes H\"older's inequality in the Fourier space and Young's inequality. We refer the reader to \cite[Lemma 5.1]{GHS} and references therein. 

\begin{lemma}[The composition inequality]\label{L:comp}
If $A^{(j)}$ is in the $C^0$ class, $j\geq 1$ an integer, then 
\[ \|\p_\a^j(A\circ f)\|_{L^r} \leq C_0\max_{1\leq j'\leq j}\Big(\|A^{(j')}\|_{L^\infty}, \|f\|_{L^\infty}^{j'-1}\Big)\|\p_\a^jf\|_{L^r}\]
for $r\in [0,\infty]$ a real. 
\end{lemma}

\begin{lemma}[Product inequalities]\label{L:prod}It follows that
\[\|\p_\a^j\G_2^k(fg)\|_{L^2}\leq 
C_0(\|f\|_{L^{\infty}}\|\p_\a^j\G_2^kg\|_{L^2}+\|\p_\a^j\G_2^kf\|_{L^2}\|g\|_{L^\infty}) \]
for $k\geq 0$ an integer. Moreover 
\[\|\G_2^k(fg)\|_{H^j}\leq C_0(\|f\|_{L^\infty}\|\G_2^kg\|_{H^j}+\|\G_2^kf\|_{H^j}\|g\|_{L^\infty})\]
\end{lemma}
for $j\geq 0$ a half integer.

See \cite[Lemma 5.2]{GHS}, for instance. 

\begin{lemma}[A Sobolev inequality]
It follows that $\|f\|_{L^\infty}\leq C_0\|f\|_{H^1}$.
\end{lemma}
\end{appendix}

\subsection*{Acknowledgments}
The author is supported by the National Science Foundation under grants Nos. DMS-1002854 and DMS-1008885 and by the University of Illinois at Urbana-Champaign under the Campus Research Board grant No. 11162. 

\bibliographystyle{amsalpha}
\bibliography{WWbib}

\def\cftil#1{\ifmmode\setbox7\hbox{$\accent"5E#1$}\else
  \setbox7\hbox{\accent"5E#1}\penalty 10000\relax\fi\raise 1\ht7
  \hbox{\lower1.15ex\hbox to 1\wd7{\hss\accent"7E\hss}}\penalty 10000
  \hskip-1\wd7\penalty 10000\box7}
\providecommand{\bysame}{\leavevmode\hbox to3em{\hrulefill}\thinspace}
\providecommand{\MR}{\relax\ifhmode\unskip\space\fi MR }
\providecommand{\MRhref}[2]{%
  \href{http://www.ams.org/mathscinet-getitem?mr=#1}{#2}
}
\providecommand{\href}[2]{#2}
\begin{thebibliography}{Amb03}

\bibitem[ABZ11]{ABZ-water}
T.~Alazard, N.~Burq, and C.~Zuily, \emph{On the water-wave equations with
  surface tension}, Duke Math. J. \textbf{158} (2011), 413--499.

\bibitem[AM05]{AM1}
David~M. Ambrose and Nader Masmoudi, \emph{The zero surface tension limit of
  two-dimensional water waves}, Comm. Pure Appl. Math. \textbf{58} (2005),
  no.~10, 1287--1315.

\bibitem[Amb03]{Am}
David~M. Ambrose, \emph{Well-posedness of vortex sheets with surface tension},
  SIAM J. Math. Anal. \textbf{35} (2003), no.~1, 211--244.

\bibitem[CHS09]{CHS-lsm}
H.~Christianson, V.~M. Hur, and G.~Staffilani, \emph{Local smoothing effects
  for the water-wave problem with surface tension}, C. R. Math. Acad. Sci.
  Paris \textbf{347} (2009), 159--162.

\bibitem[CHS10]{CHS}
H.~Christianson, V.~M. Hur, and G.~Staffilani, \emph{Strichartz estimates for
  the water-wave problem with surface tension}, Comm. Partial Differential
  Equations \textbf{35} (2010), 2195--2252.

\bibitem[CKS92]{CKS}
Walter Craig, Thomas Kappeler, and Walter Strauss, \emph{Gain of regularity for
  equations of {K}d{V} type}, Annales deI'l. H. P., section C \textbf{9}
  (1992), 147--186.

\bibitem[CL00]{CL}
Demetrios Christodoulou and Hans Lindblad, \emph{On the motion of the free
  surface of a liquid}, Comm. Pure Appl. Math. \textbf{53} (2000), no.~12,
  1536--1602.

\bibitem[CS07]{CS1}
Daniel Coutand and Steve Shkoller, \emph{Well posedness of the free-surface
  incompressible euler equations with or without surface tension}, J. Amer.
  Math. Soc. \textbf{20} (2007), no.~3, 823--930.

\bibitem[GHS07]{GHS}
Y.~Guo, C.~Hallstrom, and D.~Spirn, \emph{Dynamics near unstable, interfacial
  fluids}, Commun. Math. Phys. \textbf{270} (2007), 635--689.

\bibitem[Kat83]{Kato}
Tosio Kato, \emph{On the {C}auchy problem for the (generalized) {K}orteweg-de
  {V}ries equation}, Studies in Applied Mathematics: a volume dedicated to
  Irving Segal \textbf{8} (1983), 93--128.

\bibitem[Lan05]{Lan}
David Lannes, \emph{Well-posedness of the water-wave equations}, J. Amer. Math.
  Soc. \textbf{18} (2005), no.~3, 605--654.

\bibitem[Lin05]{Lin}
Hans Lindblad, \emph{Well-posedness for the motion of an incompressible liquid
  with free surface boundary}, Ann. of Math. (2) \textbf{162} (2005), no.~1,
  109--194.

\bibitem[SZ11]{SZ3}
Jalal Shatah and Chongchun Zeng, \emph{Local well-posedness for fluid interface
  problems}, Arch. Rational Mech. Anal. \textbf{199} (2011), 653--705.

\bibitem[Wu97]{Wu1}
Sijue Wu, \emph{Well-posedness in sobolev spaces of the full water wave problem
  in 2-d}, Invent. Math. \textbf{130} (1997), no.~1, 39--72.

\bibitem[Wu99]{Wu2}
\bysame, \emph{Well-posedness in sobolev spaces of the full water wave problem
  in 3-d}, J. Amer. Math. Soc. \textbf{12} (1999), no.~2, 445--495.

\bibitem[Wu09]{Wu3}
\bysame, \emph{Almost global wellposedness of the 2-d full water wave problem},
  Inv. Math. \textbf{177} (2009), no.~1, 45--135.

\bibitem[Yos82]{Yos1}
Hideaki Yosihara, \emph{Gravity waves on the free surface of an incompressible
  perfect fluid of finite depth}, Publ. Res. Inst. Math. Sci. \textbf{18}
  (1982), no.~1, 49--96.

\bibitem[Yos83]{Yos2}
\bysame, \emph{Capillary-gravity waves for an incompressible ideal fluid}, J.
  Math. Kyoto Univ. \textbf{23} (1983), no.~4, 649--694.

\end{thebibliography}

\end{document}